\documentclass[12pt]{amsart}
\usepackage{latexsym}
\usepackage{amsmath,amsthm,amsfonts,amscd,eucal,amssymb}
\usepackage{epsfig}
\usepackage{enumerate}
\newtheorem{thm}{Theorem}[section]
\newtheorem{cor}[thm]{Corollary}
\newtheorem{lem}[thm]{Lemma}

\newtheorem{prop}[thm]{Proposition}

\theoremstyle{definition}
\newtheorem{dfn}[thm]{Definition}
\theoremstyle{remark}
\newtheorem{rem}[thm]{Remark}

\numberwithin{equation}{section}
\def\ca{{\mathcal A}}
\def\cb{{\mathcal B}}
\def\cc{{\mathcal C}}

\def\ce{{\mathcal E}}
\def\cf{{\mathcal F}}
\def\cg{{\mathcal G}}
\def\ch{{\mathcal H}}

\def\ck{{\mathcal K}}
\def\cl{{\mathcal L}}

\def\ct{{\mathcal T}}
\def\cu{{\mathcal U}}

\def\T{\mathbb{T}}

\def\bc{{\mathbb C}}

\newcommand{\bn}{\mathbb N}
\def\br{{\mathbb R}}
\def\bt{{\mathbb T}}
\def\bz{{\mathbb Z}}

\renewcommand{\a}{\alpha}
\def\b{\beta}
\def\g{\gamma}        \def\G{\Gamma}
\def\d{\delta}        \def\D{\Delta}
\def\eps{\varepsilon}
    
\def\z{\zeta}
\def\th{\vartheta}

\def\l{\lambda}       
\def\m{\mu}
\def\n{\nu}

\def\r{\rho}
\def\s{\sigma}       
    
\def\t{\tau}

\def\f{\varphi}

\def\o{\omega}        \def\O{\Omega}

\newcommand{\norm}[1]{\left\Vert#1\right\Vert}

\newcommand{\set}[1]{\left\{#1\right\}}
\DeclareMathOperator{\diam}{diam}
\DeclareMathOperator{\Li}{Li}
\DeclareMathOperator{\re}{Re}
\DeclareMathOperator{\im}{Im}
\newcommand{\ds}{\displaystyle}

\newcommand{\de}{\partial}

\DeclareMathOperator{\tr}{tr}
\DeclareMathOperator{\Res}{Res}
\DeclareMathOperator{\vol}{vol}
\DeclareMathOperator{\dom}{dom}
\DeclareMathOperator{\Proj}{Proj}

\def\ov{\overline}
\DeclareMathOperator{\osc}{Osc}

\DeclareMathOperator{\sgn}{sgn}

\def\const{{ const}}
\DeclareMathOperator{\ind}{Ind}
\DeclareMathOperator{\Aff}{Aff}
\DeclareMathOperator{\Ci}{Ci}%Clausen cosine function
\def\pa{p_\a}

\def\dH{d_H} % Hausdorff dimension of K
\def\dE{d_E} % energy dimension of K
\def\dD{d_D}  % metric dimension of D
\def\deD{\d_D}  % energy dimension of D

\begin{document}

\title{Spectral triples for the Sierpinski Gasket}
% %
\author{Fabio Cipriani}
\address{(F.C.) Politecnico di Milano, Dipartimento di Matematica,
piazza Leonardo da Vinci 32, 20133 Milano, Italy.} \email{fabio.cipriani@polimi.it}
\author{Daniele Guido}
\address{(D.G.) Dipartimento di Matematica, Universit\`a di Roma ``Tor
Vergata'', I--00133 Roma, Italy.} \email{guido@mat.uniroma2.it}
\author{Tommaso Isola}
\address{(T.I.) Dipartimento di Matematica, Universit\`a di Roma ``Tor
Vergata'', I--00133 Roma, Italy} \email{isola@mat.uniroma2.it}
\author{Jean-Luc sauvageot}
\address{(J.-L.S.) Institut de Math\'ematiques, CNRS-Universit\'e Denis Diderot, ÊF-75205 Paris Cedex 13, France.} \email{jlsauva@math.jussieu.fr}

\thanks{This work has been partially supported by GNAMPA, MIUR,
the European Networks ``Quantum Spaces - Noncommutative Geometry"
HPRN-CT-2002-00280, and ``Quantum Probability and Applications to
Physics, Information and Biology'', GDRE GREFI GENCO, and  the ERC Advanced Grant 227458 OACFT ``Operator Algebras and Conformal Field Theory"}
\date{}

% ----------------------------------------------------------------
\begin{abstract}
We construct a family of spectral triples for the Sierpi\'nski Gasket $K$.
For suitable values of the parameters, we determine the dimensional spectrum and recover the Hausdorff measure of $K$ in terms of the residue of the volume functional $a\to\tr(a\,|D|^{-s})$ at its abscissa of convergence $d_D$, which coincides with the Hausdorff dimension $d_H$  of the fractal.
We determine the associated Connes' distance showing that it is bi-Lipschitz equivalent to the distance on $K$ induced by the Euclidean metric of the plane, and show that the pairing of the associated Fredholm module with (odd) $K$-theory is non-trivial.
When the parameters belong to a suitable range, the abscissa of convergence $\d_D$ of the energy functional $a\to\tr(|D|^{-s/2}|[D,a]|^2\,|D|^{-s/2})$ takes the value $d_E=\frac{\log(12/5)}{\log 2}$, which we call {\it energy dimension}, and the corresponding residue gives the standard Dirichlet form on $K$.
%We also recover the standard Dirichlet form on $K$, as the residue of the energy functional $a\to\tr(|D|^{-s/2}|[D,a]|^2\,|D|^{-s/2})$ at its abscissa of convergence $\d_D$, which we call the {\it energy dimension}. The fact that the volume dimension differs from the energy dimension, $d_D\neq \d_D$, reflects the fact that on $K$ volume and energy are distributed singularly.
\end{abstract}
\maketitle

\section{Introduction}

The advent of Noncommutative Geometry allowed to consider from a geometrical and analytical point of view spaces which appear to be singular when analysed by using the classical tools of Differential Calculus and Riemannian Geometry.

\medskip\noindent
\begin{tabular}{l l}
    \epsfbox{{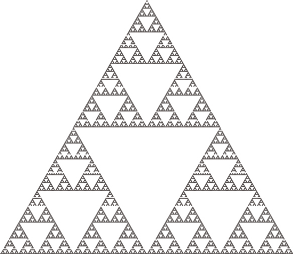}} &
    \vbox{
    \hbox{\hsize=4.4in \vbox{\lineskip=4pt\noindent
In the present paper  we approach from a NCG point of view the study of a compact subset $K$ of the plane which is a central example among fractal sets, namely the Sierpi\'nski Gasket. We associate to the gasket a family of spectral triples. For values of the parameters in a suitable range, the triple reconstructs the main known features of the gasket, namely its similarity dimension, Hausdorff measure, a distance which is bi-Lipschitz equivalent w.r.t. the Euclidean one, and the standard Dirichlet form, with the appearance of an {\it energy dimension}. Moreover, these triples pair non trivially with the K-theory of the gasket.
    }}
    }
\end{tabular}
\noindent

The fundamental topological property of $K$ is its self-similarity, by which $K$ can be reconstructed as a whole from the knowledge of any arbitrary small part of it. More precisely, considering the three similitudes $w_1 ,w_2 ,w_3$ of scaling parameter 1/2 fixing respectively the vertices $p_1 ,p_2 ,p_3$ of an equilateral triangle, one may characterize $K$ as the only compact set in $\br^2$ such that
\[
K=w_1 (K)\cup w_2 (K)\cup w_3 (K)\, ,
\]
namely $K$ is the fixed point of the map $K\mapsto w_1 (K)\cup w_2 (K)\cup w_3 (K)$ which is a contraction with respect to the Hausdorff distance on compact subsets of the plane. This allows various approximations of $K$ as, for example, the one given by finite graphs.

\medskip

\centerline{\includegraphics[width=5.0in,height=0.5in]{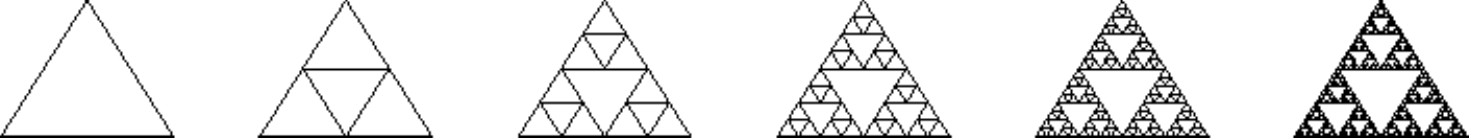}}

\medskip

The gasket $K$ was introduced by Sierpi\'nski for purely topological motivations \cite{Sierpinski}. Successively, using measure theory, it was noticed that it is a space with a non integer Hausdorff dimension $d_H=\frac{\ln 3}{\ln 2}$ \cite{Hutch}, which later attracted the attention of Probabilists, who constructed a stochastic process $X_t $ with continuous sample paths on $K$ \cite{Ku1}. The process is symmetric w.r.t. the Hausdorff measure $\mu_H$ and has a self-adjoint generator $\Delta$ in $L^2 (K,\mu_H)$, whose discrete spectrum was carefully studied by Fukushima and Shima in \cite{FS}.
Finally, Kigami \cite{Kiga} introduced on the gasket (and other fractals) the notion of harmonic structure, the most symmetric choice of which produces on $K$ the so called standard Dirichlet form, whose associated self-adjoint operator in $L^2 (K,\mu_H)$ coincides with $\Delta$.

\bigskip

We recall that the main object for the spectral description of the metric aspects of a geometry, introduced by Connes \cite{Co}, is the so called spectral triple $(\ca,\ch,D)$. It consists of an algebra $\ca$ acting on a Hilbert space $\ch$ and a self-adjoint unbounded operator $D$, the so called Dirac operator. Main requests are the boundedness of the commutators  $[D,a]$ of the elements $a\in\ca$ with $D$, and the fact that $D$ has discrete spectrum. Such triples are meant to generalize the role that, on a compact Riemannian manifold, is played by the algebra of smooth functions and by the Dirac operator acting on the Hilbert space of square integrable sections of the Clifford bundle.
 The integral $\int a\,dvol$ w.r.t. the Riemannian volume form is replaced by the functional
\begin{equation}\label{noncommvol}
\ca\ni a\mapsto\tr_\o(a|D|^{-d_D})\, ,
\end{equation}
$\tr_\o$ being the Dixmier logarithmic trace on the algebra of compact operators on a separable Hilbert space. There is a unique exponent $d_D$ (if any), called metric dimension (cf. \cite{GuIs9,CoMa}), depending on the asymptotic distribution of the eigenvalues of $D$, which gives rise to a non-trivial functional on $\ca$ via the formula above. In great generality, this functional is a positive trace on the algebra $\ca$ (see \cite{CGS}). In this way, on a manifold, one recovers the dimension and the integral. The differential 1-form $da$ of a smooth function $a$ on the manifold, is replaced by the commutator $[D,a]$, so that the Lipschitz seminorm $\|da\|_\infty$ is then replaced by the norm $\|[D,a]\|$ of the commutator. The Riemannian metric is generalized by the Connes distance between states on $\ca$, defined through the formula
\[
\r_D(\f,\psi)=\sup\{|\f(a)-\psi(a)|:\|[D,a]\|\le 1\}\,,
\]
that can be thought as generalized noncommutative Monge-Kantorovitch distance.
\par\noindent
A simple but fruitful idea is to generalize the Dirichlet energy integral of a manifold
\[
\ce[a]=\int|da|^2
\]
to the NCG setting of a spectral triple, by the ansatz:
\begin{equation}\label{marion}
\ce[a]=\tr_\o(|[D,a]|^2|D|^{-\d_D}).
\end{equation}
Similar formulas have been recently considered in \cite{PeBe,JuSa}.
Finally, we recall that spectral triples produce topological invariants. Indeed, a spectral triple gives rise to a class in the K-homology of the algebra $\ca$, hence it pairs with K-theory. Such pairing may be expressed in terms of the Fredholm module $(\ca,\ch,F)$ associated with the spectral triple, $F$ being the phase of $D$. In particular, if the triple is odd, and $u$ is an invertible element of $\ca$, the pairing is given by the index map $u\to \ind(P_+uP_+)$, where $P_+$ is the projection on the positive part of the spectrum of $F$, and $P_+uP_+$ is a Fredholm operator acting on the space $P_+\ch$.

\medskip

As mentioned above, one of the features of Noncommutative Geometry is that it non only applies to noncommutative manifolds, such as noncommutative tori, quantum groups \cite{CFK} or the various quantum spheres, but it is also able to describe some classical (e.g. dealing with points of a topological space) but singular geometries. The first example of this fact was given by Connes in \cite{Co}, where a spectral triple was assigned to the Cantor set. Such a triple could reconstruct  dimension, measure and distance of the fractal set, as well as the pairing with K-theory. The peculiar aspects of the construction are the following: the triple is a direct sum of triples associated with elementary building blocks; the building blocks are  the lacunas of the fractal, namely the boundaries of the removed intervals in the Cantor set.
The first idea can easily be adapted to self-similar fractals, by choosing as building blocks  the images of a suitable subset via compositions of similarities.
This has been exploited in \cite{CIS} for the Sierpinski gasket (as anticipated in \cite{CIL}), where the building blocks are the lacunas of the gasket, meant as the boundaries of the removed triangles in $K$. Again, dimension, measure, distance and pairing with K-theory are reconstructed in spectral terms.

Let us mention here that, as far as the volume measure is concerned, a Connes'-like formula for P.C.F.~self-similar fractals was obtained already in the paper of Kigami and Lapidus, \cite{KiLa}.

Our aim here is to add to this list the standard Dirichlet form on the gasket by making use of formula \eqref{marion}. While for regular geometries the energy form may be considered as a derived object, given the other analytic and geometric tools, on fractals the point of view is reversed.
The fundamental observation in this regard is that, for a large class of self-similar fractals, energy measures, representing the distribution of energy, are singular w.r.t. the self-similar measures, representing the distribution of the volume (cf.  \cite{BST, Ku, Ku1, Hino, HiNa}). In our noncommutative picture, this singularity is reflected in the difference between the abscissas of convergence $d_D$ and $\d_D$ in the formulas \eqref{noncommvol}, \eqref{marion}, quantities which would coincide with the dimension on Riemannian manifolds and on other "regular" spaces.

For parameters in a suitable range, $\d_D$ takes the value $d_E=\frac{\log(12/5)}{\log2}\approx1.26$, which plays the role of an {\it energy dimension}. Such dimension is thus smaller than the Hausdorff dimension $d_H:=\frac{\log3}{\log2}\approx 1.58$, and  has no apparent counterpart in the classical analysis on fractals.

Let us try to explain why we call dimension such a number $d_E$. Hausdorff dimension is, as it is said, a rarefaction index, which separates two opposite, extreme behaviors. For values above, the associated external measures vanish; for values below they are always infinite; the Hausdorff dimension is the only value (if any) which can produce a non trivial measure.
In noncommutative geometry, such a role is played by the abscissa of convergence of appropriate functionals, cf. Theorem 2.7 in \cite{GuIs9}. In the case of the volume functional $s\to \tr(a|D|^{-s})$, the abscissa $d_D$ separates the values for which the logarithmic singular trace vanishes identically from those for which such trace is infinite, the value $d_D$ being  the only value (if any) which can produce a non trivial trace functional. The same happens for our energy dimension: the abscissa of convergence  of the energy functional $s\to \tr(|[D,a]|^2|D|^{-s})$ is the only value (if any) which can reproduce the standard Dirichlet form. This happens exactly when $\d_D=\dE$.

It is interesting to notice that, in analogy to what has been discovered for the volume functional on Riemannian manifolds in \cite{LPS}, formula \eqref{marion} above, even in its symmetrized version
\begin{equation}\label{marionsymmetric}
\ce[a]=\tr_\o(|D|^{-\d_D/2}|[D,a]|^2|D|^{-\d_D/2})\, ,
\end{equation}
provides the value of the standard Dirichlet form, up to a multiplicative constant, on a suitable form core only, and not on the whole Dirichlet space of finite energy elements. On the other hand, we will show that the residue
\begin{equation}\label{marionresidue}
\ce[a]=\Res_{s=1}\tr(|D|^{-s\d_D/2}|[D,a]|^2|D|^{-s\d_D/2})
\end{equation}
does coincide, up to a multiplicative constant, with the standard Dirichlet form for all finite energy elements.

Let us mention here that the spectral triple for the gasket proposed in \cite{GuIs10} Remark 2.14, whose building blocks were spectral triples associated with the boundary points of the edges of the gasket, could indeed produce the standard Dirichlet form exactly as above, with the same energy dimension. However, being based on a discrete approximation of $K$, it could not give rise to any pairing with K-theory.

We remark that noncommutative geometry provides also a replacement for the de~Rham cohomology in terms of cyclic cohomology, however we do not pursue this direction here. Indeed differential forms have no classical analogue on fractals, and their study in this case is essentially based on the differential calculus developed in \cite{CiSa}. There, the authors associate a bimodule-valued derivation to a regular Dirichlet form and define differential 1-forms as the elements of the bimodule. In \cite{CiSa2} this first order differential calculus on p.c.f. fractals has been developed further to a pseudo differential calculus by the construction of Fredholm modules associated to the Dirichlet form. Recent developments in this direction are also contained in \cite{IRT}, while in a recent paper of ours \cite{CGIS1} concerning the gasket, we give a more concrete description of the 1-forms of \cite{CiSa} in such a way as to define their integrals on paths and their generalized potentials on suitable coverings and develop a Hodge-de Rham theory on the gasket.

\bigskip

We now come to a more detailed description of our family of spectral triples. As in \cite{CIS}, our building blocks are associated with the lacunas (boundaries of removed triangles) of the gasket, canonically identified with circles.
The starting point is the observation that formula \eqref{marionresidue} may recover the standard Dirichlet form on $K$ if and only if the standard triple on the circle is deformed (cf. Lemma \ref{traceClass} and Theorem \ref{thm:res=en}).
From the technical point of view, this is based on a single but fundamental result by A. Jonsson \cite{Jo}, which states that the restriction to lacunas of finite energy functions on $K$ with respect to the standard energy form  belongs to a suitable fractional Sobolev space. Accordingly, we deform the classical spectral triple for the circle $\T$, by replacing the standard Laplacian $\D$ with its roots $\D^\a$ for suitable $\a$'s in $(0,1]$. Notice that, by the theory of Dirichlet forms (see \cite{FOT}), the quadratic form $\ce_\a$ on $L^2 (\T)$ determined by the self-adjoint operator $\D^{\a}$ is a Dirichlet form. As proposed in \cite{CiSa}, we construct a bimodule-valued derivation $\partial_\a$, characterized by
\[
\ce_\a [a]=\|\partial_\a a\|^2\, ,\qquad \D^\a =\partial^*_\a\circ\partial_\a\, ,
\]
and define a Dirac operator on $\T$ as
\[
D:=\begin{pmatrix}0&\partial_\alpha\\\partial_\alpha^*&0\end{pmatrix}\, .
\]
While this deformation does not quantize the algebra, which remains $C(\T)$, a zest of noncommutativity is nonetheless present, since the left and right actions of functions on the Hilbert space do not coincide (functions do not commute with forms). This is related to the fact that, while for $\a=1$ the distributional kernel giving rise to the energy on $\T$ is supported on the diagonal, this is no-longer true for $\a<1$. In probabilistic terms, the stochastic process on $\T$ generated by $\Delta^\alpha$ is a diffusion (i.e. has continuous paths), when $\alpha =1$, while it is purely jumping, when $\alpha <1$.

A second deformation parameter $\b$ may be introduced, by replacing the scaling parameter $1/2$  for the gasket with $2^{-\beta}$. Even though also this deformation produces interesting spectral triples, as explained below, such deformation is not really needed, since the value $\beta=1$ is not only acceptable, but is the only one which corresponds to the metric gasket embedded in $\br^2$. Therefore, apart form the comments here, the results concerning $\beta\ne1$ are confined in Section \ref{betaalpha}.

An unexpected outcome of the construction is that the two parameters have a quite different role for the gasket as a whole. Indeed, $\a$ only plays the role of a threshold parameter. The condition $\a\leq\a_0=\frac{\log(10/3)}{\log 4}$ is a necessary condition for formula \eqref{marionsymmetric} to be finite for finite energy functions, and to reproduce the standard Dirichlet form. If, furthermore, $\a>\left(2-\frac{\log(5/3)}{\b \log 2} \right)^{-1}$, one gets a full-fledged spectral triple, whose features only depend on $\b$, which assumes the role of a deformation parameter.
\par\noindent
In fact, for $\a_0<\b\leq1$, the Connes metric $\r_{D,\b}$ is bi-Lipschitz w.r.t. the Euclidean distance raised to the power $\b$ or, equivalently, bi-Lipschitz w.r.t the geodesic metric, induced on $K$ by the Euclidean metric, raised to the power $\b$. Consequently,  the metric dimension is given by $d_{D,\b}=d_{D,1}\cdot \b^{-1}$, and, as expected, the volume measure $\m_{D,\b}$ coincides (up to a multiplicative constant) with the Hausdorff measure for the dimension $d_{D,\b}$, which in turn coincides with the Hausdorff measure for the dimension $d_H$. The energy dimension is given by $\d_{D,\b}=2-\frac{\log(5/3)}{\log 2}\b^{-1}$, and the corresponding energy form  do not even depend on $\b$, apart from a multiplicative constant.

Neither $\a$ nor $\b$ affect the pairing with K-theory. However we had to tackle another difficulty concerning the Fredholm module associated with the spectral triple. In fact, in order to implement the deformation associated with the parameter $\a$, we had to choose the Hilbert space as the module of differential forms, making the triple (and the Fredholm module) an even one. To recover the pairing with odd K-theory, we have to add a further grading, obtaining a 1-graded Fredholm module, which then has the correct pairing with odd K-theory.

We now discuss the relation with previous constructions. If we denote by $D_{\alpha,\beta}$ the Dirac operator on the gasket associated with the parameters $\alpha$ and $\beta$, one of the Dirac operators considered in \cite{CIS} is essentially equivalent to $D_{1,1}$, and a simple relation holds:  $|D_{1,1}|^\alpha=|D_{\alpha,\alpha}|$. But the only deformation that is needed is that of the Dirac on the circle, namely of the first parameter, and the choice $\alpha=\beta$ is not compatible with the reconstruction of the energy, cf. Section 5.
Moreover,  we needed to construct a suitable derivation $\partial_\alpha$ on the circle in such a way that $\partial_\alpha^*\partial_\alpha=\Delta^\alpha$, $\Delta$ being the Laplacian on the circle, and then to define the deformed Dirac operator on the circle as the anti-diagonal matrix in formula \eqref{Dalpha}. This guarantees  not only the correct spectral behavior, but also  suitable commutation relations of the global Dirac $D_{\alpha,1}$ on the gasket with the elements of the algebra $\mathcal A$.

We conclude this review about the dependence of our constructions on $\a$ and $\b$ by noticing that the requests concerning spectral triple properties reflect into independent bounds on the parameters, which finally give rise to a quite small fraction of the $(\b,\a)$-plane. The fact that this set is indeed non empty is not at all obvious, and only an analysis of a larger family of fractals and their Dirichlet forms may reveal the reasons of its non-triviality.

We note here that our triples indeed violate one of the requests of a spectral triple as defined in \cite{Co}, since the kernel of the Dirac operator is infinite dimensional. However, this degeneracy of the kernel does not cause any harm in the construction, when taking the point of view of reading $|D|^{-s}$ as the functional calculus of $D$ with the function $f(t)=0$ for $t=0$ and $f(t)=|t|^{-s}$ for $t\ne0$.

The question is more subtle when the associated Fredholm module is concerned. Indeed, denoting by $P_{\pm}$ the projection on the positive, resp. negative, part of the spectrum of $D$, the two formulas for the pairing of the module with the K-theory class of an invertible element $u$ given by $\ind(P_+\pi(u)P_+)$ and $-\ind(P_-\pi(u)P_-)$, which are  equivalent when the dimension of the kernel of $D$ is finite, may be expected to differ. We call  a Fredholm module  {\it tamely degenerate} when such equality holds, hence  the kernel of $D$ is irrelevant from the K-theoretical point of view, and check that this condition is satisfied for our triples.

We describe now some technical aspects of our construction. First, in order to construct a Dirac operator for the $\a$-deformed triples on the circle, we had to define a differential square root of $\D^\a$, or, in other terms, a derivation implementing the corresponding Dirichlet form. This has been done by realizing the corresponding Dirichlet form in terms of an  integral operator, whose distributional kernel is written in terms of a special function, the so called Clausen cosine function $\Ci_s$. We show that $-\Ci_{-2\a}\geq0$, for $0<\a<1$, and describe the Dirac operator in terms of the derivation $\partial_\a$ given by $\partial_\a f (x,y)=(-2\pi\Ci_{-2\a}(x-y))^{1/2}(f(x)-f(y))$. By means of some explicit estimates on $\Ci_\a$ we can show the relation of the Connes' distance for the $\a$-deformed circle and the $\a$-power of the Riemannian distance. In this sense, our deformed circles may be considered as quasi-circles, since the $\a$-power of the Riemannian distance clearly satisfies the so-called reverse triangle inequality  \cite{Ahlfors}. As for the case of the gasket described above, the $\a$-deformation rescales the Hausdorff dimension of the circle and leaves the volume invariant (up to a multiplicative constant).

Second, our study of the noncommutative formula for the standard Dirichlet form produces an interesting situation when Dixmier traces are concerned. Indeed, when used to describe the volume form in noncommutative geometry, the Dixmier trace is computed for elements which belong to the principal ideal of compact operators generated by $|D|^{-d}$, and the same happens for the computation of the energy form according to formula \eqref{marion} when regular spaces are considered, namely when the metric dimension $d_D$ and the energy dimension $\d_D$ coincide. It is known that the theory of singular traces on principal ideals (\cite{ Varga, AlGuPoSc, GuIs} etc.) is in a sense simpler than the corresponding theory on symmetrically normed ideals. In the case of fractal spaces however, there is no principal ideal containing all elements of the form $|[D,f]|^2|D|^{-\d}$, and Dixmier traces on symmetrically normed ideals and the analysis in \cite{CRSS,CPS} play a key role.

As for the organization of the paper,  it consists of this introduction, four sections, and an appendix. The first section is devoted to some results on (possibly degenerate) spectral triples and Fredholm modules,
the second to the description of the $\a$-deformed circles,  the third to the construction and results of the triples on the gasket, and the fourth to the two-parameter triples. The appendix contains some estimates concerning the Clausen functions.

The results contained in this paper have been described in several conferences, such as Cardiff 2010, Cambridge 2010, Cornell 2011, Paris 2011, Messina 2011, Cortona 2012, Marseille 2012.

\section{Spectral Triples}

\subsection{Spectral Triples and their Fredholm Modules}

We generalize here the notion of Spectral Triple and of Fredholm module, by allowing the Dirac operator $D$ to have an infinite dimensional kernel. This generalization, with respect to the situation usually considered in the literature, will be useful later on, when we will construct Spectral Triples on the circle and the Sierpinski Gasket, whose Dirac operators have an {\it infinite dimensional kernel}. Because of that, some extra work will be needed to construct an associated Fredholm module having nontrivial pairing with $K$-theory.
\vskip0.2truecm\noindent
We recall that the notion of Fredholm module $(F, \pi ,\ch)$ on a compact topological spaces $K$ is a generalization of the theory of elliptic pseudodifferential operators on a compact manifold. In its {\it odd} form, one requires that the elements $f$ of the algebra of continuous functions $C(K)$ are represented as bounded operators $\pi (f)$ on a Hilbert space $\ch$ on which, moreover, a distinguished self-adjoint operator $F$ of square 1  is considered, the {\it symmetry}, in such a way that the commutators $[F,\pi(f)]$ are compact operators.
\par\noindent
In its simplest example, $F$ is the Hilbert transform, a $0$-order, pseudodifferential operator acting on the space of square integrable functions on the circle. When continuous functions are considered as multiplication operators acting on the same Hilbert space, the compactness of the commutators above results from the fundamental theorem of the theory of pseudo-differential operators, according to which commutators of $0$-order operators have order $-1$ and are thus compact.
\par\noindent
As a consequence of the theory, operators like $FuF$, constructed by (matrix ampliations) of $F$ and continuous functions $u$  with values in unitary matrices, are Fredholm operators and their indexes give a pairing between the Fredholm module $(F, \pi , \ch)$ and the odd K-theory group generated by continuous functions on $K$ with values in unitary matrices. In the even case, indexes constructed similarly to those above give a pairing between the Fredholm module $(F, \pi , \ch)$ and the even K-theory group, generated by continuous functions on $K$ with values in the space of matrix projections. By a famous theorem due to Serre and Swan, elements of the even K-theory group have a direct geometric interpretation in terms of equivalence classes of locally trivial vector bundles on $K$.
\par\noindent The notion of Fredholm module, introduced by Atiyah \cite{At} on compact  spaces and generalized  to C$^*$-algebras by Mishchenko \cite{Mis}, Brown-Douglas-Fillmore \cite{BDF}, and Kasparov \cite{Kas}, lies at the core of the Connes' noncommutative differential geometry \cite{Co}, where the operator $df:=i[F,\pi (f)]$ is the operator theoretical substitute for the differential of the function $f$.
\par\noindent
On the other hand, the notion of spectral triple $(\ca, \ch, D)$ is meant to encode the {\it metric} aspects of a generalized geometry. To remain in a commutative framework, where the relevant C$^*$-algebra is $C(K)$ as above, a suitable dense subalgebra $\ca$ of $C(K)$ is required to act on a Hilbert space $\ch$ together with an unbounded, self-adjoint operator $D$, having discrete spectrum, in such a way that the commutators $[D,f]$ are bounded for all $a\in \ca$.
\par\noindent
The archetypical example of a spectral triple arises in Riemannian geometry where $D$ is the Dirac operator on the Clifford bundle and the boundedness of the commutator above is equivalent to the Lipschitz continuity of the function involved.
\par\noindent
The metric aspects of a space endowed with a spectral triple are recovered by regarding the compact operator $|D|^{-1}$ as the operator theoretic realization of the infinitesimal arc element $ds$, whereas its local features are recovered through the asymptotic behavior of the spectrum of $D$ and related operators.

\begin{dfn}(Spectral Triple)
A (possibly kernel-degenerate, compact) Spectral Triple $(\ca, \ch, D)$ consists of an involutive unital algebra $\ca$, acting faithfully on a Hilbert space $\ch$ trough a representation $\pi$, and a self-adjoint operator $(D, \dom (D))$ on it, subject to the conditions
\begin{itemize}
\item[$(i)$] the commutators $[D,\pi(a)]$, initially defined on the domain $\dom (D)\subset\ch$ through the sesquilinear forms
\[
(\xi,[D,\pi(a)]\eta):=(D\xi ,\pi(a)\eta)-(\pi(a)^*\xi,D\eta)\qquad \xi ,\eta\in dom (D)\, ,
\]
extend to bounded operators on $\ch$, for all $a\in\ca$;
\item[$(ii)$] the operator $D^{-1}$ is compact on $\ker(D)^\perp$.
\end{itemize}
The operator $(D, \dom (D))$ is referred to as the {\it Dirac operator} of the Spectral Triple.
\end{dfn}
Notice that if ${\rm ker}\, (D)$ is finite dimensional the condition in $(ii)$ reduces to the compactness of the operators $(I+D^2)^{-1}$. We recover in this way the original definition of a Spectral Triple by Connes \cite{Co}.
%The extended version  is required to deal with the Dirac operators we will consider on quasicircles and on the Sierpinski Gasket.
%\par\noindent
%Notice also that the requirement in $(ii)$ amounts to the discreteness of the spectrum of the Dirac operator $D$.
%
%\vskip0.2truecm
%
%\begin{ex} \label{DiracOnManifold}
%As an example of Spectral Triple, where the algebra is commutative, consider the algebra $\ca :=C^\infty (M)$ of smooth functions on a compact Riemannian manifold $M$, acting on the Hilbert space $\ch :=L^2 (\Lambda (M))$ of square integrable sections of the exterior bundle over $M$, as well the self-adjoint operator $D:=d+d^*$, sum of the exterior differential $d$ and its adjoint. Since  the square of the Dirac operator is just the Hodge-de Rham operator of $M$,
%\[
%D^2 =(d+d^*)^2 = dd^* +d^* d =\Delta_{HdR}\, ,
% \]
%condition $(ii)$ is verified by the discreteness of the spectrum of $\Delta_{HdR}$, which follows from the compactness of $M$. The rules of differential calculus on $M$ allow to easily verify that the norm of a commutator coincides with the Lipschitz semi-norm of a function
%\[
%\|[D,a]\|=\sup\{|d_xa|_{T^*_x M} :x\in M\}\, .
%\]
%This also shows that the Riemannian distance on $M$ can be recovered from the Spectral Triple:
%\[
%d_M(x,y)=\sup\{ |a(x)-a(y)| : a\in\ca, \|[D,a]\|\leq 1 \}\qquad x,y\in M\, .
%\]
%Notice that in this case the kernel of $D$ is just the space of harmonic forms on $M$.
%\end{ex}

\begin{dfn}(Fredholm Module) \label{FredholmModule}
A (possibly kernel-degenerate) Fredholm Module $(\ch,\pi,F)$ over a C$^*$-algebra $A$ consists of  a Hilbert space $\ch$, a representation $\pi:A\to\cb(\ch)$, and a bounded operator $F\in\cb(\ch)$ such that
\begin{itemize}
\item[$(i)$] $F^2 -I$ is a compact operator on ${\rm ker}\, (F)^\perp$,
\item[$(ii)$] $F^* -F$ is a compact operator,
\item[$(iii)$] the commutators $[F,\pi(a)]$ are compact operators, for all $a\in A$.
\end{itemize}
%It follows from the completeness of the space of compact operators in $\cb(\ch)$ that, if $(\ch,\pi,F)$ is a Fredholm module over $\ca$, then it extends automatically to a Fredholm module over the C$^*$-algebra $\overline\ca$, closure of the algebra $\ca$ in  $\cb(\ch)$.
\end{dfn}
\begin{rem}
In the following, in many occasions, the representation $\pi$ will be omitted, the $*$-algebra $\ca$ and the C$^*$-algebra $A$ being identified with subalgebras of $\cb(\ch)$.
\end{rem}

The classical formulation of Atiyah is recovered when $\ker(F)$ is finite dimensional. Again, the above generalization is required to deal with the Fredholm Modules we will construct on quasicircles and on the Sierpinski Gasket.
%\vskip0.2truecm

%\begin{ex}
%As an example of Fredholm Module,  consider the algebra $A :=C(M)$ of continuous functions on a compact Riemannian manifold $M$, acting on the Hilbert space $\ch :=L^2 (\Lambda (M))$ of square integrable sections of the exterior bundle over $M$, as well the sign $F:=\sgn (D)$ of the Dirac operator $D:=d+d^*$ considered in Example \ref{DiracOnManifold}. Here $F$ is self-adjoint, $\ker (F)$ is the finite dimensional space of harmonic forms, so that $F^* =F$, $F^2 -I$ is a finite rank operator, and the first two requirements for a Fredholm Module are fulfilled. The third one follows from the fact that the commutator $[F,a]$ is clearly a pseudo-differential operator of order $-1$ on $M$, so that it is a compact operator, by a well known result of Analysis.
%\end{ex}

\vskip0.2truecm

A classical result by Baaj and Julg \cite{BaJu} shows that Spectral Triples give rise to Fredholm modules by taking $F=\sgn(D)$ (or any other function which is asymptotic to $\sgn(t)$ for $|t|\to\infty$), whenever ${\rm dim\,ker}\,(D)<+\infty$. We need to generalize this result to allow ${\rm dim\,ker}\,(D)=+\infty$. The key point is to show that the boundedness of the commutator $[D,a]$ implies the compactness of $[F,a]$ even if $\ker(D)$ is not finite dimensional.

\begin{prop} \label{triple-module}
Let $(\ca, \ch, D)$ be a (possibly kernel-degenerate) Spectral Triple over  a unital $*$-algebra $\ca\subset\cb(\ch)$. Then, setting $F:= \sgn(D)$, $(F, \ch)$ is a (possibly kernel-degenerate) Fredholm module over the C$^*$-algebra $A$ given by the norm closure $\overline\ca$ of the $*$-algebra $\ca$.
\end{prop}

\begin{proof}
Since $F$ is self-adjoint and $F^2 -I$ vanishes  on ${\rm ker}\, (F)^\perp = {\rm ker}\, (D)^\perp$ by construction, the first two requirements in the Definition \ref{FredholmModule} of a Fredholm module hold true.
\par\noindent
To verify the third requirement, let us first observe that $\frac1{\sqrt x} = \frac2\pi\, \int_0^{+\infty} \frac{dt}{x+t^2}$ for any $x>0$, from which it follows that $F = \frac2\pi\, \int_0^{+\infty} \frac{D}{t^2+D^2} \, dt$, where the integral converges in norm since
\[
\frac2\pi\, \int_k^{+\infty} \frac{D}{t^2+D^2}\, dt =
\begin{cases}
F\big(I-\frac{2}{\pi} \arctan{\big(  k|D|^{-1} \big)} \big) & \text{on}\ (\ker D)^\perp \\
0 & \text{on}\ \ker D
\end{cases}
\]
hence
\[
\|\frac2\pi\, \int_k^{+\infty} \frac{D}{t^2+D^2}\, dt \|\le 1-\frac{2}{\pi}\arctan{\big(\frac{k}{\lambda}\big)}
\]
where $\lambda >0$ is the first non-zero eigenvalue of $|D|$. Let now $a$ belong to $\ca$, so that $a\dom(D)\subset\dom(D)$ and $Da-aD$ is bounded on $\dom(D)$.
Our aim is to prove first that the operator defined as
\[
 C:=\frac2\pi\, \int_0^{+\infty}
\frac{t}{t^2+D^2} \, [D,a] \,  \frac{t}{t^2+D^2}\, dt
- \frac2\pi\, \int_0^{+\infty} \frac{D}{t^2+D^2} \, [D,a] \,  \frac{D}{t^2+D^2} \, dt
\]
is well defined and compact and then to show that it coincides with the commutator $[F,a]$.
\par\noindent
Observe that $t\in(0, +\infty)\mapsto D(t^2+D^2)^{-1}$ is a $\ck(\ch)$-valued continuous function that can be continuously extended to $[0,+\infty)$ by assigning to it the value $D/D^{2}\in\ck (\ch)$ at $t=0$. Here we are denoting by $D/D^{2}$ the compact operator which is the inverse of $D$ on  ${\rm ker}\, (D)^\perp$ and vanishes on ${\rm ker}\, (D)$. Indeed, with  $\l$ as above, we have
$$
\Big\| \frac{D}{t^2+D^2}-\frac{D}{D^2} \Big\|
= \Big\| \frac{t^2}{(t^2+D^2)}\frac{D}{D^2} \Big\|
= \frac{t^2}{(t^2+\l^2)\l}\to 0\qquad t\to 0\, .
$$
Since $\ck (\ch)$ is a closed subspace of $\cb (\ch)$, and
$$
\Big\| \frac{D}{t^2+D^2} \, [D,a] \,  \frac{D}{t^2+D^2} \Big\|
=\|[D,a]\|\cdot \Big\| \frac{D}{t^2+D^2} \Big\|^2
\leq\frac14\|[D,a] \| \,t^{-2} \in L^1([1,\infty))\, ,
$$
we have that $\frac2\pi\, \int_0^{+\infty} \frac{D}{t^2+D^2} \, [D,a] \,  \frac{D}{t^2+D^2}  dt$ is a compact operator.
\vskip0.2truecm\noindent
On the other hand, $t\in(0,+\infty)\mapsto t(t^2+D^2)^{-1} \, [D,a] \,  t(t^2+D^2)^{-1} $ is a $\ck(\ch)$-valued continuous function which can be extended continuously to $t=0$. Indeed, when restricted to ${\rm ker}\, (D)^\perp$ it appears as a continuous function of $t\in [0,+\infty)$ of products of operators in which at least one factor is compact, and when restricted to $\ker (D)$ it reduces to
\[
\frac{t}{t^2+D^2} \, [D,a] \,  \frac{t}{t^2+D^2} = \frac{t}{t^2+D^2} \, (Da-aD) \,  \frac{t}{t^2+D^2} =
\frac{t}{t^2+D^2} \, Da \,  \frac{1}{t}  = \frac{D}{t^2+D^2} \, a
\]
so that it converges in $\cb (\ch)$ to the compact operator $(D/D^{2})\, a$ as $t\to 0$. Since, again, $\ck (\ch)$ is a closed subspace of $\cb (\ch)$ and
$$
\Big\| \frac{t}{t^2+D^2} \, [D,a] \,  \frac{t}{t^2+D^2} \Big\|
=\|[D,a]\|\cdot \Big\| \frac{t}{t^2+D^2} \Big\|^2
\leq\|[D,a]\|\,t^{-2}\in L^1 ([1,\infty))\, ,
$$
we have that $ \frac2\pi\, \int_0^{+\infty} \frac{t}{t^2+D^2} \, [D,a] \,  \frac{t}{t^2+D^2}\, dt$ is a compact operator.
\par\noindent
By the formulas above we have
\[
[F,a] = \frac2\pi\, \int_0^{+\infty} \bigg( \frac{D}{t^2+D^2} \, a - a \,  \frac{D}{t^2+D^2} \bigg)\, dt
\]
where the integral converges in norm. Therefore the identity $C=[F,a]$ follows if we prove that in the integral representations of $C$ and $[F,a]$, the integrands coincide as quadratic forms.
\par\noindent
In fact, for fixed $\xi,\eta\in \ch$, the vectors $\frac{t}{t^2+D^2}\xi\, ,\frac{t}{t^2+D^2}\eta\, \frac{D}{t^2+D^2}\xi\, ,\frac{D}{t^2+D^2}\eta$ belong to $\dom(D)$, and
\begin{align*}
& \bigg(\xi,\bigg(\frac{t}{t^2+D^2} \, [D,a] \,  \frac{t}{t^2+D^2} - \frac{D}{t^2+D^2} \, [D,a] \,  \frac{D}{t^2+D^2}\bigg)\eta\bigg) \\
& = \bigg(\frac{t}{t^2+D^2}\xi,[D,a] \,  \frac{t}{t^2+D^2}\eta\bigg)
- \bigg(\frac{D}{t^2+D^2}\xi,[D,a] \,  \frac{D}{t^2+D^2}\eta\bigg) \\
& = \bigg(\frac{D}{t^2+D^2}\xi, a\frac{t^2}{t^2+D^2}\eta\bigg)
- \bigg(\frac{D^2}{t^2+D^2}\xi,  a\frac{D}{t^2+D^2}\eta\bigg) \\
& - \bigg(\frac{t^2}{t^2+D^2}\xi, a\frac{D}{t^2+D^2}\eta\bigg)
+ \bigg(\frac{D}{t^2+D^2}\xi,  a\frac{D^2}{t^2+D^2}\eta\bigg) \\
& = \bigg(\frac{D}{t^2+D^2}\xi, a\eta\bigg)
- \bigg(\xi,  a\frac{D}{t^2+D^2}\eta\bigg) \\
& = \bigg(\xi ,\bigg( \frac{D}{t^2+D^2} \, a - a \,  \frac{D}{t^2+D^2} \bigg)\eta\bigg)\, .
\end{align*}
\end{proof}

Recall that a {\it symmetry} of a Hilbert space $\ch$ is a bounded operator $\gamma\in\cb(\ch)$  such that
\[
\gamma^* =\gamma\, ,\qquad \gamma^2 =I\, .
\]
\begin{dfn}(Even and odd Spectral Triples and Fredholm Modules)
\vskip0.1truecm\noindent
A Spectral Triple $(\ca, \ch, D)$ is called {\it even} if there exists a symmetry $\g$ such that
\[
D\gamma + \gamma D =0\, ,\qquad a\gamma - \gamma a =0\, \qquad a\in\ca\, .
\]
A Fredholm Module  $(F, \ch)$ is called {\it even} if there exists a symmetry $\g$ such that
\[
F\gamma + \gamma F =0\, ,\qquad a\gamma - \gamma a =0\, \qquad a\in\ca\, .
\]
\end{dfn}
In other words, the operator $D$ (resp. $F$) of an even Spectral Triple (resp. Fredholm module) acts as an antidiagonal matrix  with respect to the orthogonal decomposition of $\ch =\ch_+\oplus\ch_-$ in eigenspaces $\ch_\pm$ of the symmetry $\gamma$ corresponding to its eigenvalues $\pm 1$, while the elements of $\ca$ act diagonally.

\begin{cor}
Let $(\ca, \ch, D, \gamma)$ be an even Spectral Triple. Then, setting $F:= \sgn(D)$, $(F, \ch, \gamma)$ is an even Fredholm module over the C$^*$-algebra $\overline\ca$.
\end{cor}

Fredholm modules represent elements of the $K$-homology groups of the algebra $A$ \cite{Co}. These can be paired with elements of the $K$-theory groups of $A$. In particular, odd Fredholm modules couple with elements of the group $K_1 (A)$, whose elements are represented by invertible or unitary  elements of $A$.
Indeed,  assume $F$ to be selfadjoint.
In this case, for any invertible element $u\in A$, the operator $P_+\pi(u)P_+$ is Fredholm on $P_+\ch$, where $P_+$ is the projection on the positive part of the spectrum of $F$, and the pairing  is given by
$$
\langle F,u\rangle=\ind(P_+\pi(u)P_+).
$$
In the following, we allow $F$ to have a infinite dimensional kernel. The following Proposition justifies in some cases the treatment of such  kernel-degenerate Fredholm modules.

\begin{prop}\label{tamedegeneration}
Let $F$ be a self-adjoint operator whose spectrum is $\sigma (F):=\{-1,0,1\}$. Assume $[F,\pi(a)]$ to be compact for any $a\in A$, and denote by $P_ \l$ the spectral projection for the eigenvalue $ \l\in\{-1,0,1\}$. Then
\begin{itemize}
\item[$(i)$] when $u$ is invertible, $P_ \l\pi(u)P_ \l$ is Fredholm, for all $ \l\in\{-1,0,1\}$, and
\[
\sum_ \l\ind(P_ \l\pi(u)P_ \l) =0\, ,
\]
\item[$(ii)$] $[P_ \l,\pi(a)]$ is compact, for any $a\in A$, and $ \l\in\{-1,0,1\}$.
\end{itemize}
\end{prop}
\begin{proof}
$(i)$ As $[F,\pi(u)]\in \ck(\ch)$, for all $ \l, \l'\in \{-1,0,1\}$, we have $P_ \l[F,\pi(u)]P_{ \l'}\in \ck(\ch)$,
\[
P_ \l[F,\pi(u)]P_{ \l'} = ( \l - \l')P_ \l\pi(u)P_{ \l'}
\]
and, in particular, $P_ \l[F,\pi(u)]P_{ \l}=0$.
Since
\[
\pi (u) = \sum_{ \l} P_ \l\pi(u)P_{ \l} + \sum_{ \l\ne \l'} ( \l - \l')^{-1}P_ \l[F,\pi(u)]P_{ \l'}
\]
and $u$ is invertible, then $\sum_{ \l} P_ \l\pi(u)P_{ \l}$ and $P_ \l\pi(u)P_{ \l}$ are Fredholm operators,  for all $ \l\in\{-1,0,1\}$, and
\[
\sum_{ \l}\ind(P_ \l\pi(u)P_ \l)=\ind(\pi(u))=0\, .
\]
\\
$(ii)$ Observe that, for all $a\in\ca$ and $\set{ \l,  \l',  \l''}$ a permutation of $\{-1,0,1\}$, we have
\begin{align*}
[P_ \l,\pi(a)]&=\sum_{\l,\m\in\s(F)}P_{\l}[P_ \l,\pi(a)]P_{\m}
=P_ \l\pi(a)P_{ \l'} + P_ \l\pi(a)P_{ \l''} - P_{ \l'}\pi(a)P_ \l - P_{ \l''}\pi(a)P_ \l.
\end{align*}
Since all summands in the last expression are compact by $(i)$, the thesis follows.
\end{proof}

\begin{cor} \label{TameDegeneracy}
Let $(\ch, \pi,  F)$ be a Fredholm module over a C$^*$-algebra $A$, in the sense of Definition \ref{FredholmModule}, with $F^*=F$, and $F^2=I$ on $(\ker F)^\perp$, and assume that for all invertible $u\in A$ we have
\begin{equation*}
\ind (P_0\pi (u)P_0)=0\, .
\end{equation*}
Then there exists a Fredholm module $(\ch, \pi,  F')$ such that $F'^2 = I$ and
\[
\ind (P_ \l\pi (u)P_ \l)=\ind (P'_ \l \pi (u)P'_ \l)\qquad  \l =-1\, ,+1\, .
\]
Here $P_ \l$ (resp. $P'_ \l$) denotes the projection of the operator $F$ (resp. $F'$) associated to the eigenvalue $ \l\in\{-1 , +1\}$.
\end{cor}
\begin{proof}
Defining $F' := F+P_0$ we have $F'^* =F'$, $F'^2 =I$ and $\sigma (F')=\{-1 ,+1\}$. Since $[F',\pi(a)]=[F,\pi(a)]+[P_0,\pi(a)]$,  by Proposition \ref{tamedegeneration} $(ii)$ we have $[F',\pi(u)]\in \ck(\ch)$, so that $(\ca, (\pi,\ch), F')$ is a Fredholm module. Finally, since $P'_1 =P_1  + P_0$, and $P_1\pi(u)P_0), P_0\pi(u)P_1$ are compact, by the proof of Proposition \ref{tamedegeneration} $(i)$, and since, by assumption, $\ind(P_0\pi(u)P_0)=0$, we have
\[
\ind(P'_1\pi(u)P'_1)=\ind((P_1 + P_0)\pi(u)(P_1 + P_0))=\ind(P_1\pi(u)P_1)+\ind(P_0\pi(u)P_0)=\ind(P_1\pi(u)P_1)\, .
\]
\end{proof}

\begin{dfn} \label{dfn:kernelDeg}
A (possibly kernel-degenerate) odd Fredholm module $(\ch,\pi,F)$ will be called tamely degenerate if
\begin{equation}\label{tameF}
\ind (P_0\pi (u)P_0)=0\, ,
\end{equation}
for all invertible $u\in Mat_k(A)$, $k\in\bn$, where $P_0$ denotes the projection onto $\ker F\otimes I_{\bc^k}$.
\end{dfn}
Corollary \ref{TameDegeneracy} proves that a tamely degenerate Fredholm module is equivalent to a (non-kernel-degenerate) Fredholm module, as far as their indexes are concerned.

\medskip

We now recall the definition of $p$-graded Fredholm module.

\begin{dfn} \label{dfn:pgrad} [\cite{HiRoBook}, Defs 8.1.11 \& A.3.1]  Let $p \in \{-1,0, ,...\}$. A $p$-graded Fredholm module over a C$^*$-algebra
 $A$ is given by the following data:
\begin{itemize}
\item[(a)] a separable  Hilbert space $\ch$;
\item[(b)]    $p+1$ unitary operators $\eps_0,... , \eps_p$
such that $\eps_i\eps_j + \eps_j \eps_i = 0$ if $i\ne j$, $ \eps_i^2 = -1$, for $i\ne0$, $ \eps_0^2 = 1$.
\item[(c)] a representation $\pi: A \to \cb(\ch)$ such that
$[\eps_i,\pi(a)]=0$ for any $i=0,\dots,p$, any $a\in\ca$
\item[(d)] an operator $F$ on $\ch$ such that $\eps_i F-F\eps_i=0$,  $i\ne0$,  $\eps_0 F+F\eps_0=0$, and, for all $a \in \ca$,
$(F^2-1)$, $F-F^*$, $[F,\pi(a)]$ are compact.
\end{itemize}
In particular, odd Fredholm modules are $(-1)$-graded, and  even Fredholm modules are $(0)$-graded.
\end{dfn}

Endowed with the equivalence relation given by stable homotopy \cite{HiRoBook}, the set of equivalence classes of $p$-graded Fredholm modules becomes an abelian group, with addition given by direct sum, which is denoted $K^{-p}(A) = KK^{-p}(A,\bc)$, and called $(-p)$-th K-homology group of $A$. Because of Bott periodicity (cf. Proposition 8.2.13 in \cite{HiRoBook}), $K^{-p}(A)$ and $K^{-p-2}(A)$ are naturally isomorphic, so there are really two K-homology groups of $A$, the odd one $K^1(A)$, and the even one $K^0(A)$.
It turns out that (equivalence classes of) $p$-graded Fredholm modules pair with odd $K$-theory when $p$ is odd, and with even $K$-theory, when $p$ is even.

A particular instance of Bott periodicity, which we will need in the following sections, says that, given a 1-graded Fredholm module $\cf=(\ch,\pi,F,\g,\eps)$, and setting $\ch=\ch^+\oplus\ch^-$, $\pi=\pi^+\oplus\pi^-$, $F = \begin{pmatrix} 0 & F_{12} \\ F_{21} & 0\end{pmatrix}$, $\g = \begin{pmatrix} I_{\ch^+} &0\\ 0 & -I_{\ch^-} \end{pmatrix}$, $\eps = \begin{pmatrix} 0 & -iV \\ -iV^* & 0\end{pmatrix}$, $F^+ = V F_{21} = F_{12}V^*$, then $\cf^*=(\ch^+,\pi^+,F^+)$ is an odd Fredholm module on $A$, giving the same pairing with K-theory. Proposition \ref{Ex8.8.4} shows that weakening some of the conditions in the definition of $1$-graded module does not alter the previous result.

Let us observe that, given an even Fredholm module $(\pi,\ch,F,\g)$ on $\ca$, we can make it a 1-graded Fredholm module $(\pi,\ch,F,\g,\eps)$ simply by adding a skew-adjoint unitary operator $\eps$ which commutes with $F$, anticommutes with $\g$, and commutes with $\pi(a)$ (possibly up to compact operators).

\begin{dfn}($1$-graded Fredholm Module) \label{FredholmModule1graded}
A (possibly kernel-degenerate) $1$-graded Fredholm Module $(\ch,\pi,F,\g,\eps)$ over a C$^*$-algebra $A$,  consists of an even (possibly kernel-degenerate) Fredholm Module  $(\ch,\pi,F,\g)$, and an operator $\eps\in\cu(\ch)$ such that
\begin{itemize}
\item[$(i)$] $\eps^2 +I = 0$ on ${\rm ker}\, (F)^\perp$,
\item[$(ii)$] $\eps^* +\eps  =0$,
\item[$(iii)$] the commutators $[\eps,\pi(a)]$ are compact operators, for all $a\in A$.
\end{itemize}
\end{dfn}

\begin{prop}\label{Ex8.8.4}
	Let $(\ch,\pi,F,\g,\eps)$ be a (possibly kernel-degenerate) $1$-graded  Fredholm module, with $F$ self-adjoint.
	Then $\eps_0 = P^+-P^-$, where $P^\pm\in\Proj(\ch)$, $P^++P^-=I$. Setting $\ch^\pm:= P^\pm\ch$, one gets $\pi = \pi^+\oplus \pi^-$, $\eps = \begin{pmatrix} 0 & -iV \\ -iV^* & 0 \end{pmatrix}$, where $V:\ch^-\to\ch^+$ is a partial isometry, and $F = \begin{pmatrix} 0 & F_{12} \\ F_{21} & 0 \end{pmatrix}$. Setting $F^+ = VF_{21} = F_{12}V^*$, $F^- = V^*F_{12} = F_{21}V$, we have that the spectrum of $F^\pm$ is $\{-1,0,1 \}$, and let $P^\pm$, $N^\pm$, $Z^\pm$ be the spectral projections on the positive, negative, and zero eigenvalue of $F^\pm$.

Moreover, $(\ch^+,\pi^+, F^+)$, $(\ch^-,\pi^-, F^-)$ are  (possibly kernel-degenerate) odd  Fredholm modules, and, for all invertible $u\in Mat_k(A)$, it holds (with $P^\pm$ denoting $P^\pm\otimes I_{\bc^k}$, and analogously for $N^\pm$, $Z^\pm$, and $\pi^\pm$ properly extended to $Mat_k(A)$)
\begin{align*}
	\ind(P^+\pi^+(u)P^+) & = \ind(P^-\pi^-(u)P^-), \\
	\ind(N^+\pi^+(u)N^+) & = \ind(N^-\pi^-(u)N^-), \\
	\ind(Z^+\pi^+(u)Z^+) & = \ind(Z^-\pi^-(u)Z^-).
\end{align*}
Finally, if any of the modules $(\ch,\pi,F,\g,\eps)$,  $(\ch^+,\pi^+, F^+)$, $(\ch^-,\pi^-, F^-)$ is tamely degenerate, then so are the other two. %$\Ind(P^+\pi^+(u)P^+) = \Ind(P^+\pi^+(u)P^+)$
\end{prop}
\begin{proof}
In the course of this proof we set $A\approx B$ to mean equality modulo compact operators, $i.e.$ $A-B$ is a compact operator.
From the properties of $\g$, it follows that $\ch=\ch^+\oplus\ch^-$, where $\ch^\pm$ is the eigenspace relative to the eigenvalue $\pm1$ of $\g$. Moreover $\pi = \pi^+\oplus \pi^-$, where $\pi^\pm  A \to \cb(\ch^\pm)$ is a representation of $A$. From $\g\eps+\eps\g=0$ and $\eps+\eps^*=0$ it follows $\eps = \begin{pmatrix} 0 & iV \\ iV^* & 0 \end{pmatrix}$, where $V: \ch^-\to\ch^+$.
 In addition, for $a\in A$, we have $0 \approx \eps\pi(a)-\pi(a)\eps = i \begin{pmatrix} 0& V\pi^-(a) - \pi^+(a)V \\ V^*\pi^+(a)-\pi^-(a)V^* &0 \end{pmatrix}  \implies V\pi^-(a) \approx \pi^+(a)V$.

As for $F$, $0 = F\g+\g F \implies F = \begin{pmatrix} 0 & F_{12} \\ F_{21} & 0 \end{pmatrix}$, whereas $0 = F\eps-\eps F  \implies F_{12}V^* = VF_{21}$, $V^*F_{12} = F_{21}V$. Moreover, denoting by $P_0$ the projection onto $\ker F$, we have $\begin{pmatrix} VV^* & 0 \\ 0 & V^*V \end{pmatrix} = -\eps^2 = I-P_0 = F^2 = \begin{pmatrix} F_{12}F_{21} & 0 \\ 0 & F_{21}F_{12} \end{pmatrix}$, so that $VV^* = F_{12}F_{21}$, and $V^*V = F_{21}F_{12}$ are projections, so that $V$ is a partial isometry.

Let us set $F^+=VF_{21}$, $F^-=V^*F_{12}$, so that $(F^+)^* = F_{21}^*V^* = F_{12}V^* = F^+$, $\begin{pmatrix} F^+ & 0 \\ 0 & F^- \end{pmatrix} = -i\eps F \implies \begin{pmatrix} (F^+)^2 & 0 \\ 0 & (F^-)^2 \end{pmatrix} = -\eps F \eps F = -\eps^2 F^2 = F^2 = I-P_0$, which implies that the spectrum of $F^\pm$ is $\{-1,0,1 \}$, and let $P^\pm$, $N^\pm$, $Z^\pm$ be the spectral projections on the positive, negative, and zero eigenvalue of $F^\pm$. Therefore, $F^\pm = P^\pm-N^\pm$, and $P^++N^+ = I-Z^+ = VV^*$, $P^-+N^- = I-Z^- = V^* V$. Moreover, $F^+V=VF_{21}V = VF^-$, so that $(P^+-N^+)V = V(P^--N^-)$. Besides, $(P^++N^+)V = VV^*V = V(P^-+N^-)$, from which we conclude $P^+V=VP^-$, and $N^+V=VN^-$.

In order to conclude that $(\ch^\pm,\pi^\pm, F^\pm)$ are odd Fredholm modules, we only need to prove  the properties of $F^\pm$.
For all $a\in A$, we have
\[
	I-P_0 \approx  (F^2-I)\pi(a)   = \begin{pmatrix} \big( (F^+)^2 -I \big) \pi^+(a) & 0 \\ 0  &\big( (F^-)^2-I \big)\pi^-(a) \end{pmatrix}
	\]
$\implies  \big( (F^\pm)^2 -I \big) \pi^\pm(a)  \approx I-Z^\pm$.
$0  \approx  F\pi(a)-\pi(a)F  \implies$
\begin{align*}
 0  \approx - i\eps F\pi(a)+ i\eps \pi(a)F &\approx -i\eps F\pi(a)+ \pi(a) i\eps F \\
&=   \begin{pmatrix}  F^+\pi^+(a)-\pi^+(a)F^+ & 0 \\ 0 & F^-\pi^-(a)-\pi^-(a)F^- \end{pmatrix}
\end{align*}
$\implies  F^\pm\pi^\pm(a)-\pi^\pm(a)F^\pm \approx 0$.
%che implica $F_{21}\pi^+(a)-\pi^-(a)F_{21} \in\ck(\ch^-,\ch^+)$, $F_{12}\pi^-(a)-\pi^+(a)F_{12} \in\ck(\ch^+,\ch^-)$, da cui segue $F'\pi^+(a)-\pi^+(a)F' = iV^*F_{21}\pi^+(a)-\pi^+(a)iV^*F_{21} =  iV^*F_{21}\pi^+(a)-iV^*\pi^-(a)F_{21} = iV^*(F_{21}\pi^+(a)-\pi^-(a)F_{21}) \in\ck(\ch^+)$.
%
%Inoltre si ha, per ogni $a\in A$,
$0 \approx  (F^*-F)\pi(a) \implies$
\begin{align*}
	 0 \approx& -i\eps (F^*-F)\pi(a) = -(iF^*\eps-i\eps F)\pi(a) =  (iF^*\eps^*+i\eps F)\pi(a) \\
= & \big( (-i\eps F)^*-(-i\eps F)\big)\pi(a)
	 = \begin{pmatrix}  \big(  (F^+)^*-F^+ \big)\pi^+(a) & 0 \\ 0 & \big(  (F^-)^*-F^- \big)\pi^-(a) \end{pmatrix}
\end{align*}
$\implies  \big(  (F^\pm)^*-F^\pm \big)\pi^\pm(a) \approx 0$,
that is, $(\ch^\pm,\pi^\pm, F^\pm)$ are (possibly kernel-degenerate) odd Fredholm modules.
%che implica $\big( (F_{21})^*-F_{12}\big)\pi^-(a) \in\ck(\ch^-,\ch^+)$, $\big((F_{12})^*-F_{21}\big)\pi^+(a) \in\ck(\ch^+,\ch^-)$, da cui segue $(F'^*-F')\pi^+(a) = -i(F_{21})^*V\pi^+(a)-iV^*F_{21}\pi^+(a) =  iV^*(F_{12})^*\pi^+(a)-iV^*F_{21}\pi^+(a) = iV^* \big( (F_{12})^*-F_{21} \big)\pi^+(a) \in\ck(\ch^+)$.
%
%Infine si ha, per ogni $a\in A$,
%\begin{align*}
%	\ck(\ch)\ni (F^2-I)\pi(a) & = \begin{pmatrix} 0&F_{12}\\ F_{21}&0 \end{pmatrix} \begin{pmatrix} 0&F_{12}\\ F_{21}&0 \end{pmatrix}  \begin{pmatrix} \pi^+(a) & 0 \\ 0& \pi^-(a) \end{pmatrix} - \begin{pmatrix} \pi^+(a) & 0 \\ 0& \pi^-(a) \end{pmatrix} \\
%	& = \begin{pmatrix}  \big( F_{12}F_{21}-I \big)\pi^+(a) & 0 \\ 0 & \big( F_{21}F_{12}-I\big)\pi^-(a)  \end{pmatrix}
%\end{align*}
%che implica $\big( F_{12}F_{21}-I \big)\pi^+(a) \in\ck(\ch^+)$, $\big( F_{21}F_{12}-I\big)\pi^-(a) \in\ck(\ch^-)$, da cui segue $(F'^2-I)\pi^+(a) = (-V^*F_{21}V^*F_{21}-I)\pi^+(a) =  ( V^*VF_{12}F_{21}-I)\pi^+(a) = (F_{12}F_{21}-I)\pi^+(a) \in\ck(\ch^+)$.
Finally, for all invertible $u\in A$, we have
\begin{align*}
	\ind(P^+\pi^+(u)P^+) & = \ind(P^+VV^*\pi^+(u)VV^*P^+) = \ind(VP^-V^*V\pi^-(u)P^-V^*)  \\ & = \ind(VP^-\pi^-(u)P^-V^*)  = \ind(P^-\pi^-(u)P^-),
\end{align*}
where the second equality follows from the intertwining properties of $V$, and $V\pi^-(a) \approx \pi^+(a)V$, and the last equality follows from the fact that $V\in\cu(P^-\ch,P^+\ch)$. The equality $\ind(N^+\pi^+(u)N^+)  = \ind(N^-\pi^-(u)N^-)$ is proved analogously, whereas $\ind(Z^+\pi^+(u)Z^+)  = \ind(Z^-\pi^-(u)Z^-)$ follows from the above and Proposition \ref{tamedegeneration} $(i)$. An analogous argument proves the above equalities for any invertible $u\in Mat_k(A)$.

Therefore, if $(\ch,\pi,F,\g,\eps)$ is tamely degenerate, $0 = \ind(P_0\pi(u)P_0) = \ind(Z^+\pi^+(u)Z^+)  + \ind(Z^-\pi^-(u)Z^-)$, which implies $\ind(Z^+\pi^+(u)Z^+)  = \ind(Z^-\pi^-(u)Z^-) = 0$, that is $(\ch^\pm,\pi^\pm,F^\pm)$ are tamely degenerate. Viceversa,  $(\ch^+,\pi^+,F^+)$ is tamely degenerate $\iff (\ch^-,\pi^-,F^-)$ is, and in this case $(\ch,\pi,F,\g,\eps)$ is tamely degenerate as well.
\end{proof}

\subsection{Spectral triples on self-similar fractals}\label{triplesonfractals}
A self-similar fractal (in $\br^n$) is described by a finite set of similitudes $w_1,\dots w_k$, with scaling parameters $\l_1,\dots \l_k$, $\l_i<1$, as the unique compact set $X$ such that
$$
\bigcup_{i=1}^k w_i(X)=X.
$$
A standard way to construct spectral triples on such fractal is the following:
\begin{itemize}
\item Select a subset $S\subset X$ together with a triple $\ct=(\pi,\ch,D)$ on $\cc(S)$.
\item Set  $\ct_\emptyset=(\pi_\emptyset,\ch_\emptyset,D_\emptyset)$ on $\cc(X)$, where $\pi_\emptyset(f)=\pi(f|_S)$, $\ch_\emptyset =\ch$, $D_\emptyset=D$.
\item Set  $\ct_\sigma:=(\pi_\sigma,\ch_\emptyset,D_\sigma)$ on $\cc(X)$, with
$\pi_\sigma(f)=\pi_\emptyset(f\circ w_\sigma)$, $D_\sigma=\lambda_\sigma^{-1}D_\emptyset$, $\lambda_\sigma=\prod_{i=1}^{|\sigma|}\l_{\sigma_i}$.
\item Set $\ct=\bigoplus_\sigma\ct_\sigma$ on $\cc(X)$ and consider the $^*$-algebra  $\ca=\{f\in\cc(X):[D,f]$ is bdd\}.
\end{itemize}
This type of construction was used in  \cite{GuIs9,GuIs10} to reproduce some of the features of the fractals. It was also the basis of the construction of the triples for the Sierpinski gasket $K$ in \cite{CIL,CIS}, by  choosing $S$ as the lacuna $\ell_\emptyset$ isometrically identified with the circle $\bt$, with the standard triple given by $\ch=L^2(\bt)$ and $D=-id/d\th$.

As shown below (Lemma \ref{traceClass}), this choice does not allow the recovery of the energy via residues. In order to do so, one has to deform the triple on the circle  by replacing the standard Laplacian $\D$ with one of its fractional powers $\D^\a$, $\a<1$. However, in order to obtain a Dirac operator $D$ based on a suitable ``differential'' square root of $\D^\a$, we need to double the Hilbert space.
So, the deformed triples on $\bt$, which we describe below, are obtained by deforming the triple $(\ca,\ck,D)$, where $\ca=Lip(\bt)$,
%=\{f\in\cc(\bt):\|[D,f]\|$ is bounded\},
$\ck:=L^2(\O^*(\T))=L^2(\O^1(\T))\oplus L^2(\O^0(\T))$, and
$D=\begin{pmatrix}
0&d\\d^*&0
\end{pmatrix}$.
However, such a triple on $\bt$ (and its associated Fredholm module) is even, with  standard grading
$\g=\begin{pmatrix}
-1&0\\0&1
\end{pmatrix}$,
so that, to maintain the correct pairing with odd K-theory, we add a further grading $\eps$,
\begin{equation}\label{epsilon}
\eps=-i\begin{pmatrix}
0&V\\V^*&0
\end{pmatrix},
\text{ where } Ve^{int}=\frac{\sgn(n)}{n}de^{int}=ie^{int}dt\,, t\in\bt.
\end{equation}
It is not difficult to show that $(\ca,\ck,D,\g,\eps)$ is a 1-graded spectral triple,  that is to say,  $(\cc(\bt),\ck,F,\g,\eps)$ is a 1-graded Fredholm module according to Definition \ref{dfn:pgrad} above, where $F$ is the phase of $D$.
Such 1-graded Fredholm module  is equivalent to the original odd one by Bott periodicity, cf. Proposition \ref{Ex8.8.4}. The triple $(\ca,\ck,D,\g,\eps)$ allows the required $\a$-deformation, which is described in the next Section.

\section{Spectral triples on quasi-circles}

In this section we build a family of spectral triples on the algebra $C(\T)$ of continuous functions on the circle $\mathbb{T}:=\br/\bz$, depending on a parameter $\a\in (0,1]$. For $\a=1$ we consider the  triple $(\ca,\ck,D,\g,\eps)$ given at the end of the preceding section, which describes the circle with the standard differential structure. The rest of the Section is devoted to the construction of the triples for $\a<1$, which may be considered as deformations of the original one, the circle being replaced by quasi-circles.

\subsection{Preliminaries about quadratic forms on $\T$}

We will use the following notation. For any $f\in C(\bt)$, the Fourier coefficients are $f_k := \frac1{2\pi}\int_\bt f(t)e^{-ikt}\,dt$, $k\in\bz$, the convolution between $f$ and $g\in C(\bt)$ is $f*g(t) := \frac1{2\pi} \int_\bt f(t-\th)g(\th)\, d\th$, and if $\Psi$ is a distribution on $\bt$ and $f\in C^\infty(\bt)$, the (sesquilinear) pairing $\langle \Psi, f \rangle$ is defined as the (weakly continuous) extension of the scalar product in $L^2(\ct)$.
For any positive sequence $\{a_k\}$ of polynomial growth on $\bz$ we consider the quadratic form on functions in $\cc^\infty(\T)$ given by
$$
Q[f]=\sum_{k\in\bz}a_k |f_k|^2,
$$
and the distribution $\Phi$  given by the Fourier series $\displaystyle\sum_{k\in\bz}a_ke^{ikt}$,  so that $\langle \Phi, f \rangle = \sum_{k\in\bz} a_kf_k$, and
$$
Q[f]=\langle \Phi, f^**f\rangle,
$$
where $f^*(t):=\overline{f}(-t)$.
%\end{lem}

\begin{dfn}
A sequence $\{a_k \in \bc: k\in \bz\}$ is called {\it positive definite} if
\begin{equation}\label{posdef}
\sum_{m,n\in\bz}a_{m-n}\overline{c}_mc_n\geq0
\end{equation}
 for any finitely supported sequence $\{c_k\}$.
 %Equivalently, its Fourier series $\Phi$ is a positive distribution.
A sequence $\{a_k\}$ is called {\it conditionally positive definite} if
\begin{equation}\label{condposdef}
\sum_{m,n\in\bz}a_{m-n}(\partial\overline{c})_m (\partial c)_n \geq0
\end{equation}
 for any finitely supported sequence $\{c_k \in \bc: k\in \bz\}$, where $(\partial c)_k=c_k-c_{k-1}$. A sequence is {\it (conditionally) negative definite}  if it is the opposite of a (conditionally) positive definite one.
%Equivalently, the pairing of its Fourier series $\Phi$ with a $C^\infty$ function with a zero of order $s$ in $t=0$ is described by a positive distribution.
\end{dfn}

\begin{thm}\label{thm:condposdef}
Let $\{a_k\}$ be a conditionally positive definite sequence. Then there exist a positive measure $\mu$ on $\T$ and a constant $b$ such that
$$
\langle\Phi,f\rangle=\int_\T(f(t)-f(0)-f'(0)\sin t)\ d \mu+ a_0 f(0)+\frac1{2i} (a_1-a_{-1})f'(0)+bf''(0).
$$
\end{thm}
\begin{proof}
The proof is analogous to that of Thm 1, Chapter II of \cite{GeVi4}, but we give the details for the convenience of the reader.

Passing to Fourier series, eq.~\eqref{condposdef}, which clearly holds also for fast decreasing sequences $c_k$, may be rephrased as
\begin{equation}\label{condposdef2}
\langle |1-e^{-it}|^{2}\Phi,|f|^2\rangle\geq0,
\end{equation}
where $f(t)=\sum_{k\in\bz}c_ke^{ikt}$. Since such sums describe all $\cc^\infty$ functions, and $|1-e^{-it}|^{2}=2(1-\cos t)$, this is equivalent to $\langle (1-\cos t)\Phi,g\rangle\geq0$ for any positive function $g\in\cc^\infty$, namely $(1-\cos t)\Phi$ is a positive measure $\n$. Equivalently,
 $\langle \Phi,(1-\cos t) g\rangle=\int g\, d\n$ for any  $g\in\cc^\infty$. Since any function $h$ with a zero of order $2$ may be written as $h=(1-\cos t) g$, we get
 $ \langle \Phi,h\rangle=\int h(t)(1-\cos t)^{-1}\, d\n$ for any function $h$ with a zero of order $2$ in $t=0$.
 We then separate the part of $\n$ with support in 0, setting $\n=b\d_0+(1-\cos t)\m$, thus getting
 $$
 \langle \Phi,h\rangle=\int_{(0,2\pi)} h(t)\, d\m+bh^{''}(0)\,.
 $$
Then, since $f(t)-f(0)-f'(0)\sin t$ has a zero of order $2$ for $t=0$, we get
\begin{align*}\label{condposformula}
\langle\Phi,f\rangle
&=\int_\T(f(t)-f(0)-f'(0)\sin t)\ d \mu+\langle\Phi,f(0)+f'(0)\sin t\rangle+bf''(0)\\
&=\int_\T(f(t)-f(0)-f'(0)\sin t)\ d \mu+ a_0 f(0)+\frac1{2i} (a_1-a_{-1})f'(0)+bf''(0).
\end{align*}
\end{proof}

\subsection{Sobolev norms and Clausen functions}
Let $s\in\bc$. Then the polylogarithm function of order $s$ is defined as
$$
\Li_s(z):= \sum_{k\in\bn} \frac{z^k}{k^s},\quad |z|<1.
$$
It has an analytic continuation on the whole complex plane with the line $[1,+\infty)$ removed, cf. the Appendix.
The  Clausen cosine function $\Ci_s(t)$ is defined as the sum of the Fourier series
$$
\sum_{k\in\bn}\frac{\cos kt}{k^s},\quad \re s> 1.
$$
When $\re s\leq 1$ it can be defined as the real part of $\Li_s(e^{it})$, hence it is a smooth function for $t\neq0$.

Some properties of the Clausen function are contained in Lemma \ref{Clausen-estimate} and Proposition \ref{Prop:pairing}.

\begin{prop}
Let  $\a\in(0,1)$, $a_k=|k|^{2\a}$, $k\in\bz$, and $\Phi$ the associated distribution as above. Then
\begin{itemize}
\item[$(i)$] the sequence $a_k$ is conditionally negative definite,
\item[$(ii)$] for any $\cc^\infty$ function $f$,
$$
\langle\Phi,f\rangle=\frac1\pi \int_\T\Ci_{-2\a}(t)(f(t)-f(0))\ dt\,.
$$
In particular, the Clausen function $\Ci_{-2\a}$ is negative.
\end{itemize}
\end{prop}

\begin{proof}
$(i)$ It is well known that $k^2$ is a conditionally negative definite sequence, therefore so is $k^{2\a}$, for $\a\in(0,1]$ (\cite{BCR}, page 78).

$(ii)$  Assume $f(0)=0$. Since $\Phi$ is even, the pairing with the odd part of $f$ vanishes, while, by Proposition \ref{Prop:pairing}, the pairing with the even part is given by the integral against $\frac1\pi \Ci_{-2\a}$. According to  the results of  Theorem \ref{thm:condposdef}, the measure $d\m$ (which is now negative) should be replaced by $\frac1\pi \Ci_{-2\a}(t)\,dt$, showing in particular that $\Ci_{-2\a}$ is negative. For a general $f$, again using the parity of $\Phi$, the pairing becomes
$$
\langle\Phi,f\rangle=\frac1\pi \int_\T\Ci_{-2\a}(t)(f(t)-f(0))\ dt+b f^{''}(0)\,,
$$
hence we get the result if we show that $b=0$. By definition, for any continuous function $g$, $\langle\Phi, (1-\cos t)g(t)\rangle=b\, g(0)+\int (1-\cos t)g(t)\,d\m$. In particular, if $g$ has suitably small support,
$b\, g(0)=\lim_{\eps\to0}\langle\Phi, (1-\cos t)g(t/\eps )\rangle$. Choosing $g(t)=\chi_{[-1,1]}(1-|t|)$, a direct computation shows that $b=0$.
\end{proof}

\begin{cor}\label{ClausenSobolev}
Let  $\a\in(0,1]$, $a_k=|k|^{2\a}$, $k\in\bz$,  and denote by $\ce_\a$ the corresponding quadratic form. Then
\begin{itemize}
\item[$(i)$] $\|f\|^2_2+\ce_\a[f]$ is the square of the norm for the Sobolev space $H^\a(\T)$,
\item[$(ii)$] the quadratic form $\ce_\a$ is given by
$$
\ce_\a[f]=- \frac1{2\pi} \int_{\T\times\T}\Ci_{-2\a}(x-y)|f(x)-f(y)|^2\ dxdy-b\|f'\|^2_2,
$$
where $b=0$ when $\a<1$, while  $\Ci_{-2}=0$ and $b=-1$ when $\a=1$.
\end{itemize}
\end{cor}
\begin{proof}
$(i)$ Obvious.
\\
$(ii)$ We have
$$
\ce_\a[f]=\langle\Phi,f^**f\rangle= \frac1{\pi}\int_\T\Ci_{-2\a}(t)((f^**f)(t)-(f^**f)(0))\ dt+b (f^**f)^{''}(0).
$$
Since $\Ci_{-2\a}(x-y)=\Ci_{-2\a}(y-x)$, we have
\begin{align*}
2\int_\T\Ci_{-2\a}(t)&((f^**f)(t)-(f^**f)(0))\ dt\\
&=2\int_{\T\times\T}\Ci_{-2\a}(t)\big(\ov{f}(x-t)f(x)-\ov{f}(x)f(x)\big)\,dt\,dx\\
&=2\int_{\T\times\T}\Ci_{-2\a}(x-y)\big(\ov{f}(y)f(x)-\ov{f}(x)f(x)\big)\,dy\,dx\\
&=\int_{\T\times\T}\Ci_{-2\a}(x-y)\big(\ov{f}(y)f(x)-\ov{f}(x)f(x)+\ov{f}(x)f(y)-\ov{f}(y)f(y)\big)\,dy\,dx\\
&=-\int_{\T\times\T}\Ci_{-2\a}(x-y)\big|f(x)-f(y)\big|^2\,dy\,dx\,.
\end{align*}
As for the second summand,
$$
(f^**f)^{''}(0)=-\big((f')^**f'\big)(0)=-\|f'\|^2_2\, ,
$$
which proves the equation. Since the quadratic form gives rise to the Sobolev norm, the last summand should vanish when $\a<1$. For $\a=1$, the Clausen function vanishes by a direct computation, and $\ce_\a[f]=\|f'\|^2_2$, giving $b=-1$.
\end{proof}

\subsection{The construction of the triple}\label{alphatripleonS1}

Let us consider, for each fixed $0<\alpha\leq1$, the Dirichlet form $\ce_\alpha$ on $L^2 (\T)$,
with  domain $\cf_\alpha :=\{f\in L^2 (\T):\ce_\alpha [f]<+\infty \}$.
\par\noindent
As shown in Corollary \ref{ClausenSobolev}, the Sobolev space $H^\alpha (\T )$ coincides with $\cf_\alpha$ and has norm
\[
\|f\|_\alpha^2 =\|f\|^2_{L^2 (\T)} + \ce_\alpha [f]\, .
\]
We summarize below the main known properties of the Dirichlet spaces on the circle we are considering.
Proofs may be found in \cite{FOT}.
\begin{prop}
The Dirichlet space $(\ce_\alpha ,\cf_\alpha)$ on $L^2 (\T)$ is regular in the sense that the Dirichlet algebra $\cf_\alpha \cap C(\T)$ is dense both in $C(\T)$ with respect to the uniform norm and in $\cf_\alpha$ with respect to the graph norm. In particular, the algebra $C^\gamma (\T)$ of H\"{o}lder continuous functions of order $\gamma \in(\a,1]$ is a form core contained in the Dirichlet algebra.
We observe that $\cf_\a \subset C(\T)$,  for $\a>\frac12$.
\end{prop}
We now construct a Spectral Triple associated to the above Dirichlet space for each value of the parameter $0<\alpha< 1$, the case $\a=1$ having been described above. The construction is  given in terms of a closable derivation with values in a suitable bimodule, underlying any regular Dirichlet form (see  \cite{CiSa},  \cite{Cip}).
%
%When $\a=1$ we set $\partial_\a f=df$, and we get $\ce_\a[f]=\|\partial_\a f\|^2_{L^2(\O^1(\T))}$.
%
%Let us now choose $\a<1$.
%
By Corollary \ref{ClausenSobolev},
\[
\ce_\alpha [f] = \frac1{4\pi^2}\int_\T \int_\T \f_\a(z-w) |f(z)-f(w)|^2\, dzdw,
\]
where we set $\f_\a=-2\pi\Ci_{-2\a}$.

The linear map defined as
\[
\partial_\alpha :\cf_\alpha\rightarrow L^2 (\T\times\T)\qquad \partial_\alpha (f)(z,w)=
\f_\a(z-w)^{1/2}(f(z)-f(w))
\]
is a closed operator acting on $L^2(\T)$, since $\ce_\alpha [f]=\|\partial_\alpha f\|^2_{L^2 (\T\times\T )}$.

Endowing the Hilbert space $L^2 (\T\times\T)$ with the $C(\T)$-bimodule structure defined by the left and right actions of $C(\T)$ given by
\[
(f\xi)(z,w):=f(z)\xi (z,w)\, ,\qquad (\xi g)(z,w):=\xi (z,w) g(w)\,, \qquad z,w\in \T\,,
\]
and by the anti-linear involution
\[
(\mathcal{J}\xi ) (z,w):=\overline{\xi (w,z)},\qquad z,w\in \T\, ,
\]
for $f,g\in C(\T)$ and $\xi\in L^2 (\T\times\T )$, it is easy to see that the map $\de_\a$ is a {\it derivation} on the Dirichlet algebra $\cf_\a \cap C(\T)$,
since it is {\it symmetric}
\[
\mathcal{J}(\partial_\alpha (f))=\partial_\alpha (\overline{f}),\qquad f\in C^\gamma (\T),
\]
and satisfies the {\it Leibniz rule}
\[
\partial_\alpha (fg)=(\partial_\alpha f)g+f(\partial_\alpha g),\qquad f,g\in C^\gamma (\T)\, .
\]
Moreover, the map $\de_\a$ is a {\it differential square root} of the self-adjoint operator $\Delta^\alpha$ on $L^2 (\T)$ having $(\ce_\alpha , \cf_\alpha )$ as closed quadratic form, because of the identities
\[
\ce_\alpha [f]:=\|\Delta^{\alpha/2} f\|^2_{L^2 (\T)} =
\|\partial_\alpha f\|^2_{L^2 (\T\times\T )},\qquad f\in\cf_\alpha\, .
\]
We accommodate in the following Lemma some technical results which will be useful later.

\begin{lem}\label{lem:technical}
Let us denote by $\{e_k :k\in\mathbb{Z}\}$ the orthonormal basis of eigenfunctions of the standard Laplacian $\Delta$:
\[
e_k (t):= e^{ikt},\qquad \Delta e_k = k^2e_k.
\]
\begin{enumerate}
\item $\ce_\alpha (e_k ,e_j)=(e_k ,\partial_\a^*\partial_\a e_j)=|k|^{2\a}\d_{kj}$.
\item Let $e'_n=|n|^{-\a}\partial_\a e_n$, $n\in\bz\setminus\set{0}$. Then, the family $\{e'_n\}_{n\in\bz\setminus\set{0}}$ is an orthonormal basis for the range of $\partial_\a$.
 \item The following equation holds:
\begin{equation}\label{d*(de_p)e_n}
\partial_\a^* ((\partial_\a e_p)e_n)=\frac12(|p|^{2\a }+|n+p|^{2\a }-|n|^{2\a })e_{n+p}.
\end{equation}
%\item For any $k,p\in\bz$, $|p|^{2\a }+|k|^{2\a }-|k-p|^{2\a }\leq 2 |k|^{\a}|p|^{\a}$.
\item For any $k,p\in\bz$, $(|p|^{2\a }+|k|^{2\a }-|k-p|^{2\a })^2\leq 4 |k|^{2\a}|p|^{2\a}$.
\item Let us consider the map $S_f :C(\T)\rightarrow L^2 (\T\times\T )$ defined, for a fixed $f\in C(\T)$, as $S_f g:=(\partial_\alpha f)g$, with $g\in C(\T)$. Then, for $s>\a^{-1}$ and $f\in H^\a$, the operators $(\partial_\a\partial_\a^*)^{-s/4}S_{f}S^*_{f}(\partial_\a\partial_\a^*)^{-s/4}$ and $(\partial_\a^*\partial_\a)^{-s/4}S^*_{f}S_{f}(\partial_\a^*\partial_\a)^{-s/4}$ are trace class, and
\begin{align}
&\tr((\partial_\a\partial_\a^*)^{-s/4}S_{f}S^*_{f}(\partial_\a\partial_\a^*)^{-s/4})
\leq  2\z(\a s)\ce_\a[f]
=\tr((\partial_\a^*\partial_\a)^{-s/4}S^*_{f}S_{f}(\partial_\a^*\partial_\a)^{-s/4}).
\label{lem:technical-V}\\
&\Res_{s=1/\a}\tr((\partial_\a\partial_\a^*)^{-s/4}S_{f}S^*_{f}(\partial_\a\partial_\a^*)^{-s/4})=
\begin{cases}0 &\mathrm{ if }\ \a<1,\\
2\ce_\a[f] &\mathrm{ if }\ \a=1.
\end{cases} \label{Res1} \\
&\Res_{s=1/\a}\tr((\partial_\a^*\partial_\a)^{-s/4}S^*_{f}S_{f}(\partial_\a^*\partial_\a)^{-s/4})=\frac2\a\ce_\a[f].\label{Res2}
\end{align}
\end{enumerate}
\end{lem}
\begin{proof}
The equality $\partial_\a^*\partial_\a=\Delta^\a$ gives $(1)$, while
$(2)$ follows from a direct computation, and $(3)$ amounts to verify that $(\partial_\a e_k,(\partial_\a e_p)e_n)=\frac12(p^{2\a }+(n+p)^{2\a }-n^{2\a })\d_{k,n+p}$.
We now show $(4)$.
We observe that it certainly holds for $p=0$ or $k=0$. When they do not vanish, we must prove that
\begin{equation}
-1\leq\frac {|p|^{2\a }+|k|^{2\a }-|k-p|^{2\a }}{2 |k|^{\a}|p|^{\a}}\leq 1,
\end{equation}
or, setting $|p/k| = e^{2t}$, where we may assume $t\geq0$,
\begin{equation}
-1\leq \frac12\left(e^t\mp e^{-t}\right)^{2\a}-\cosh(2\a t)\leq 1,
\end{equation}
the $\pm$ sign being the sign of $pk$. Taking the worst cases, we get
\begin{equation}
-1\leq \frac12\left(e^t- e^{-t}\right)^{2\a}-\cosh(2\a t),
\quad \frac12\left(e^t+ e^{-t}\right)^{2\a}-\cosh(2\a t)\leq 1,
\end{equation}
or, equivalently, $2\sinh(\a t)\leq\left(2\sinh t\right)^{\a}$ and
$\left(2\cosh t\right)^{\a}\leq 2\cosh(\a t)$. Passing to the logarithms, it is enough to prove that
$f_\a(t):=\log(2\sinh(\a t))-\a\log \left(2\sinh t\right)\leq0 $ and
$ g_\a(t):=\a\log\left(2\cosh t\right)-\log(2\cosh(\a t))\leq 0$.
This follows because both functions tend to 0 for $t\to+\infty$, and
$f_\a'(t) =\a(\coth(\a t)-\coth t)  \geq 0$ for $\a\in[0,1]$ since $\coth$ is decreasing, and $g_\a'(t) = \a(\tanh t-\tanh(\a t))  \geq 0$ for $\a\in[0,1]$ since $\tanh$ is increasing.
\\
As for the inequality in \eqref{lem:technical-V}, we have
\begin{equation}\label{estimate1}
\begin{aligned}
 \|S^*_{f}\partial_\a e_k\|^2
&=\sum_{n}|(e_n,S^*_{f}\partial_\a e_k)|^2
=\sum_{n}  |(\partial_\a f,((\partial_\a e_k)e_{-n})|^2\\
&=\frac14\sum_{n}  (|k|^{2\a }+|n-k|^{2\a }-|n|^{2\a })^2  |( f,e_{n+k})|^2\\
&=\frac14\sum_{p}  (|k|^{2\a }+|p|^{2\a }-|p-k|^{2\a })^2 |( f,e_{p})|^2\\
&\leq |k|^{2\a}\sum_{p} |p|^{2\a }|( f,e_{p})|^2\
= |k|^{2\a} \ce_\a[ f]\, ,
\end{aligned}
\end{equation}
where the inequality in the last row follows by $(4)$. Then
\begin{equation} \label{estimate2}
\begin{aligned}
\tr&((\partial_\a\partial_\a^*)^{-s/4}S_{f}S^*_{f}(\partial_\a\partial_\a^*)^{-s/4})
=\sum_k((\partial_\a\partial_\a^*)^{-s/4}e'_k,S_{f}S^*_{f}(\partial_\a\partial_\a^*)^{-s/4}e'_k)\\
&=\sum_k |k|^{-(s+2)\a} \|S^*_{f}\partial_\a e_k\|^2
\leq\sum_{k} |k|^{-s\a} \|\partial_\a f\|^2_{L^2(\T\times\T)}
=  2\z(\a s) \ce_\a[ f]\, .
\end{aligned}
\end{equation}
Concerning the equality in \eqref{lem:technical-V} we have
\begin{equation}\label{formula3}
\begin{aligned}
\tr((\partial_\a^*\partial_\a)&^{-s/4}S^*_{f}S_{f}(\partial_\a^*\partial_\a)^{-s/4})
=\sum_k((\partial_\a^*\partial_\a)^{-s/4}e_k,S^*_{f}S_{f}(\partial_\a^*\partial_\a)^{-s/4}e_k)\\
&=\sum_k |k|^{-s\a} \|S_{f}e_k\|^2
=\sum_{k} |k|^{-s\a}\|(\partial_\a f)e_k)\|^2
=  2\z(\a s) \ce_\a[ f]\, .
\end{aligned}
\end{equation}
Eq. \eqref{Res1} for $\a=1$ is straightforward, we now prove it  for $\a<1$. By \eqref{estimate1},
$$
|k|^{-2\a}\|S^*_{f}\partial_\a e_k\|^2=
\sum_{p}  \left(\frac{|k|^{2\a }+|p|^{2\a }-|p-k|^{2\a }}{2 |k|^{\a}|p|^{\a}}\right)^2
(|p|^\a|( f,e_{p})|)^2,
$$
where the first factor in the series is bounded by 1, and the second is in $\ell^1(\bz)$. By dominated convergence,
$$
\lim_{|k|\to\infty}|k|^{-2\a}\|S^*_{f}\partial_\a e_k\|^2=
\sum_{p}  \lim_{|k|\to\infty}\left(\frac{|k|^{2\a }+|p|^{2\a }-|p-k|^{2\a }}{2 |k|^{\a}|p|^{\a}}\right)^2
(|p|^\a|( f,e_{p})|)^2=0.
$$
Formula \eqref{estimate2} and the asymptotic character of the residue prove the thesis. Finally, \eqref{Res2} follows directly by \eqref{formula3}.
\end{proof}

We now come to the promised family of spectral triples.
As mentioned above, we consider a deformation of the standard $L^2$ De Rham complex $(L^2(\O^*(\T)),\partial)$ on $\T$, where $L^2(\O^0(\T))$ resp. $L^2(\O^1(\T))$ denotes the $L^2$ functions, resp. $L^2$ 1-forms on $\T$, and $\partial$ is the (densely defined) external derivation.
Our deformation will be the $L^2$ complex $(L^2(\O_\a^*(\T)),\partial_\a)$ on $\T$, where $L^2(\O_\a^0(\T)):=L^2(\O^0(\T))$,  $L^2(\O^1(\T)):=L^2(\T\times\T)$, and the
deformed external derivation $\partial_\a$ has been defined above.
The triple $(\mathcal{A}_\alpha, \mathcal{K}_\alpha, D_\alpha)$ is then given by
the Hilbert space $\ck_\a:=L^2(\O_\a^*(\T))=L^2(\O^1_\a(\T))\oplus L^2(\O^0(\T))$, the Dirac operator $(D_\alpha,{\rm dom}(D_\alpha))$ on  $\mathcal{K}_\alpha$ is defined as
\begin{equation}\label{Dalpha}
D_\alpha:=\left(
     \begin{array}{cc}
       0 & \partial_\alpha \\
       \partial_\alpha^* & 0 \\
     \end{array}
   \right),
   \quad
   \text{so that}
   \quad
D_\alpha\left(
   \begin{array}{c}
     \xi \\
     f \\
   \end{array}
 \right)
 =\left(
   \begin{array}{c}
     \partial_\alpha f \\
     \partial_\alpha^* \xi \\
   \end{array}
 \right),
\end{equation}
on the domain ${\rm dom}(D_\alpha):={\rm dom}(\partial_\alpha^*)\oplus \cf_\alpha$, and the $*$-algebra $\ca_\a$ is defined as $\ca_\a=\{f\in C(\T):\|[D,L_f]\|<\infty\}$, where, if $f\in C(\T)$, $L_f$ denotes its left action on $\mathcal{K}_\alpha$ resulting from the direct sum of those on $L^2 (\T\times\T )$ and on $L^2 (\T)$.
We also consider the seminorm $p_\a$ given by
\begin{equation}\label{p-alphaseminorm}
\pa(f)^2= \frac1{2\pi}\, \sup_{x\in\T}\int_\T \,\f_{\a}(x-y)|f(x)-f(y)|^2 dy.
\end{equation}
\begin{prop}\label{smoothAlg}
For $\a\in(0,1)$ the algebra $\mathcal{A}_\alpha$ defined above coincides with
$\{f\in C(\T):\pa(f)<\infty\}$, and  $C^{0,\a+\eps}(\T) \subset \mathcal{A}_\alpha$, hence it is a uniformly dense subalgebra of $C(\T)$. Moreover, for $\a \geq \frac12$,
$\mathcal{A}_\alpha \subset C^{0,\a}(\T)$. Analogous results hold true upon replacing the left module structure of $\mathcal{K}_\alpha$ by the right one.
\end{prop}
\begin{proof}
Let us consider first the map $S_f :C(\T)\rightarrow L^2 (\T\times\T )$ defined as $S_f g=(\partial_\alpha f)g$ for a fixed $f\in C(\T)$.
This map extends to a bounded map on $L^2 (\T)$, provided $f\in \mathcal{A_\alpha}$, because
\begin{eqnarray}
\nonumber  \|S_f g\|^2_{L^2 (\T\times\T )} &=& \frac1{4\pi^2}\int_{\T\times\T } |(\partial_\alpha f)(z,w)g(w)|^2 \, dzdw \\
\nonumber   &=& \frac1{4\pi^2} \int_\T  |g(w)|^2 \int_\T  \f_\a(z-w)|f(z)-f(w)|^2  \, dzdw\\
\nonumber   &\leq & \frac1{2\pi} \|g\|^2_{L^2 (\T)} \sup_{w\in\T} \int_\T \f_\a(z-w)|f(z)-f(w)|^2\, dz,\qquad g\in L^2(\T),
\end{eqnarray}
so that
\begin{equation}\label{norm-equality}
\|S_f \|^2 = \frac1{2\pi}\, \sup_{w\in\T} \int_\T  \f_\a(z-w)|f(z)-f(w)|^2\, dz =\pa(f)^2\, .
\end{equation}
Let us compute now, using the Leibniz rule for the derivation $\partial_\alpha$, the quadratic form of the commutator $[D_\alpha ,L_f]$, defined on the domain ${\rm dom}(D_\alpha):={\rm dom}(\partial_\alpha^*)\oplus \cf_\alpha$:
\begin{eqnarray}
\nonumber \bigl(\xi '\oplus g'|[D_\alpha ,L_f]\xi\oplus g\bigr) &=& \bigl(D_\alpha (\xi '\oplus g')|L_f (\xi\oplus g)\bigr)- \bigl(L_{f^*}(\xi '\oplus g' )|D_\alpha (\xi\oplus g)\bigr) \\
\nonumber    &=& \bigl(\partial_\alpha g'\oplus \partial_\alpha^*\xi '|f\xi\oplus fg\bigr)- \bigl(f^*\xi '\oplus f^*g'|\partial_\alpha g\oplus \partial_\alpha^*\xi\bigr) \\
\nonumber    &=& (\partial_\alpha g'|f\xi) + (\partial_\alpha^* \xi '|fg) - (\xi '|f\partial_\alpha g) - (f^*g'|\partial_\alpha^*\xi ) \\
\nonumber    &=& (f^*\partial_\alpha g'|\xi) + (\xi '|\partial_\alpha (fg)) - (\xi '|f\partial_\alpha g) - (\partial_\alpha (f^*g' )|\xi ) \\
\nonumber    &=& (\xi '|(\partial_\alpha f)g) - ((\partial_\alpha f^* )g'|\xi )\, . \\
\nonumber    &=& (\xi '|S_f g) - (S_{f^*} g'|\xi )\, .
\nonumber
\end{eqnarray}
Hence
\begin{equation}\label{commutator}
[D_\alpha ,L_f] = \left(
                    \begin{array}{cc}
                      0 & S_f \\
                      -S^*_{f^*} & 0 \\
                    \end{array}
                  \right)\, ,\qquad a\in \mathcal{A}_\alpha,
\end{equation}
therefore $[D_\alpha ,L_f]$ extends to a bounded operator on $\mathcal{K}_\alpha$ if and only if $p_\alpha(f)<\infty$, and $\| [D_\alpha ,L_f]\|=\|S_f\|=\pa(f)$. The relations w.r.t. the spaces of H\"{o}lder continuous functions follow by  Proposition  \ref{Prop:Holder}.
\end{proof}

\begin{thm}\label{tripleOnS1}
%Let us consider the regular Dirichlet form $(\ce_\alpha ,\cf_\alpha)$ on $L^2(\T)$, the associated derivation $(\partial_\alpha ,\cf_\alpha)$ and the triple $(\ca_\a, \ck_\a, D_\a)$ described above. Then,
Let $\a\in(0,1]$. The triple $(\mathcal{A}_\alpha, \mathcal{K}_\alpha, D_\alpha)$ described above is a densely defined Spectral Triple on the algebra $C(\T)$, in the sense of  Connes. In particular,
\begin{enumerate}
\item[$(i)$] $D_\alpha^{-1}$ has compact resolvent, and
the function $\zeta_D(s)=\tr(|D_\a|^{-s})=4\zeta(\a s)$, where by $|D_\a|^{-s}$ we mean the functional calculus with the function
$\begin{cases}t^{-s} & t>0,\\
0  & t=0.
\end{cases} $
\item[$(ii)$] The   dimension of the triple is $\a^{-1}$, and
$\Res_{s=1/\a}\tr(f|D_\a|^{-s})=\frac4\a \int f(t)\ dt$.
\item[$(iii)$] The distance $d_D$ induced on $\T$ by the spectral triple satifies, for any $\eps>0$, $d_D(x,y) \geq \frac1{c_\eps} |x-y|^{\a+\eps}$, $x,y\in\T$. Moreover, if $\a\geq\frac12$, $d_D(x,y) \leq \frac1{\tilde{c}_\a} |x-y|^{\a}$, $x,y\in\T$. Here, $c_\eps$ and $\tilde{c}_\a$ are as in Proposition \ref{Prop:Holder}.
\item[$(iv)$] The Dirichlet form $\ce_\a$ can be recovered, for any $f\in H^\a(\bt)$, via the formulas
\begin{equation*}
\ce_\a[f]
%=\tr_\o\left(|D|^{-1/2}|[D,f]|^2|D|^{-1/2}\right)
=\frac2\a\lim_{s\to 1} (s-1) \tr (|D|^{s/2}|[D,f]|^2|D|^{s/2})
\end{equation*}
\end{enumerate}
\end{thm}
\begin{proof}
$(i)$ Notice first that, since the self-adjoint operators $\partial_\alpha^* \partial_\alpha$ and $\partial_\alpha \partial_\alpha^*$ are unitarily equivalent on the orthogonal complement of their kernels, it suffices to prove that $\partial_\alpha^* \partial_\alpha$ has discrete spectrum on $L^2 (\T)$. Indeed,  Lemma \ref{lem:technical} $(1)$ shows that the spectrum of the self-adjoint operator $\partial_\alpha^* \partial_\alpha$ is discrete and coincides with $\{k^{2\a} : k\in\mathbb{N}\}$. Since any non-zero eigenvalue of $D_\a$ has multiplicity 4, we get the formula for $\zeta_D$.
\item[$(ii)$] Follows by $(i)$ and a straightforward computation.
\item[$(iii)$]  	Observe that, using the notation of Proposition \ref{Prop:Holder},
	\begin{align*}
		d_D(x,y) & = \sup \set{ |f(x)-f(y)| : \norm{[D_\a,f]}  \leq 1} = \sup \set{ |f(x)-f(y)| : p_\a(f)  \leq 1} \\
		& \geq \frac1{c_\eps} \sup \set{ |f(x)-f(y)| : \norm{f}_{  C^{0,\a+\eps}(\T) } \leq 1 } = \frac1{c_\eps} |x-y|^{\a+\eps},
	\end{align*}
	and, if $\a\geq \frac12$,
	\begin{align*}
		d_D(x,y) &  = \sup \set{ |f(x)-f(y)| : p_\a(f)  \leq 1} \\
		& \leq \frac1{\tilde{c}_\a} \sup \set{ |f(x)-f(y)| : \norm{f}_{  C^{0,\a}(\T) } \leq 1 } = \frac1{\tilde{c}_\a} |x-y|^{\a}.
	\end{align*}
\item[$(iv)$] Follows by \eqref{commutator} and Lemma \ref{lem:technical} (5).

\end{proof}

%\subsection{The pairing with $K$-theory}
%
%In this section we show that the phase of the Dirac operator gives rises to a Fredholm module that pairs nontrivially with respect to K-theory on the circle. As a matter of facts, up to now we have an even spectral triple, which would give rise to an even Fredholm module over $C(\T)$, whereas, in order to detect the non-trivial topology of $\T$, we need an odd module. There is a well-known device to obtain an odd module from an even one, and consists in adding a further grading to the module, which results in a $1$-graded module, and then in making use of Bott periodicity to reduce the latter to an odd one. We mimick this procedure in the following Proposition.

As explained above, the pairing with odd K-theory is recovered by a 1-graded Fredholm module  $ (\ck_\a,F_\a,\g,\eps_\a)$, where $F_\a$ is the phase of $D_\a$, and the further grading $\eps_\a$ is a simple deformation of the grading $\eps$ described in formula \eqref{epsilon}, namely $\eps_\a=-i \begin{pmatrix} 0 & V_\a \\ V_\a^* & 0 \end{pmatrix}$, where
$V_\a$ is  given by $V_\a e_j=\sgn(j)e'_j=\sgn(j)|j|^{-\a}\partial_\a e_j$, $j\ne0$, $V_\a e_0=0$.

\begin{prop}\label{FrModS1}
Let $\a\in(0,1]$. The quadruple $\cf_\a= (\ck_\a,F_\a,\g,\eps_\a)$ is a tamely degenerate
 $1$-graded Fredholm module on $ C(\bt)$. The module $\cf_\a^+$ constructed as in
 Proposition \ref{Ex8.8.4}, is an   odd Fredholm module on $ C(\bt)$, and
 the pairing with K-Theory gives $\langle\, \cf_\a^+ ,e_k\rangle=k$.
\end{prop}
\begin{proof}
For notational simplicity,  we drop the index $\a$.
Let us observe that $F=\begin{pmatrix}0&W\\ W^*&0\end{pmatrix}$, where $W$ is the partial isometry given by $We_i=e'_i$, $i\ne0$, $We_0=0$. As a consequence, setting $Se_j=\sgn(j)e_j$, we get $V=WS$, hence $i\eps=(I\oplus S)F(I\oplus S)$. A direct computation  shows
$P_0=[\ker(F)]=[\ker(F^2)]=1-F^2= Q_0\oplus (1- S^2)$, where $Q_0$ is the projection on $\ker(\partial_\a^*)$. Then the support of $(I\oplus S)F(I\oplus S)$ coincides with the support of $F$, which is $1-P_0$. Therefore $-\eps^*\eps=\big( (I\oplus S)F(I\oplus S)\big)^2=F^2=1-P_0$. We then compute
\begin{equation}\label{Ftilde}
i\eps F=\begin{pmatrix}WSW^*&0\\ 0&S\end{pmatrix}=iF\eps\,,
\end{equation}
hence $[\eps,F]=0$. Moreover, $\eps$ is clearly skew-adjoint. We now prove the compactness of $[\eps,\pi(f)]$. Indeed,
\begin{align*}
[i\eps,\pi(f)]
=&[(I\oplus S)F(I\oplus S),\pi(f)]\\
=&[(I\oplus S),\pi(f)]F(I\oplus S)
+(I\oplus S)[F,\pi(f)](I\oplus S)
+(I\oplus S)F[(I\oplus S),\pi(f)]\,.
\end{align*}
The compactness of $[F,\pi(f)]$ follows by the spectral triple properties (cf. Proposition \ref{triple-module}), and the compactness of $[(I\oplus S),\pi(f)]=0\oplus[S,f]$ follows by the Toeplitz theory.
Therefore $\cf$ is a kernel-degenerate $1$-graded Fredholm module, and so, by Proposition \ref{Ex8.8.4}, $\cf^+$ is a kernel-degenerate ungraded Fredholm module. Moreover, $\cf^+$ is indeed non-degenerate, because $F^+ = V^*W = S$ has a one-dimensional kernel, hence $(F^+ )^2-I$ is compact.
\end{proof}

\section{Spectral triples on the Sierpinski gasket}

\subsection{Sierpinski Gasket and its Dirichlet form}

We denote by $K$ the Sierpi\'nski gasket. Introduced in \cite{Sierpinski} as a curve with a dense set of ramification points, it has been the object of various investigations in Analysis \cite{Kiga}, Probability \cite{Ku,Bar} and Theoretical Physics \cite{RT}.
\par\noindent
Let $p_0, p_1,p_2\in \mathbb{R}^2$ be the vertices of an equilateral triangle of unit length and consider the contractions $w_i :\mathbb{R}^2\rightarrow\mathbb{R}^2$ of the plane: $w_i(x):= p_i+\frac12(x-p_i)\in\br^2$. Then $K$ is the unique fixed-point w.r.t. the contraction map $E \mapsto w_0(E)\cup w_1(E)\cup w_2(E)$ in the set of all compact subsets of $\br^2$, endowed with the Hausdorff metric. Two ways of approximating $K$ are shown in Figures \ref{Fig1} and \ref{Fig2}.
\vskip0.2truecm\noindent
Let us denote by $\Sigma_m :=\{0,1,2\}^{m}$ the set of words of length $m\ge 0$ composed by $m$ letters chosen in the alphabet of three letters $\{0,1,2\}$ and by $\Sigma :=\bigcup_{m\ge 0}\Sigma_m$ the whole vocabulary (by definition $\Sigma_0 :=\{\emptyset\}$). A word $\sigma\in \Sigma_m$ has, by definition, length $m$ and this is denoted by $|\sigma |:=m$. For $\sigma =\s_1 \s_2\dots \s_m\in\Sigma_m$ let us denote by $w_\sigma$ the contraction $w_\sigma :=w_{\s_1}\circ w_{\s_2}\circ\dots \circ w_{\s_m}$.
\vskip0.2truecm\noindent
Let $V_0 :=\{p_0 ,p_1 ,p_2\}$ be the set of vertices of the equilateral triangle and $E_0 :=\{e_{0}, e_{1}, e_{2}\}$ the set of its edges, with $e_i$ opposite to $p_i$. Then, for any $m\ge 1$,  $V_m :=\bigcup_{|\sigma |=m} w_\sigma (V_0)$ is the set of vertices of a finite graph ($i.e.$ a one-dimensional simplex) denoted by $(V_m,E_m)$ whose edges are given by $E_m :=\bigcup_{|\sigma |=m} w_\sigma (E_0)$ (see Figure 2). The self-similar set $K$ can be reconstructed also as an Hausdorff limit either of the increasing sequence $V_m$ of vertices or of the increasing sequence $E_m$ of edges, of the above finite graphs. Set $V_* := \cup_{m=0}^\infty V_m$, and $E_* := \cup_{m=0}^\infty E_m$.
% figura
     \begin{figure}[ht]
 	 \centering
	 \psfig{file=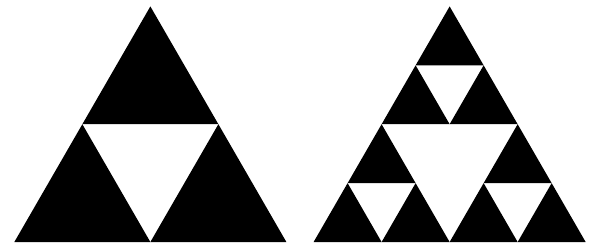,height=1.2in}
	 \psfig{file=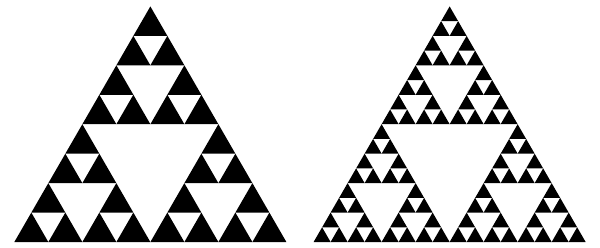,height=1.2in}
	 \caption{Approximations from above of the Sierpinski gasket.}
	 \label{Fig1}
     \end{figure}
     \begin{figure}[ht]
 	 \centering
	 \psfig{file=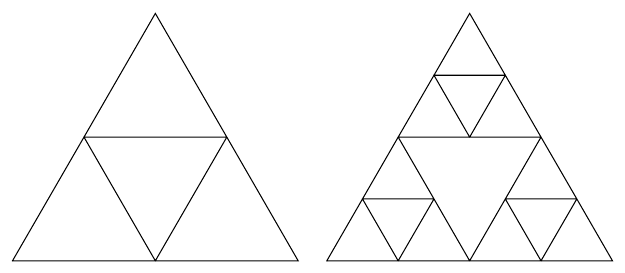,height=1.2in}
	 \psfig{file=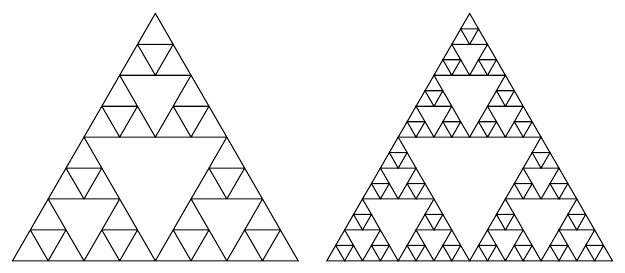,height=1.2in}
	 \caption{Approximations from below of the Sierpinski gasket.}
	 \label{Fig2}
     \end{figure}

In the present work a central role is played by the quadratic form $\ce : C(K)\to [0,+\infty ]$ given by
$$
\ce[f]=\lim_{m\to\infty} \left(\frac53\right)^m\sum_{e\in E_m}|f(e_+)-f(e_-)|^2,
$$
where each edge $e$ has been arbitrarily oriented, and $e_-,e_+$ denote its source and target. It is a regular Dirichlet form since it is lower semicontinuous, densely defined on the subspace  $\cf:= \set{f\in C(K) : \ce[f]<\infty}$ and satisfies the {\it Markovianity property}
\begin{equation}\label{contraction}
\ce [f\wedge 1]\leq \ce [f]\qquad f\in C(K;\br).
%\footnote{Here and in the following, we will denote by $C(K)$ the space of  real valued continuous functions. As a consequence, the quadratic Dirichlet form $\ce$ will be considered as a symmetric bilinear form over $\cf$.} .
\end{equation}

The existence of the limit above and the mentioned properties are consequences of the theory of {\it harmonic structures} on self-similar sets developed by  Kigami \cite{Kiga}.

As a result of the theory of Dirichlet forms \cite{BD,FOT}, the domain $\cf$ is an involutive subalgebra of $C(K)$ and, for any fixed  $f,g\in \cf$, the functional
\begin{equation}\label{rem:cs}
\cf\ni h\mapsto   \Gamma (f,g)(h):=\frac12\big(  \ce(f,hg)-\ce(fg,h)+\ce(g,fh) \big) \in \br
\end{equation}
defines a finite Radon measure called the {\it energy measure} (or {\it carr\'e du champ}) {\it of $f$ and $g$}. In particular, for $f\in\cf$, the measure $\Gamma (f,f)$ is nonnegative and one has the representation
\[
\ce [f]=\int_K 1\, d\Gamma (f,f) =\Gamma (f,f)(K)\qquad f\in\cf\, .
\]
In applications, $f$ may represent a configuration of a system, $\ce [f]$ its corresponding total energy and $\Gamma (f,f)$ represents its distribution. In homological terms,  $\Gamma$ is (up to the constant $1/2$) the Hochschild co-boundary of the 1-cocycle $\phi (f_0,f_1):=\ce (f_{0} ,f_1)$ on the algebra $\cf$.

\medskip

The Dirichlet or energy form $\ce$ should be considered as a Dirichlet integral on the gasket.
It is closable with respect to any Borel regular probability measure on $K$ which is positive on open sets and vanishes on finite sets (see \cite{Kiga} Theorem 3.4.6 and \cite{Kiga3} Theorem 2.6).
%
%
%It is lower semicontinuous on the space $L^2 (K,m)$, finite on the subspace  $\cf$, with respect to a wide range of positive Borel measures on $K$ and,
%
Once the measure $m$ has been chosen, $\ce$ is the quadratic form of a positive, self-adjoint operator on $L^2 (K, m)$, which may be thought of as a Laplace-Beltrami operator on $K$.
%However, since in the present work the Dirichlet form solely will play a role, the Laplace-Beltrami operator we need will be understood as the operator $\Delta : \cf\to \cf^*$ such that
%\[
% \langle \Delta f,g\rangle:=\ce (f,g)\quad f,g\in\cf\, .
%\]
A function $f\in \cf$ is said to be {\it harmonic} in a open set $A\subset K$ if, for any $g\in\cf$ vanishing on the complementary set $A^c$, one has
\[
\ce (f,g)=0\, .
\]
%Equivalently, $f$ is harmonic in $A$ if the measure $\Delta f$ is supported in $A^c$.
As a consequence of the Markovianity property \eqref{contraction}, a Maximum Principle holds true for harmonic functions on the gasket \cite{Kiga}.
In particular, one calls $0$-{\it harmonic} a function $u$ on $K$ which is harmonic in $V_0^c$. Equivalently, for given boundary values on $V_0$, $u$ is the unique function in $\cf$ such that $\ce[u] = \min\set{\ce[v]: v\in\cf, v|_{V_0} = u}$. More generally, one may call  $m$-{\it harmonic} a  function that,  given its values on $V_m$, minimizes the energy among all functions in $\cf$. For such functions we have
$$
\ce[u]=\left(\frac53\right)^m\sum_{e\in E_m}|u(e_+)-u(e_-)|^2\, .
$$
%It is not difficult to check that $f\in \cf$ is $m$-harmonic if and only if $\Delta f$ is a linear combination of Dirac measures supported on the vertices $V_m$.

\begin{dfn} (Cells, lacunas)
For any word $\s\in\Sigma_m$, define a corresponding {\it cell} in $K$ as follows
\[
C_\s :=w_{\s}(K)\, .
\]
%its {\it perimeter} by $\perim C_\s =w_{\s}(E_0)$, its (combinatorial) {\it boundary} by $\bordo C_\s =w_{\s}(V_0)$ and its (combinatorial) {\it interior} by $C_\s^o =C_\s\setminus \bordo C_\s$.
We will also define the {\it lacuna} $\ell_\emptyset$, see Fig.~\ref{lacuna}, as the boundary of the first removed triangle according to the approximation in Fig.~\ref{Fig1}. For any $\s\in\Sigma$, the lacuna $\ell_\s$ is defined as $\ell_\s :=w_{\s}(\ell_\emptyset)$. We shall use the notation $\ce_{C_\s}[u] = \lim_{m\to\infty} \left(\frac53\right)^m\sum_{e\in E_m, e\subset C_\s}|u(e_+)-u(e_-)|^2$.
     \begin{figure}[ht]
 	 \centering
	 \psfig{file=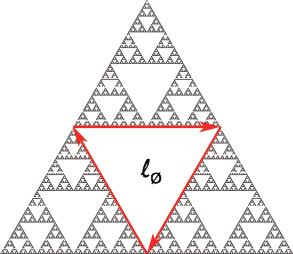}
	  \caption{The lacuna $\ell_\emptyset$}
	 \label{lacuna}
     \end{figure}
\end{dfn}

%-------------------------------------------------------------------------------------------------------------------

\subsection{The dimensional spectrum and volume}

We now choose $\alpha \in (0,1]$ and construct a triple on $K$ according to the prescriptions given in Section \ref{triplesonfractals}.
Let $S$ be the main lacuna $\ell_\emptyset$ of the gasket, identified isometrically with $\T$, and consider the triple  $\ct=(\pi,\ch,D_\a)$ constructed in Section \ref{alphatripleonS1}. Then $\ct_\emptyset=(\pi_\emptyset,\ch_\emptyset,D_\emptyset)$ is given by  $\pi_\emptyset(f)=\pi(f|_{\ell_\emptyset})$, $\ch_\emptyset=\ch$, $D_\emptyset=D_\a$, and,  for any $\sigma\in \cup_n\{0,1,2\}^n$, consider  the triple $(\pi_\s,\ch_\sigma,D_\sigma)$ on $\cc(K)$, where $\ch_\sigma=\ch_\emptyset$, $D_\sigma=2^{ |\sigma|}D_\emptyset$, and  $\pi_\sigma(f)=\pi_\emptyset(f\circ w_\sigma)$.

\begin{dfn}\label{dfn:thetriple}
Let us consider the following triple: $(\ca,\ch,D)$, where $\ch=\oplus_{\s\in\Sigma}\ch_\s$, $D=\oplus_{\s\in\Sigma}D_\s$, and $\ca$ is the subalgebra of $\cc(K)$ consisting of functions with bounded commutator with $D$, acting on $\ch$ via the representation $\pi=\oplus_{\s\in\Sigma}\pi_\s$.
According to the prescriptions of noncommutative geometry, we set $\oint f:=\tr_\omega(f|D|^{-d})$, where $tr_\omega$ is the (logarithmic) Dixmier trace, and $d$ is the abscissa of convergence of the zeta function $s\to \tr(|D|^{-s})$.
\end{dfn}

\begin{thm}\label{thm:volmeas}
Let $\alpha \in (0,1]$. The zeta function $\mathcal{Z}_D$ of $(\ca,\ch,D)$, i.e. the meromorphic extension of the function
$s\in \mathbb{C}\mapsto \tr(|D|^{-s})$, is given by
\[
\mathcal{Z}_D (s)=\frac{4\z(\a s)}{1-3\cdot 2^{- s}}\, ,
\]
where  $\z$ denotes the Riemann zeta function. Therefore, the dimensional spectrum  of the spectral triple is
$$
\mathcal{S}_{\it dim} =\{\a^{-1}\}\cup \Big\{\frac{\log3}{ \log 2} \Big( 1+\frac{2\pi i}{\log 3}k\Big): k\in \mathbb{Z} \Big\} \subset \mathbb{C}\, .
$$

As a consequence, the metric dimension $d_D$ of the spectral triple $(\ca,\ch,D)$, namely the abscissa of convergence of its zeta function, is $\max\{\a^{-1},\dH\}$,  $\dH=\frac{\log3}{\log 2}$ being the Hausdorff dimension.

When $\a> \frac{\log 2}{\log 3}$, i.e. $\dD=\dH$, $\mathcal{Z}_D$ has a simple pole in $d_D$, and the measure associated via Riesz theorem with the functional $f\to\oint f$ coincides with a multiple of the Hausdorff measure $H_{\dH}$:
$$
\vol(f) \equiv \int_K f\,d\,{\rm vol}:=tr_\omega(f|D|^{-\dH})=\frac{4\dH}{\log 3 }\frac{\z(\dH)}{(2\pi)^{\dH}} \int_K f\, dH_{\dH} \qquad f\in C(K).
$$
\end{thm}
\begin{proof}
The non vanishing eigenvalues of  $|D_\s|$ are exactly $\{2^{ |\s|}(2\pi k^\alpha )\}$, each one with multiplicity 4.

Hence $\tr(|D_\s|^{-s})=4\cdot 2^{-s  |\s|}(2\pi)^{-s}\sum_{k>0}(k^\alpha)^{-s}=4(2\pi)^{-s}2^{-s |\s|}\z(\alpha s)$ and for $\re s>\dH$ we have
\begin{align*}
\tr(|D|^{-s})
&=\sum_\s \tr(|D_\s|^{-s})=4(2\pi)^{-s}\z(\alpha s)\sum_\s 2^{-s |\s|}\\
&=4(2\pi)^{-s}\z(\alpha s)\sum_{n\geq0}2^{-s n}\sum_{|\s|=n}1\\
&=4(2\pi)^{-s}\z(\alpha s)\sum_{n\geq0}\left(3\cdot 2^{-s}\right)^n
=4(2\pi)^{-s}\z(\alpha s)(1-3\cdot 2^{-s})^{-1}\, .
\end{align*}
As the Riemann zeta function has just one pole at $s=1$ we have $\mathcal{S}_{\it dim} = \{ \alpha^{-1}\}\cup \big\{ \dH \big(1+\frac{2\pi i}{\log 3}k \big): k\in \mathbb{Z} \big\}\subset \mathbb{C}$.
Now we assume that  $\a> \frac{\log 2}{\log 3}$, i.e. $\dD=\frac{\log3}{\log 2}$, and prove that the volume measure is a multiple of the Hausdorff measure $H_{\dH}$.

Clearly, the functional $\vol(f)=tr_\omega(f|D|^{-\dH})$ makes sense also for bounded Borel functions on $K$, and we recall that the logarithmic Dixmier trace may be calculated as a residue (cf. \cite{Co}): $tr_\omega(f|D|^{-\dH})={\rm Res}_{s=\dH}\ \tr(f|D|^{-s})$, when the latter exists. Then, for any word $\t$,
\begin{align*}
tr_\omega (\chi_{C_\t}|D|^{-\dH})
&={\rm Res}_{s=\dH}\ \tr(\chi_{C_\t}|D|^{-s})\\
&=\lim_{s\to \dH^+}(s-\dH)\ \tr(\chi_{C_\t}|D|^{-s})\\
&=\lim_{s\to \dH^+}(s-\dH)\sum_\s \tr(\chi_{C_\t}\circ w_\s|D_\s|^{-s}),
\end{align*}
and we note that $\chi_{C_\t}\circ w_\s$ is not zero either when $\s<\t$ or when $\s\geq \t$. In the latter case, $\chi_{C_\t}\circ w_\s=1$.
Since  $\dH>1$, $\tr(\chi_{C_\t}|D_\s|^{-s})\leq \tr(|D_\s|^{-s})=4(2\pi)^{-s}2^{-s |\s|}\z(\alpha s)\to 4(2\pi)^{-\dH}3^{-|\s|}\z(\alpha \dH)$ when $s\to \dH^+$, hence
 $\displaystyle\lim_{s\to \dH^+}(s-\dH)\tr(\chi_{C_\t}|D_\s|^{-s})=0$. Therefore we may forget about the finitely many $\s<\t$, and get
\begin{align*}
tr_\omega (\chi_{C_\t}|D|^{-\dH})
&=\lim_{s\to \dH^+}(s-\dH)\sum_{\s\geq\t} \tr(|D_\s|^{-s})\\
&=\lim_{s\to \dH^+}(s-\dH)4(2\pi)^{-s}\z(\a s)\sum_\s 2^{-s(|\s|+|\t|)}\\
&=4\frac{\z(\a \dH)}{(2\pi)^{\dH}}2^{-\dH |\t|}\lim_{s\to \dH^+}\frac{s-\dH}{1-3\cdot 2^{-s}}\\
&=\frac{4\dH}{\log 3 }\frac{\z(\alpha \dH)}{(2\pi)^{\dH}}\left(\frac13\right)^{|\t|}
=\frac{4\dH}{\log 3 }\frac{\z(\alpha \dH)}{(2\pi)^{\dH}}H_{\dH}(C_\t)\, .
\end{align*}
This implies that for any $f\in\cc(K)$ for which $f\leq\chi_{C_\t}$, $\ds \vol(f)\leq \frac{4\dH}{\log 3 }\frac{\z(\alpha \dH)}{(2\pi)^{\dH}}\left(\frac13\right)^{|\t|}$, therefore points have zero volume, and $\vol(\chi_{\dot{C}_\t})=\vol(\chi_{C_\t})$, where $\dot{C}_\t$ denotes the interior of $C_\t$.
As a consequence, for the simple functions given by finite linear combinations of characteristic  functions of cells or vertices, $\ds \vol(\f)=\frac{4\dH}{\log 3 }\frac{\z(\alpha \dH)}{(2\pi)^{\dH}}\int\f\,d H_{\dH}$.
Since continuous function are Riemann integrable w.r.t. such simple functions, the thesis follows.
\end{proof}

\begin{rem}
In this case the functional $f\to\oint f$ does not reproduce the Hausdorff measure outside the algebra of continuous functions. Indeed such functional only depends on the behavior of $f$ on the union of all lacunas, a set which is negligible w.r.t. the Hausdorff measure.
\end{rem}

\subsection{The commutator condition and Connes metric}

In this section we will show that for $\a\in(0,1]$ the triple $(\ca,\ch,D)$ considered above is a spectral triple in the sense of Connes \cite{Co}, up to the infinite dimensionality of $\ker(D)$. Moreover, the commutator $\|[D,f]\|$ gives a Lip-norm in the sense of Rieffel \cite{Rief}. Such condition for spectral triples has been recently considered in \cite{BMR}, where these triples are called spectral metric spaces.

\begin{dfn}
We shall consider the following seminorms on functions defined on lacunas $\ell_\s$:
$$
L_{\s,\eta}(f)=\|f\|_{C^{0,\eta}(\ell_\s)}2^{|\s|(1-\eta)}
$$
\end{dfn}

\begin{prop}\label{prop:UppEst}
Let $\a\in(0,1]$. If  $f\in C^{0,1}(K)$, $p_\a$ is defined in \eqref{p-alphaseminorm}, and $c_\eps$ is given in Proposition \ref{Prop:Holder}, then
\begin{align*}
\|[D,f]\|=\sup_{\s\in\Sigma}2^{ |\s|} p_\a(f\circ w_\s)
\leq c_{1-\a}\ \sup_{\s\in\Sigma}L_{\s,1}(f)\leq c_{1-\a}\|f\|_{C^{0,1}(K)}.
\end{align*}
\end{prop}
\begin{proof}
By equation \eqref{norm-equality} and Proposition \ref{Prop:Holder}
\begin{align*}
\|[D,f]\|
&=\|\bigoplus_{\s\in\Sigma}[D_\s,\pi_\s(f)]\|=\sup_{\s\in\Sigma}\|[D_\s,\pi_\s(f)]\|\\
&=\sup_{\s\in\Sigma} 2^{ |\s|} \|[D_\a,\pi_\emptyset(f\circ w_\sigma)]\|
=\sup_{\s\in\Sigma} 2^{ |\s|} \|S_{f\circ w_\sigma}\|
=\sup_{\s\in\Sigma} 2^{ |\s|} p_\a(f\circ w_\s)\\
&\leq c_{1-\a}\ \sup_{\s\in\Sigma} 2^{ |\s|} \|f\circ w_\s\|_{C^{0,1}(\ell_\emptyset)}
= c_{1-\a}\ \sup_{\s\in\Sigma}\|f\|_{C^{0,1}(\ell_\s)} .
\end{align*}
\end{proof}

The previous Proposition gives an estimate from above of the norm of the commutator. However, by making use of Lemma \ref{Clausen-estimate}, we may get an estimate from below.

\begin{lem}\label{lem:LowEst}
Let $\a\in[\frac12,1]$, and $\tilde{c}_\a$ be as in Proposition \ref{Prop:Holder}. Then
\begin{align*}
\|[D,f]\| & \geq \tilde{c}_\a\, \sup_{\s\in\Sigma}L_{\s,\a}(f).
\end{align*}
\end{lem}
\begin{proof}
\begin{align*}
\|[D,f]\| & = \sup_{\s\in\Sigma} 2^{ |\s|} p_\a(f\circ w_\s)
\geq \tilde{c}_\a\, \sup_{\s\in\Sigma} 2^{ |\s|} \|f\circ w_\s\|_{C^{0,\a}(\ell_\emptyset)}
= \tilde{c}_\a\, \sup_{\s\in\Sigma}\|f\|_{C^{0,\a}(\ell_\s)} 2^{|\s|(1-\a)}.
\end{align*}
\end{proof}

\begin{prop}\label{prop:LowEst}
Let $\a\in[\frac12,1]$. There exists a constant $k(\a)$ such that
\begin{equation}
\|f\|_{C^{0,1}(K)}\leq k(\a)\|[D,f]\|.
\end{equation}
\end{prop}
\begin{proof}
Our aim  is to estimate $|f(x)-f(y)|$ for a continuous function $f$ for which $\|[D,f]\|<\infty$.
\\
1$^{st}$ step. Let $C_\s$ be a  cell of level $m$, $x$ a point in $C_\s$. We now construct inductively a sequence of cells $C_{\s(j,x)}$, $j\geq1$, such that $x\in C_{\s(j,x)}$, $C_{\s(1,x)}:= C_{\s}$, $C_{\s(j+1,x)}\subset C_{\s(j,x)}$, $|\s(j,x)|=m+j-1$ (if $x$ is not a vertex such sequence is uniquely determined).
We then construct a sequence $\{x_j\}_{j\geq1}$ of points as follows: $x_1$ is a  vertex of $\ell_\s$ contained in $C_{\s(2,x)}$, $x_j$ is the unique point in $\ell_{\s(j-1,x)}\cap \ell_{\s(j,x)}$, $j>1$. By construction, $x_j\to x$ and   the points $x_j,x_{j+1}$ belong to the lacuna $\ell_{\s(j,x)}$.

We now observe that, by Lemma \ref{lem:LowEst},
\begin{align*}
|f(x_{j+1})-f(x_j)|
& \leq \|f\|_{C^{0,\a}(\ell_{\s(j,x)})}d(x_{j+1},x_j)^\a
\leq L_{\s(j,x), \a}(f) 2^{-(m+j-1)(1-\a)}(\diam(\ell_{\s(j,x)}))^\a\\
& \leq \tilde{c}_\a^{-1}2^{-\a}\|[D,f]\| 2^{-(m+j-1)}.
\end{align*}
As a consequence,
\begin{align*}
|f(x_1)-f(x)| & \leq \sum_{j\geq 1}|f(x_{j+1})-f(x_j)| \leq \tilde{c}_\a^{-1}2^{1-\a} \|[D,f]\| 2^{-m}.
\end{align*}
\\
2$^{nd}$ step. If $x_0$ is a vertex of level $n\ne0$, and $m\geq n$, the butterfly shaped neighborhood $W(x_0,m)$ is the union of the two cells of level $m$ containing $x$.
For $x,y\in K$, let $W(x_0,m)$ be a minimal butterfly shaped neighborhood containing them. Observe that, by minimality, at least one of the points, say $x$, does not belong to $W(x_0,m+1)$, hence $\r_{geo}(x,y)\geq\r_{geo}(x,x_0)\geq2^{-(m+1)}$.

Let us now choose $W(x_1,m+1)$ contained in one of the wings of  $W(x_0,m)$ and containing both $x$ and $x_0$. Reasoning as in the first step,
\begin{align*}
|f(x_0)-f(x)|\leq |f(x_0)-f(x_1)| + |f(x_1)-f(x)|
\leq 2\tilde{c}_\a^{-1}2^{1-\a} \|[D,f]\|2^{-m},
\end{align*}
hence,
\begin{align*}
|f(y)-f(x)|
& \leq |f(y)-f(x_0)| + |f(x_0)-f(x)|
\leq 4\tilde{c}_\a^{-1} 2^{1-\a} \|[D,f]\|2^{-m}\\
& \leq 8\tilde{c}_\a^{-1} 2^{1-\a} \|[D,f]\| \, \r_{geo}(x,y).
\end{align*}
The thesis follows.
\end{proof}

\begin{thm}\label{cor:triple}
For any $\a\in(0,1]$,  the algebra $\ca$ contains $\cc^{0,1}(K)$ and the triple $(\ca,\ch,D)$ is a spectral triple. Moreover, for $\a\in[\frac12,1]$, $\ca$ coincides with $\cc^{0,1}(K)$, and the seminorm $f\to\|[D,f]\|$ is a Lip-norm according to Rieffel  \cite{Rief}. In particular, the metric
$$
\r_D(x,y)=\sup_{f\in\ca}\frac{|f(x)-f(y)|}{\|[D,f]\|}
$$
 is bi-Lipschitz w.r.t. the  Euclidean geodesic metric $\r_{geo}$ on $K$.
\end{thm}
\begin{proof}
It follows from Proposition \ref{prop:UppEst} that $\ca$ is dense in $C(K)$, which, together with the results in the previous Sections, give the spectral triple property.  The Lip-norm property follows by Proposition \ref{prop:LowEst}, cf. \cite{Rief}. Indeed, it implies that functions for which $\|[D,f]\|\leq1$ are equicontinuous, which gives the compactness property of the set of elements for which  $\|[D,f]\|\leq1$ and $\|f\|\leq1$, and it also implies that $\|[D,f]\|$ vanishes only on constant functions. The equivalence of the seminorms follows from Propositions \ref{prop:UppEst} and \ref{prop:LowEst}. The other results easily follow.
\end{proof}

\subsection{The gasket in K-homology}
Let $\a\in(0,1]$, and denote by $(\ca,\ch,D)$   the spectral triple for the gasket considered above. Let $F$ be the phase of $D$, and, for any $\s\in\Sigma$, denote by $\g_\s,\eps_\s,S_\s$, a copy of the operators $\g,\eps,S$ associated, by Proposition \ref{FrModS1}, to the lacuna $\ell_\s$, identified with $\T$. Finally, let $\g=\oplus_{\s\in\Sigma} \g_\s$,  $\eps =\oplus_{\s\in\Sigma} \eps_\s$,  $S = \oplus_{\s\in\Sigma} S_\s$.

\begin{thm}
Let $\a\in(0,1]$. The quintuple $\cf = (\ch,\pi,F,\g,\eps)$ is a tamely degenerate 1-graded Fredholm module on $\ca$, as in Definition \ref{dfn:kernelDeg}.
The ungraded Fredholm module $\cf^+ = (\ch^+,\pi^+,F^+)$ associated to it by Proposition \ref{Ex8.8.4} is a tamely degenerate module on $\ca$, and $F^+= S$.
The module $\cf^+$ is non trivial in K-homology. In particular, it pairs non trivially with the generators of the (odd) K-theory of the gasket associated with the lacunas.
\end{thm}
\begin{proof}
First step. We  check the compactness of
$[\eps ,\pi(f)]=\oplus_\s[\eps_\s,\pi_\emptyset(f\circ w_\s)]$. As in the proof of Proposition \ref{FrModS1}, this amounts to prove that $\oplus_\s[S_\s,\pi_\emptyset^0(f\circ w_\s)]$ is compact, where $\pi_\emptyset^i$ denotes the action of $C(\ell_\emptyset)$ on $L^2(\Omega^i(\ell_\emptyset))$, $i=0,1$. Even though each summand is compact, the compactness of the direct sum is not obvious.

We first consider an affine function $f$ in the plane, restricted to the gasket, and observe that consequently $f\circ w_\s |_{\ell_\emptyset}=\const+2^{-|\s|}f |_{\ell_\emptyset}$. Let us denote by $\{s_n\}$ the sequence of the singular values with multiplicity, arranged in a non increasing order, of $[S_\emptyset,\pi_\emptyset^0(f)]$. Then, for any given $\s$,
$$
[S_\s,\pi_\emptyset^0(f\circ w_\s)]
=2^{-|\s|}[S_\emptyset,\pi_\emptyset^0(f)],
$$
namely the sequence of singular values for $\oplus_\s[S_\s,\pi_\emptyset(f\circ w_\s)]$ is $\{2^{-|\s|}s_n:\s\in\Sigma,n\in\bn\}$, showing that $[\eps,\pi(f)]$ is compact.

Now, for any given $n\in\bn$, consider the piece-wise affine functions $\Aff_n(K)$ on the gasket, which are affine when restricted to cells of level $n$. Reasoning as above, we obtain that, for $|\s|=n$, the operator $\oplus_{\t\geq \s}[\eps_\s,\pi_\emptyset(f\circ w_\t)]$ is compact, from which the compactness of $[\eps,\pi(f)]$ follows again. Since $\cup_n \Aff_n(K)$ is dense in $\ca$, the thesis is proved. The other properties being obvious, we have proved that $\cf$ is a kernel-degenerate $1$-graded Fredholm module. Therefore, by Proposition \ref{Ex8.8.4}, $\cf$  is a kernel-degenerate ungraded Fredholm module.

Second step. According to Proposition \ref{Ex8.8.4}, it is sufficient to prove the tame degeneracy of the ungraded Fredholm module. Since  $K^1(K)$ is a direct sum of countably many copies of $\bz$ it is sufficient to verify the  equation \eqref{tameF} only for the generators, namely for the unitaries $u_\s$ having winding number 1 around $\ell_\s$ and winding number 0 around all other lacunas. However, for the unitary $u_\s$, the global index in \eqref{tameF} is equal to the index on the lacuna $\ell_\s$, which is clearly trivial. Tameness follows.
\end{proof}

\subsection{The Dirichlet form}

Let us recall that the integral $\oint a$ of an element $a\in\ca$ in noncommutative geometry is defined as the Dixmier trace $\tr_\o(a|D|^{d})$, where $d$ is the metric dimension of the triple. Such integral  may be computed, for a  positive bounded $a$, in two equivalent ways:

\begin{align}
%& \lim_n \frac{\s_n(|D|^{-d/2}a|D|^{-d/2})}{\log n};\label{formulaV1}\\
&\lim_{s\to 1} (s-1) \tr \left( (|D|^{-d/2}a|D|^{-d/2})^s\right);\label{formulaV2}\\
&\lim_{s\to 1} (s-1) \tr (a|D|^{-sd})=d^{-1} \lim_{t\to d}(t-d) \tr (a|D|^{-t});\label{formulaV3}
\end{align}
when such limits exist, cf. \cite{Co} Proposition 4 p.306, and  \cite{CPS} Corollary 3.7 (in this case the noncommutative integral is independent of the choice of the ultrafilter $\omega$ on $\mathbb{N}$).

However, things change when we consider unbounded $a$'s. First of all, we replace $a|D|^{-sd}$ with $|D|^{-sd/2}a|D|^{-sd/2}$ in such a way that the trace is well defined (possibly infinite). Moreover, while the boundedness of $(s-1) \tr\left( (|D|^{-d/2}a|D|^{-d/2})^s\right)$ for $s>1$ is equivalent to $|D|^{-d/2}a|D|^{-d/2}\in\cl^{1,\infty}$ (cf. \cite{CRSS} Thm. 4.5), the boundedness of  $(s-1) \tr (|D|^{-sd/2}a|D|^{-sd/2})$ for $s>1$ is in general a weaker condition  (cf. Lemma \ref{HalfMiRoTrain}). Indeed, when  classical $d$-manifolds $M$ are concerned, with $|D|=\Delta^{1/2}$,  \cite{LPS} shows that the residue at $1$ of $(s-1) \tr (|D|^{-sd/2}f|D|^{-sd/2})$ is finite and gives the integral of $f$ on $M$ w.r.t. the volume form (up to a multiplicative constant) for any function $f\in L^1 (M)$, that the same is true for the residue at $1$ of $(s-1)  \tr\left( (|D|^{-d/2}a|D|^{-d/2})^s\right)$ only when $f\in L^{1+\eps} (M)$, and examples are given of $f\in L^1 (M)$ such that $|D|^{-d/2}a|D|^{-d/2}$ does not belong to $\cl^{1,\infty}$.

\bigskip

Our aim here is to use the NCG translation table to define a Dirichlet energy on spectral triples, cf. also \cite{PeBe,JuSa}. Starting with the classical Dirichlet integral $f\mapsto\int_M|\nabla f|^2 \, d{\rm vol}$ on a Riemannian manifold $M$, we replace the gradient $\nabla f$ with the commutator $[D,f]$, the integral as explained above, and get the nonnegative quadratic functional
\begin{equation}\label{formulaE2}
a\mapsto\ce_D[a]:=\tr_\o\left(|D|^{-\d/2}|[D,a]|^2|D|^{-\d/2}\right),
\end{equation}
for a suitable $\d$. As above, we may hope to compute the energy also as
\begin{align}
%& \lim_n \frac{\s_n(|D|^{\d/2}|[D,a]|^2|D|^{\d/2})}{\log n};\label{formulaE1}\\
%&\lim_{s\to 1} (s-1) \tr (|D|^{\d/2}|[D,a]|^2|D|^{\d/2})^s;\label{formulaE2}\\
&\lim_{s\to 1} (s-1) \tr (|D|^{s\d/2}|[D,a]|^2|D|^{s\d/2})=\d^{-1}\lim_{t\to \d} (t-\d) \tr (|D|^{t/2}|[D,a]|^2|D|^{t/2}).\label{formulaE3}
\end{align}
For classical Riemannian $d$-manifolds $M$, the energy is finite for functions in the Sobolev space  $f\in H^1 (M)$, namely $|\nabla f|^2\in L^1 (M)$, therefore   the analysis in \cite{LPS} (for $\d=d$) shows that formula \eqref{formulaE3} is finite and  recovers a multiple of the Dirichlet energy form for all $f\in H^1 (M)$, while formula \eqref{formulaE2}   recovers a multiple of  the Dirichlet energy form only for functions in a proper subset.

This will be the case also in our present setting (even though we could not produce a counterexample) where the manifold is replaced by the Sierpinski gasket and the Dirichlet integral by the standard Dirichlet form. In particular, we prove that formula \eqref{formulaE3} is finite and  recovers a multiple of the standard Dirichlet form for all finite energy functions, while we are able to prove that formula \eqref{formulaE2}   recovers a multiple of  the standard Dirichlet form only on a form core.

Moreover, as a counterpart of the results of Kusuoka \cite{Ku}, Ben-Bassat-Strichartz-Teplyaev \cite{BST} (see also \cite{Hino,HiNa})  showing that self-similar measures and energy measures are singular on the gasket, the exponent $\d$ in \eqref{formulaE2}, \eqref{formulaE3}, which we call {\it energy dimension}, is smaller than the volume dimension $d$. As a consequence,  we cannot exclude that the multiplicative constant  relating $\ce_D[a]$ in formula \eqref{formulaE2} with the standard Dirichlet form may even depend on the generalized limit $\o$.

Let us remark here that we prove that the forms described in  \eqref{formulaE2} and  \eqref{formulaE3} coincide (up to a constant)  with the standard energy form on the gasket on suitable domains. The question whether  such formulas directly give a Dirichlet form, either for general spectral triples or for the case of fractals, is not treated here, and will be the subject of future research.

Also, the residue and the Dixmier trace formulas, which are essentially equivalent for the case of the volume, are proved here in a quite independent way, the second requiring rather technical results on singular traces; this is why the domains of their validity are different, and the two constants giving the relation with the standard energy are unrelated, only estimates being available.

\bigskip

Now we prove that formula \eqref{formulaE3} reproduces a multiple of the standard Dirichlet form on the gasket. In the following Theorem, a result of Jonsson \cite{Jo} on the regularity of the trace of a finite energy function on an edge of the gasket, will imply that the standard Dirichlet form on the gasket can be recovered via the spectral triple only if $\a$ is not too close to 1.
In this section, when $f$ is a continuous function on the gasket, $\ce[f]$
denotes the values of the (possibly infinite) standard Dirichlet form on $f$.
Let us first observe that
\begin{align}\label{trace1}
\tr(|D|^{-s/2}|[D,f]|^2|D|^{-s/2})
& =\sum_{\s}
\tr(|D_\s|^{-s/2}|[D_\s,\pi_\s(f)]|^2|D_\s|^{-s/2})\\
\nonumber & =\sum_{\s}
2^{(2-s)|\s|}
\tr(|D_\emptyset|^{-s/2}|[D_\emptyset,\pi_\emptyset(f\circ w_\s)]|^2|D_\emptyset|^{-s/2}).
\end{align}
However, the following holds.

\begin{lem}\label{traceClass}
Let $s>\frac1\a$, $\a_0=\frac{\log(10/3)}{\log4}\approx 0.87$. Then:
\begin{itemize}
\item[$(i)$]  $\tr(|D_\emptyset|^{-s/2}|[D_\emptyset,g]|^2|D_\emptyset|^{-s/2})$ is finite if and only if $g\in H^\a(\ell_\emptyset)$.
\item[$(ii)$]  $\tr(|D|^{-s/2}|[D,f]|^2|D|^{-s/2})<\infty$, $\forall f$ with  finite energy  on $K$ $\Rightarrow$ $\a\leq\a_0$.
\end{itemize}
\end{lem}
\begin{proof}
$(i)$ By \eqref{commutator}, and the definition of $D_\emptyset$, we get
\begin{multline}\label{trace2}
\tr (|D_\emptyset|^{-s/2}|[D_\emptyset,\pi_\emptyset(g)]|^2|D_\emptyset|^{-s/2})
=\tr\left(|D_\emptyset|^{-s/2}
\begin{pmatrix}S_{g}S^*_{g}&0\\
0&S^*_{g^*}S_{g^*}\end{pmatrix}
 |D_\emptyset|^{-s/2}\right)\\
=
\tr((\partial_\a^* \partial_\a)^{-s/4}S^*_{g^*}S_{g^*}(\partial_\a^* \partial_\a)^{-s/4})
+
\tr((\partial_\a \partial_\a^*)^{-s/4}S_{g}S^*_{g}(\partial_\a \partial_\a^*)^{-s/4}).
\end{multline}
As a consequence, Lemma \ref{lem:technical} implies
\begin{equation}\label{doubleEstimate}
2\z(\a s)\|\partial_\a g\|_{L^2(\ell_\emptyset \times \ell_\emptyset)}^2
\leq
\tr (|D_\emptyset|^{-s/2}|[D_\emptyset,\pi_\emptyset(g)]|^2|D_\emptyset|^{-s/2})
\leq
4 \z(\a s)\|\partial_\a g\|_{L^2(\ell_\emptyset \times \ell_\emptyset)}^2\,.
\end{equation}
$(ii)$ Let $\a>\a_0$, $g\in H^{\a_0}(\ell_\emptyset)\setminus H^{\a}(\ell_\emptyset)$. By \cite{Jo} Thm. 5.1, there exists a finite energy function $f$ on $K$ such that $f|_{\ell_\emptyset}=g$. By $(i)$, $\tr(|D_\emptyset|^{-s/2}|[D_\emptyset,\pi_\emptyset(f)]|^2|D_\emptyset|^{-s/2})$ is infinite, and the thesis follows.
\end{proof}

Next Theorem shows that condition $\a\leq\a_0$, together with $\a>\frac{\log 2}{\log 12/5}$, is also sufficient for the recovery of the energy via formula \eqref{formulaE3}.

\begin{thm}\label{thm:res=en}
Let $\a\in(0,\a_0]$, $\deD =\max \{ \a^{-1},\dE \}$, with $\dE:=\frac{\log 12/5}{\log 2} \approx 1.26$. Then:
\begin{itemize}
\item[$(i)$]
For any $f$ with finite energy, $s>\d_D$, $|D|^{-s/2}|[D,f]|^2|D|^{-s/2}$ is a trace class operator.
\item[$(ii)$]
For $\a\in (\dE^{-1}, \a_0]$, so that $\deD =\dE$, and $f$ with finite energy,  the functional
$$Z_{D,f}(s):=\tr(|D|^{-s/2}|[D,\pi(f)]|^2|D|^{-s/2}),$$
defined for $\Re s >\dE$, has abscissa of convergence $\dE$, where it has a simple pole,  and there exists a constant $A$ such that
\begin{equation}\label{energy1}
\Res_{s=\dE}Z_{D,f}(s) =A\ \ce[f].
\end{equation}
\end{itemize}
\end{thm}

\begin{proof}
$(i)$ According to formulas \eqref{trace1}, \eqref{trace2}, \eqref{doubleEstimate},
\begin{equation}\label{traceineq1}
\tr(|D|^{-s/2}|[D,\pi(f)]|^2|D|^{-s/2})
\leq 4\z(\a s) \sum_{\s}
2^{(2-s)|\s|}\|\partial_\a \pi_\emptyset(f\circ w_\s)\|_{L^2(\ell_\emptyset \times \ell_\emptyset)}^2\,.
%\| f\circ w_\s\|_{H^\a(\ell_\emptyset)}^2
\end{equation}

By \cite{Jo} Thm. 4.1, the restriction map from $\cf$ to $H^\a(\ell_\emptyset)$ is continuous (for $\a\leq\a_0)$, implying in particular that there exists a constant $K_1=K_{1,\a}$ such that
\begin{equation}\label{K1Est}
\|\partial_\a g\|_{L^2(\ell_\emptyset\times \ell_\emptyset)}^2\leq K_1 \ce[g]
,\quad \forall g\in\cf.
\end{equation}
Hence,
\begin{align}\label{normEnEst}
\sum_{|\s|=n}\|\partial_\a \pi_\emptyset(f\circ w_\s)\|_{L^2(\ell_\emptyset)}^2
\leq K_1\sum_{|\s|=n} \ce[ f\circ w_\s]
= K_1 \left(\frac35\right)^{n}\sum_{|\s|=n}\ce_{C_\s}[ f]
= K_1 \left(\frac35\right)^{n}\ce[ f].
\end{align}
As a consequence, if $s > \deD = \max \{ \a^{-1},\frac{\log 12/5}{\log 2} \}$,
\begin{equation}\label{zetaEst}
\tr(|D|^{-s/2}|[D,\pi(f)]|^2|D|^{-s/2})
\leq
4K_1 \z(\a s)\sum_n\left(\frac35 2^{2-s} \right)^{n} \ce[f]
= \frac{4K_1 \z(\a s)}{1-\frac35 2^{2-s}} \ce[f],%\left(\right)^{-1}
\end{equation}
proving the first statement of the theorem.

\medskip\noindent
$(ii)$ $(a)$ We first determine the constant $A$. Let us observe that, up to a multiplicative constant,  the matrix $\begin{pmatrix}2&-1&-1\\-1&2&-1\\-1&-1&2\end{pmatrix}$ determines the unique non-degenerate positive semidefinite quadratic form $Q[v]=\sum_{i,j}|v_i-v_j|^2$ on $\bc^3$, invariant under permutations of the components and vanishing on constant vectors.
We then consider the linear map which associates to any vector $\vec{v}=(v_0,v_1,v_2)\in\bc^3$,  the $0$-harmonic function $g=\f(\vec{v})$ on the gasket taking values $v_i$, $i=0,1,2$,  on the extreme points of the lacuna $\ell_\emptyset$.  Then, the form $\vec{v}\to\|\partial_\a \pi_\emptyset(\f(\vec{v}))\|^2_{L^2(\ell_\emptyset\times\ell_\emptyset)}$ possesses al the properties of the quadratic form on $\bc^3$ induced by the matrix above.
Indeed, for $g= \f(\vec{v})$, $\|\partial_\a \pi_\emptyset(g)\|^2=\sum_{p} |p|^{2\a }|( \pi_\emptyset(g),e_{p})|^2 $, which is positive semidefinite and vanishes on constants. More precisely, it vanishes only if $g$ is constant on $\ell_\emptyset$, which implies $g$ is constant, since $g$ is zero-harmonic. Finally, such form is invariant under isometries of the circle, hence, in particular, it is invariant under the symmetry group of the triangle, that is, for 0-harmonic functions, under the permutation group of the vertices.

Hence $\|\partial_\a \pi_\emptyset(\f(\vec{v}))\|^2_{L^2(\ell_\emptyset\times\ell_\emptyset)}$ is a multiple of $Q[v]$, which in turn is $\frac1{25}\ce[g]$, since $g$ is $0$-harmonic, namely there exists a non-zero constant $K_2=K_{2,\a}$ such that
\begin{equation}\label{normEnEstBis}
\|\partial_\a \pi_\emptyset(g)\|^2_{L^2(\ell_\emptyset \times \ell_\emptyset)}
= K_2\ce[g],\quad \forall \text{ 0-harmonic }g.
\end{equation}
By formula \eqref{estimate1}, $\|S^*_{\pi_\emptyset(g)}\partial_\a e_k\|^2=\frac14\sum_{p}  (|k|^{2\a }+|p|^{2\a }-|p-k|^{2\a })^2 |( \pi_\emptyset(g),e_{p})|^2$.
As above, for $k\ne0$, $g\to\|S^*_{\pi_\emptyset(g)}\partial_\a e_k\|^2$ is a non-degenerate positive semidefinite quadratic form on $0$-harmonic functions, invariant under symmetries of the triangle and vanishing on constant vectors. Therefore it is again a multiple of the energy of $g$,  namely $\forall k\ne0\,\exists C_k>0$ such that
\begin{equation}\label{normEnEstTer}
\|S^*_{\pi_\emptyset(g)}\partial_\a e_k\|^2=C_k\ce[g],\quad \forall \text{ 0-harmonic }g.
\end{equation}
The constant $A$ is then defined as
$$A=\frac{ 2K_2 \z(\a \dE)+C(\dE) }{ \log 2 }, $$
where we set $C(s)=\sum_kC_k|k|^{-(s+2)\a}$.
We note that, by formula \eqref{estimate1},  $0<C_k\leq |k|^{2\a}$ for  $k\ne0$ and, by relations \eqref{estimate1}, \eqref{normEnEstBis}, \eqref{normEnEstTer},
$$
\ce[g]\sum_kC_k|k|^{-(s+2)\a}
=\sum_k |k|^{-(s+2)\a}\|S^*_{g}\partial_\a e_k\|^2
\leq 2\z(s\a)\| \de_\a g\|_{L^2(\ell_\emptyset \times \ell_\emptyset)}^2
=2K_2\z(s\a)\ce[g],
$$
 for any $0$-harmonic  $g$, namely $C(s)\leq 2K_2\z(\a s)$. In particular, $C$ is analytic for $s>1/\a$.

\medskip\noindent
$(b)$ formula \eqref{energy1} for $q$-harmonic functions. According to formulas \eqref{trace1}, \eqref{trace2}, \eqref{estimate2}, and \eqref{formula3}, we have
\begin{equation}\label{Zenergy}
Z_{D,f}(s)=\sum_\s 2^{(2-s)|\s|}\left(
2\zeta(\a s)\|\partial_\a \pi_\emptyset(f\circ w_\s)\|_{L^2(\ell_\emptyset \times \ell_\emptyset)}^2 +
\sum_{k\neq 0} |k|^{-(s+2)\a}\|S^*_{\pi_\emptyset(f\circ w_\s)}\partial_\a e_k\|_{L^2(\ell_\emptyset )}^2\right).
\end{equation}
Assume now $f$ to be $q$-harmonic, and $s> \dE$. Then, when $|\s|\geq q$, $f\circ w_\s$ is $0$-harmonic. Hence, making use of the equalites in  \eqref{normEnEstBis}, \eqref{normEnEstTer}, %and \eqref{Zenergy},

\begin{equation}\label{ZenergyBis}
\begin{aligned}
&\sum_{|\s|\geq q} \tr(|D_\s|^{-s/2}|[D_\s,\pi_\s(f)]|^2|D_\s|^{-s/2})\\
&=\sum_{|\s|\geq q} 2^{(2-s)|\s|}\left(
2K_2\zeta(\a s)\ce[ f\circ w_\s] +
\sum_{k\neq 0} C_k |k|^{-(s+2)\a}\ce[f\circ w_\s]\right)\\
&= 2K_2\zeta(\a s)\sum_{n\geq q}2^{(2-s)n}\sum_{|\s|=n}
\ce[ f\circ w_\s] +
\sum_{n\geq q}2^{(2-s)n}\sum_{k\neq 0} C_k |k|^{-(s+2)\a}\sum_{|\s|=n}\ce[f\circ w_\s]\\
&= \ce[f]\left(
2K_2\zeta(\a s)+C(s)
\right)\left(\frac35 2^{2-s}\right)^{q}\left(1-\frac35 2^{2-s}\right)^{-1}.
\end{aligned}
\end{equation}

As a consequence,
\begin{equation}\label{ZenergyTer}
\begin{aligned}
\Res_{s=\dE}Z_{D,f}(s) =&\lim_{s\to \dE^+}(s-\dE)Z_{D,f}(s)\\
=&\lim_{s\to \dE^+}(s-\dE)\Big(\sum_{|\s|<q} \tr(|D_\s|^{-s/2}|[D_\s,\pi_\s(f)]|^2|D_\s|^{-s/2})\\
&+ \ce[f]\left(2K_2\zeta(\a s)+C(s)\right)
\left((3/5) 2^{2-s}\right)^{q}\left(1-(3/5) 2^{2-s}\right)^{-1}\Big)\\
=& \frac{2K_2\zeta(\a\, \dE)+C(\dE)}{\log2}\,\ce[f]=A\, \ce[f].
\end{aligned}
\end{equation}

\medskip\noindent
$(c)$ formula \eqref{energy1} for  general functions in  $\mathcal{F}$.
Let us observe that, for any $s>\dE$, the functional
$$
f\to N_{s}(f):=Z^{1/2}_{D,f}(s)
$$
is a seminorm on $\mathcal{F}$. We now choose $f\in\cf$, and let $g$ be a finitely-harmonic function. Then,
\begin{align*}
&|(s-\dE)^{1/2}N_s(f)-A^{1/2}\ce^{1/2}[f]|\\
&\leq (s-\dE)^{1/2}|N_s(f)-N_s(g)|
+|(s-\dE)^{1/2}N_s(g)-A^{1/2}\ce^{1/2}[g]|
+A^{1/2}|\ce^{1/2}[g]-\ce^{1/2}[f]|\\
&\leq (s-\dE)^{1/2}N_s(f-g)
+|(s-\dE)^{1/2}N_s(g)-A^{1/2}\ce^{1/2}[g]|
+A^{1/2}\ce^{1/2}[f-g]\\
&\leq
\left((4K_1 \z(\a s))^{1/2}\left(\frac{s-\dE}{1-\frac35 2^{2-s}}\right)^{1/2} +A^{1/2}\right)\ce^{1/2}[f-g]
+|(s-\dE)^{1/2}N_s(g)-A^{1/2}\ce^{1/2}[g]|
\\
&\longrightarrow
\left(\left(\frac{4K_1 \z(\a\,\dE)}{\log2}\right)^{1/2} +A^{1/2}\right)\ce^{1/2}[f-g],
\quad s\to\dE^+
\end{align*}
where for the last inequality we used inequality \eqref{zetaEst}.
Since $g$ varies among finitely-harmonic functions, the last term may be made arbitrarily small, namely
$$
\exists\lim_{s\to\dE^+}|(s-\dE)^{1/2}N_s(f)-A^{1/2}\ce^{1/2}[f]|=0,
$$
and the thesis is proved.
\end{proof}

\subsection{Standard Dirichlet form and Dixmier traces}

In this section we reconstruct the standard Dirichlet form on the Sierpinski gasket using the Dixmier trace. In particular, the self-similar energy of a function in a suitable form core coincides with the evaluation, by the Dixmier trace, of the square of the modulus of its commutator with the Dirac operator $D$ times a symmetrized weight proportional to a negative power of $|D|$. Throughout all this section we shall assume $\a\in (\dE^{-1}, \a_0]$, so that $\deD=\dE$.

\begin{dfn}\label{decay}
For any $\eps>0$, we shall consider the set $\cb_\eps$ defined as follows:
$$
\cb_\eps:=\{f\in\cf:\exists  c_f>0\textit{ such that }\ce_{C_\s}[f]\leq c_f\, e^{-\eps |\s|}\ce[f], \s\in\Sigma\},
$$
and set $\cb:=\cup_{\eps>0}\cb_\eps$.
\end{dfn}

\begin{lem}\label{harmonicInB}
Let $f$ be a  $k$-harmonic function. Then $f\in\cb_{\eps}$ for $\eps\leq\log(5/3)$. More precisely, for any $\s\in\Sigma$, $\ce_{C_\s}[f]\leq(3/5)^{(|\s|-k)}\ce[f]$.
\end{lem}
\begin{proof}
It is easy to check that if $f$ is a harmonic function in the interior of a cell $C$ and $C_1$ is one of its three sub-cells, then $\osc(f)(C_1)\leq\frac35\osc(f)(C)$ (see for example \cite{St06} Chapter 1 Exercise 1.3.6). From this the thesis follows.

\end{proof}

\begin{prop}
Each $\cb_\eps$, $\eps>0$, is a vector space, $\cb$ is an algebra.
\end{prop}
\begin{proof}
We first prove additivity. The case $f_1+f_2 = {\it const}$ being trivial, we assume $\ce[f_1+f_2]\neq0$.
Then
\begin{align*}
\ce_{C_\s}[f_1+f_2]&\leq2\ce_{C_\s}[f_1]+2\ce_{C_\s}[f_2] \\
&\leq 2(c_1 \ce[f_1] + c_2 \ce[f_2] )e^{-\eps |\s|}=c\, e^{-\eps |\s|}\ce[f_1+f_2],
\end{align*}
where $c=\ce[f_1+f_2]^{-1}2(c_1 \ce[f_1] + c_2 \ce[f_2] )$.
As for multiplicativity, assuming as before $\ce[f_1 f_2]\neq0$, we get
\begin{align*}
\ce_{C_\s}[f_1 f_2]&\leq \|f_2|_{C_\s}\|_\infty \ce_{C_\s}[f_1]+ \|f_1|_{C_\s}\|_\infty \ce_{C_\s}[f_2]
\\
&\leq \|f_2\|_\infty c_1\ce[f_1] e^{-\eps_1 |\s|}+ \|f_1\|_\infty c_2\ce[f_2] e^{-\eps_2 |\s|}
=c\, e^{-\eps  |\s|}\ce[f_1 f_2],
\end{align*}
where $\eps=\min \{ \eps_1,\eps_2 \}$ and $c=\ce[f_1 f_2]^{-1}(\|f_2\|_\infty c_1\ce[f_1] + \|f_1\|_\infty c_2\ce[f_2] )$.
\end{proof}

\begin{rem}
Let us note that the inequalities in the proof of the following Lemma are very close to those in \eqref{zetaEst},  the Hilbert-Schmidt norm being replaced by the uniform norm. For functions in $\cb$, this allows us to prove an estimate for values of $s$ below $\dE$.
\end{rem}

\begin{lem}\label{s-range}
Assume  $0<\eps<(\dE-\a^{-1})\log2$, $f\in \cb_\eps$. Then,
for $s \geq\dE-\eps/\log2$, $[D,f]|D|^{-\frac12s}$ is bounded.
\end{lem}
\begin{proof}
Making use of  \eqref{doubleEstimate}, \eqref{K1Est}, we get:
\begin{align*}
\|[D,\pi(f)] \ |D|^{-\frac12s}\|^2
&=\sup_\s \|[D_\s,\pi_\s(f)]|\,|D_\s|^{-\frac12s}\|^2\\
&=\sup_\s 2^{|\s|(2-s)}\|[D_\emptyset,\pi_\emptyset(f\circ w_\s)]|\,|D_\emptyset|^{-\frac12s}\|^2\\
&\leq\sup_\s 2^{|\s|(2-s)}\tr(|D_\emptyset|^{-s/2}|[D_\emptyset,\pi_\emptyset(f\circ w_\s)]|^2|D_\emptyset|^{-s/2})\\
&\leq 4\sup_\s 2^{|\s|(2-s)} \z(\a s)\|\partial_\a \pi_\emptyset(f\circ w_\s) \|_{L^2(\ell_\emptyset \times \ell_\emptyset)}^2\\
&\leq 4K_1\z(\a s)\sup_\s 2^{|\s|(2-s)}   \ce[f\circ w_\s]\\
&\leq 4K_1\z(\a s)\sup_\s 2^{|\s|(2-s)} \left(\frac35\right)^{|\s|}  \ce_{C_\s}[ f]\\
%&\leq 4K_1\z(\a s)\sup_n \exp(((2-s)\log2+\log(3/5))n)\max_{|\s|=n}   \ce_{C_\s}[ f]\\
&\leq 4K_1\z(\a s)\sup_n \Big( \frac{12}5 2^{-s} \Big)^n \max_{|\s|=n}   \ce_{C_\s}[ f]\\
%&\leq 4cK_1\z(\a s) \sup_n \exp(((2-s)\log2+\log(3/5)-\eps)n)   \ce[ f].
&\leq 4c_fK_1\z(\a s) \sup_n \Big( \frac{12}5 2^{-s}e^{-\eps} \Big)^n   \ce[ f].
\end{align*}
We get a non trivial bound when $-s\log2+\log(12/5)-\eps\leq0$ and $\a s>1$, namely
$$
s\geq \max \Big\{ \a^{-1},\dE-\frac{\eps}{\log2} \Big\}.
$$
Since $\eps<(\dE-\a^{-1})\log2$, this amounts to $s\geq\dE -\frac{\eps}{\log2}$.
\end{proof}

\begin{rem}
Let us notice that the actual bound on the norm of $[D,f]|D|^{-\frac12s}$ does not play any role, see inequality \eqref{Dix<En}.
\end{rem}

\begin {thm}\label{FiniteMarionEnergy}
$(i)$ $\forall\eps>0\,\exists M_\eps\in\br:\forall f\in \cb_\eps,
\tr_\o\left(|D|^{- \dE/2}|[D,f]|^2|D|^{- \dE/2}\right)\leq M_\eps\ce[f]$.
When $0<\eps<(\dE-\a^{-1})\log2$ we may choose $M_\eps=4 e K_1 \eps^{-1}\z(\a \dE -\frac{\a\eps}{\log2}).$
\\
$(ii)$  In particular, $\displaystyle
|D|^{-\frac12 \dE}|[D,f]|^2|D|^{-\frac12 \dE}\in \cl^{1,\infty}(\ch)$, $f\in\cb$.
\end{thm}
\begin{proof}
It is enough to give the proof for $0<\eps<(\dE-\a^{-1})\log2$.
We shall use  Lemma 4.5 in \cite{FK} with the contraction $U$ given by the operator $[D,f]|D|^{-\frac12s}$ suitably normalized,  the positive operator $T$ given by $|D|^{-( \dE-s)}$, and the convex function $f(x)=x^{1+t}$, with $t>0$. Then, we take  $\dE -\frac{\eps}{\log2}\leq s<\dE$, so that  $[D,f]|D|^{-\frac12s}$ is bounded and $ |D|^{-\frac12( \dE+t( \dE-s))}|[D,f]|^2|D|^{-\frac12( \dE+t( \dE-s))} $ is trace class.
\begin{align*}
\tr\big( (|D|^{-\frac12 \dE}&|[D,f]|^2|D|^{-\frac12 \dE})^{1+t}\big)\\
&=
\tr\big( (|D|^{-\frac12( \dE-s)}|D|^{-\frac12s}|[D,f]|^2|D|^{-\frac12s}|D|^{-\frac12( \dE-s)})^{1+t}\big)\\
&=
\||D|^{-\frac12s}|[D,f]|^2|D|^{-\frac12s}\|^{1+t}\tr\big( (T^{\frac12}U^*UT^{\frac12})^{1+t}\big)\\
&=
\|[D,f]|D|^{-\frac12s}\|^{2(1+t)}\tr\big( (UTU^*)^{1+t}\big)\\
&\leq
\|[D,f]|D|^{-\frac12s}\|^{2(1+t)}\tr (UT^{1+t}U^*)\\
&=
\|[D,f]|D|^{-\frac12s}\|^{2(1+t)}\tr (T^{\frac12(1+t)}U^*UT^{\frac12(1+t)})\\
&=
\|[D,f]|D|^{-\frac12s}\|^{2t}\tr (|D|^{-\frac12( \dE-s)(1+t)}|D|^{-\frac12s}|[D,f]|^2|D|^{-\frac12s}|D|^{-\frac12( \dE-s)(1+t)})\\
&=
\|[D,f]|D|^{-\frac12s}\|^{2t}\tr (|D|^{-\frac12( \dE+t( \dE-s))}|[D,f]|^2|D|^{-\frac12( \dE+t( \dE-s))}).
\end{align*}
The previous inequality, together with  equation \eqref{zetaEst}, gives
$$
\tr\big( (|D|^{-\frac12 \dE}|[D,f]|^2|D|^{-\frac12 \dE})^{1+t}\big)
\leq \|[D,f]|D|^{-\frac12s}\|^{2t} 4K_1 \z(\a s)\ce[ f]\left(1-\frac35 2^{2- \dE-t( \dE-s)}\right)^{-1}
$$
hence, for $f\in\cb_\eps$,
$\displaystyle
\limsup_{t\to0}t\tr\big( (|D|^{-\frac12 \dE}|[D,f]|^2|D|^{-\frac12 \dE})^{1+t}\big)
\leq\frac{4K_1 \z(\a s)}{ ( \dE-s)\log2 }\ce[ f]$.
By Lemma \ref{s-range}, we may choose $s=\dE -\frac{\eps}{\log2}$, hence
\begin{equation}\label{Dix<En}
\limsup_{t\to0}t\tr\big( (|D|^{-\frac12 \dE}|[D,f]|^2|D|^{-\frac12 \dE})^{1+t}\big)
\leq 4K_1 \eps^{-1}\z(\a \dE -\frac{\a\eps}{\log2})\ce[ f].
\end{equation}
By Theorem 4.5 $(i)$ in \cite{CRSS}, we get
$$
\limsup_n\frac1{\log n}\sum_{k=1}^n\m_k(|D|^{-\frac12 \dE}|[D,f]|^2|D|^{-\frac12 \dE})
\leq  M_\eps\ce[f].
$$
%The result
$(i)$ follows by the definition of $\tr_\omega$, $(ii)$ follows by the definition of $\cl^{1,\infty}(\ch)$.
\end{proof}

\begin{lem}\label{HalfMiRoTrain}
Let $0<\d<d$, and $B$ be a densely defined, positive (possibly unbounded) operator on $\ch$, $T\in\cb(\ch)_+$ such that  $T^s\in\cl^1(\ch)$  for $s>d$, and   $T^{s/2}BT^{s/2}\in\cl^1(\ch)$  for $s>\d$. Then
\begin{equation}\label{StrResIneq}
 \limsup_{s\to\d^+}(s-\d)\tr (T^{s/2}BT^{s/2})
 \leq d\cdot \limsup_{r\to\infty}\frac1r\,\tr (T^{\d/2}BT^{\d/2})^{1+\frac1r}.
\end{equation}
If  $\displaystyle\lim_{s\to\d^+}(s-\d)\tr (T^{s/2}BT^{s/2})$ exists and $\displaystyle\limsup_{r\to\infty}\frac1r\,\tr (T^{\d/2}B^2T^{\d/2})^{1+\frac1r}$ is finite, then, for any dilation invariant state  $\o$ on  $\ell^\infty$ vanishing on $c_0$,
\begin{equation}\label{StrResIneq2}
 \lim_{s\to\d^+}(s-\d)\tr (T^{s/2}BT^{s/2})
 \leq d\cdot \tr_\o (T^{\d/2}BT^{\d/2}).
\end{equation}
\end{lem}
\begin{proof} %Let us  set $A=BT^{\d/2}$.
For $r>0$,   H\"older inequality (\cite{Dix},Thm 6) for the exponents $1+\frac1r$, $2(r+1)$, $2(r+1)$  gives
$$
\tr (T^{s/2}BT^{s/2})=\tr \big(T^{(s-\d)/2}(T^{\d/2}BT^{\d/2})T^{(s-\d)/2}\big)
\leq
\big(\tr(T^{\d/2}BT^{\d/2})^{1+\frac1r}\big)^{\frac{r}{r+1}}\big(\tr(T^{(s-\d)(r+1)})\big)^{\frac1{r+1}}.
$$
Setting $\displaystyle r=\frac{d+\eps+\d-s}{s-\d}$ for $\eps>0$, i.e. $(s-\d)(r+1)=d+\eps$, we get
\begin{align*}
\limsup_{s\to\d^+}(s-\d) \tr (T^{s/2}BT^{s/2})
&\leq\limsup_{r\to\infty}\frac{d+\eps}{r+1}
\big(\tr (T^{\d/2}BT^{\d/2})^{1+\frac1r}\big)^{\frac{r}{r+1}}\big(\tr(T^{(d+\eps)})\big)^{\frac1{r+1}}\\
&\leq(d+\eps)\limsup_{r\to\infty}\frac{r}{r+1}
\big(\frac1r\,\tr (T^{\d/2}BT^{\d/2})^{1+\frac1r}\big)^{\frac{r}{r+1}}\\
&=(d+\eps)\limsup_{r\to\infty}\frac1r\,\tr (T^{\d/2}BT^{\d/2})^{1+\frac1r}.
\end{align*}
Inequality \eqref{StrResIneq} follows by the arbitrariness of $\eps$. With the same notations as in \cite{CPS}, we may replace the $\limsup$ in the argument above with a $\tilde\o-\lim$. Then \eqref{StrResIneq2} follows by \cite{CPS}, Thm 3.1.
\end{proof}

\begin{prop}\label{SelfSimInv}
The quadratic form $f\to\tr_\o(|D|^{- \dE/2}|[D,f]|^2|D|^{- \dE/2})$ defined on  continuous functions with values in $[0,\infty]$  is self-similar with parameter $5/3$ and invariant under the symmetries of the triangle.
\end{prop}
\begin{proof}
Let us prove self-similarity.
\begin{align*}
\sum_{i=1,2,3}&\tr_\o(|D|^{- \dE/2}|[D,\pi(f\circ w_i)]|^2|D|^{- \dE/2})
=\sum_{i=1,2,3}\tr_\o\left(\bigoplus_\s|D_\s|^{- \dE/2}|[D_\s,\pi_\s(f\circ w_i)]|^2|D_\s|^{- \dE/2}\right)\\
&=\tr_\o\left(\bigoplus_{i=1,2,3}\bigoplus_\s
2^{(2- \dE)|\s|}|D_\emptyset|^{- \dE/2}|[D_\emptyset,\pi_\emptyset(f\circ w_i\circ w_\s)]|^2|D_\emptyset|^{- \dE/2}\right)\\
&=\tr_\o\left(\bigoplus_{\t\ne\emptyset}
2^{(2- \dE)(|\t|-1)}|D_\emptyset|^{- \dE/2}|[D_\emptyset,\pi_\emptyset(f\circ w_\t)]|^2|D_\emptyset|^{- \dE/2}\right)\\
&=2^{-(2- \dE)}\big(\tr_\o(|D|^{- \dE/2}|[D,\pi(f)]|^2|D|^{- \dE/2}) -
\tr_\o(|D_\emptyset|^{- \dE/2}|[D_\emptyset,\pi_\emptyset(f)]|^2|D_\emptyset|^{- \dE/2})\big)\\
&=\frac35\tr_\o(|D|^{- \dE/2}|[D,\pi(f)]|^2|D|^{- \dE/2}),
\end{align*}
where $\tr_\o(|D_\emptyset|^{- \dE/2}|[D_\emptyset,\pi_\emptyset(f)]|^2|D_\emptyset|^{- \dE/2})$  vanishes since $|D_\emptyset|^{- \dE/2}|[D_\emptyset,\pi_\emptyset(f)]|^2|D_\emptyset|^{- \dE/2}$ is trace class, as shown in Lemma \ref{traceClass}.

We now prove symmetry invariance. We observe that the symmetry group of the triangle is generated by the reflections along the axes. Hence, it is enough to show that, for a  reflection $T$ along an axis of the gasket, and setting $f^T(x)=f(Tx)$, the operator $|D|^{- \dE/2}|[D,\pi(f^T)]|^2|D|^{- \dE/2}$ and the operator $ |D|^{- \dE/2}|[D,\pi(f)]|^2|D|^{- \dE/2}$ are unitary equivalent. Since
$$
|D|^{- \dE/2}|[D,\pi(f)]|^2|D|^{- \dE/2}
=\bigoplus_{\s}
2^{(2- \dE)|\s|}|D_\emptyset|^{- \dE/2}|[D_\emptyset,\pi_\emptyset(f\circ w_\s)]|^2|D_\emptyset|^{- \dE/2},
$$
$|\s^T|=|\s|$, and $f^T\circ w_\s=f\circ w_{\s^T}\circ T$, where $\s\to \s^T$ denotes the natural action of $T$ on the word $\s$,
$$
|D|^{- \dE/2}|[D,\pi(f^T)]|^2|D|^{- \dE/2}
=\bigoplus_{\s}
2^{(2- \dE)|\s|}|D_\emptyset|^{- \dE/2}|[D_\emptyset,\pi_\emptyset(f\circ w_\s\circ T)]|^2|D_\emptyset|^{- \dE/2}.
$$
Moreover, as in \eqref{trace2},
\begin{align*}
|D_\emptyset|^{- \dE/2}&|[D_\emptyset,g]|^2|D_\emptyset|^{- \dE/2}\\
&=\begin{pmatrix}
(\partial_\a^* \partial_\a)^{-\dE/4}S^*_{g^*}S_{g^*}(\partial_\a^* \partial_\a)^{-\dE/4}&0\\
0&(\partial_\a \partial_\a^*)^{-\dE/4}S_{g}S^*_{g}(\partial_\a \partial_\a^*)^{-\dE/4}
\end{pmatrix}.
\end{align*}
Therefore, it is enough to show that, for any $g\in H^\a(\ell_{\emptyset})$,  $(\partial_\a^* \partial_\a)^{-\dE/4}$ $S^*_{g^*\circ T}S_{g^*\circ T}$ $(\partial_\a^* \partial_\a)^{-\dE/4}$ and $(\partial_\a^* \partial_\a)^{-\dE/4}$ $S^*_{g^*}S_{g^*}$ $(\partial_\a^* \partial_\a)^{-\dE/4}$ are unitary equivalent, and $(\partial_\a \partial_\a^*)^{-\dE/4}$ $S_{g\circ T}S^*_{g\circ T}$ $(\partial_\a \partial_\a^*)^{-\dE/4}$  and $(\partial_\a \partial_\a^*)^{-\dE/4}$ $S_{g}S^*_{g}$ $(\partial_\a \partial_\a^*)^{-\dE/4}$ are unitary equivalent.

Let us consider the (self-adjoint) unitary operator $U_{T,0}$ on $L^2(\ell_\emptyset)$ given by $(U_{T,0}\xi)(x)=\xi(Tx)$ and the (self-adjoint) unitary operator $U_{T,1}$ on $L^2(\ell_\emptyset\times\ell_\emptyset)$ given by $(U_{T,1}\eta)(x,y)=\eta(Tx,Ty)$. A direct computation shows that $U_{T,1}\partial_\a=\partial_\a U_{T,0}$ and $M_{g^T}=U_{T,0}M_g U_{T,0}$, hence $S_{g^T}=U_{T,1}S_gU_{T,0}$. As a consequence, for $g\in H^\a(\ell_{\emptyset})$, we get
\begin{align*}
(\partial_\a^* \partial_\a)^{-\dE/4}S^*_{g^*\circ T}S_{g^*\circ T}(\partial_\a^* \partial_\a)^{-\dE/4}
& =U_{T,0}
(\partial_\a^* \partial_\a)^{-\dE/4}S^*_{g^*}S_{g^*}(\partial_\a^* \partial_\a)^{-\dE/4}
U_{T,0} \,, \\
(\partial_\a \partial_\a^*)^{-\dE/4}S_{g\circ T}S^*_{g\circ T}(\partial_\a \partial_\a^*)^{-\dE/4}
& =U_{T,1}
(\partial_\a \partial_\a^*)^{-\dE/4}S_{g}S^*_{g}(\partial_\a \partial_\a^*)^{-\dE/4}
U_{T,1}\,,
\end{align*}
and the thesis follows.
\end{proof}

\begin{rem}
With the same argument as in the Proposition above, we may show that the quadratic form $f\to\tr_\o(|D|^{- \d/2}|[D,f]|^2|D|^{- \d/2})$ is self-similar with parameter $2^{2- \d}$. As a consequence the value $\dE$ is uniquely determined by requiring self-similarity with scaling parameter $5/3$. Hence the energy dimension $\dE$ is not just an abscissa of convergence, but is completely determined by the structure of the Dirac operator and the scaling of the energy under self-similarity.
\end{rem}

We state here a simple variant of a well known uniqueness result, cf. e.g. \cite{Sabot,CuSt}.

\begin{lem}\label{easyLemma}
A quadratic form $\cg$ which is finite on finitely harmonic functions on the gasket, vanishes only on constants, is self-similar with parameter $5/3$, and invariant under the symmetries of the triangle coincides with a multiple of the standard Dirichlet form on finitely harmonic functions.
\end{lem}
\begin{proof}
Let us consider the linear map which associates to any vector $\vec{v}=(v_0,v_1,v_2)\in\bc^3$,  the $0$-harmonic function $g=\f(\vec{v})$ on the gasket taking values $v_i$, $i=0,1,2$,  on the extreme points of the lacuna $\ell_\emptyset$.  Then, the quadratic form $\vec{v}\to\cg[g]$ is positive semidefinite on $\bc^3$, vanishes only on constant vectors, since $g=\f(\vec{v})$ is constant  iff $\vec{v}$ is constant, and is invariant under permutations of the components, because of the symmetry invariance in the assumptions.
As in the proof of Theorem \ref{thm:res=en} $(ii) (a)$, there exists a constant $k$ such that $\cg[g]=k\ce[g]$ for any $0$-harmonic function $g$.
Let now $h$ be $n$-harmonic, so that $h\circ w_\s$ is $0$-harmonic for $|\s|=n$.
Then, by self-similarity,
\begin{align*}
\cg[h]
%&
=\left(\frac53\right)^n\sum_{|\s|=n}\cg[h\circ w_\s]
%\\&
=k\left(\frac53\right)^n\sum_{|\s|=n}\ce[h\circ w_\s]=k\ce[h].
\end{align*}
The thesis follows.
\end{proof}

\begin{cor}\label{cor:KigamiEnergy}
On the algebra $\cb$, the  standard Dirichlet form $\ce$ and the quadratic form $f\to\ce_D[f]:=\tr_\o(|D|^{- \dE/2}|[D,f]|^2|D|^{- \dE/2})$ coincide up to a multiplicative constant, namely there exists a constant
 $B_\o$, which may depend on the generalized limit $\o$, such that,
 \begin{equation}\label{Dix=Dir}
 \ce_D[f] = B_\o\ce[g],\quad\forall  g\in\cb.
\end{equation}
\end{cor}
\begin{proof}
We first prove that, for $f\in \cb$, the quadratic form in the statement is bounded from above and from below by multiples of the standard Dirichlet form on the gasket. In the inequality \eqref{En<Dix<En} below the upper bound seems to depend on $\eps$, but, as soon as the statement of the theorem is proved, the smallest $M_\eps$ will work on the whole $\cb$.

Let $f\in\cb_\eps$. We  may then apply inequality \eqref{StrResIneq2} with $T=|D|^{-1}$, $B=|[D,f]|^2$, $\d= d_E$. In fact, condition
$\exists\displaystyle\lim_{s\to\d^+}(s-\d)\tr (T^{s/2}BT^{s/2})$ is satisfied by Theorem \ref{thm:res=en}, and condition $\displaystyle\limsup_{r\to\infty}\frac1r\,\tr (T^{\d/2}B^2T^{\d/2})^{1+\frac1r}<+\infty$ follows by \eqref{Dix<En}. Then, Theorems \ref{thm:res=en}, \ref{FiniteMarionEnergy} and inequality \eqref{StrResIneq2} give, for sufficiently small $\eps$,
\begin{equation}\label{En<Dix<En}
\frac{A}{d}\ce[f]\leq \ce_D[f] \leq M_\eps\ce[f].
\end{equation}
In particular, by Lemma \ref{harmonicInB}, previous inequalities hold for finitely harmonic functions.

Then, we observe that \eqref{Dix=Dir} holds for finitely harmonic functions. Indeed, by inequality \eqref{En<Dix<En} and Proposition \ref{SelfSimInv}, the assumptions of Lemma \ref{easyLemma} are satisfied.

Finally, we may proceed as in the proof of Theorem \ref{thm:res=en} $(ii) (c)$.
Choose $f\in\cb_\eps$ and let $g$ be a finitely harmonic function. Then,
\begin{align*}
|\ce^{1/2}_D[f]& - B^{1/2}_\o \ce^{1/2}[f]|\\
&\leq |\ce^{1/2}_D[f] -  \ce^{1/2}_D[g]|+|\ce^{1/2}_D[g] - B^{1/2}_\o \ce^{1/2}[g]|+B^{1/2}_\o|\ce^{1/2}[g] -  \ce^{1/2}[f]|\\
&\leq \ce^{1/2}_D[f-g] +B^{1/2}_\o\ce^{1/2}[f-g] \leq (M^{1/2}_\eps+B^{1/2}_\o)\ce^{1/2}[f-g].
\end{align*}
Since $g$ varies among finitely-harmonic functions, the last term may be made arbitrarily small, and the Theorem is proved.
\end{proof}

% \newpage

\section{A two parameter deformation of spectral triple: the $(\b,\a)$ plane}\label{betaalpha}

In this subsection we consider a deformation of our construction. Namely, after having chosen $\alpha \in (0,1]$, we introduce a further parameter $\b>0$, and, for any $\sigma\in \cup_n\{0,1,2\}^n$, consider the triple $(\cc(K),\ch_\sigma,D_\sigma)$, where $\ch_\sigma=\ch_\emptyset$, $D_\sigma=2^{\b |\sigma|}D_\emptyset$, and the algebra $\cc(K)$ acts via the representation $\pi_\sigma$, with $\pi_\sigma(f)=\pi_\emptyset(f\circ w_\sigma)$. Here, we only give the statements concerning this two-parameter family of spectral triples. On the one hand, the proofs are more or less direct extensions of the proofs for the case $\b=1$. On the other hand, detailed proofs are contained in a previous version of this paper, available on the arXiv as http://arxiv.org/abs/1112.6401v2.

\begin{dfn}\label{dfn:thebetatriple}
Let us consider for values of the parameters $\alpha\in (0,1]$ and $\beta>0$, the triple $(\ca,\ch,D)$, where the Hilbert space and the self-adjoint operator are given by $\ch=\oplus_{\s\in\Sigma}\ch_\s$, $D=\oplus_{\s\in\Sigma}D_\s$ and $\ca$ is the subalgebra of $\cc(K)$ consisting of functions with bounded commutator with $D$, acting on $\ch$ via the representation $\pi=\oplus_{\s\in\Sigma}\pi_\s$.

As customary in Noncommutative Geometry, we denote by $\oint f:=tr_\omega(f|D|^{-d})$ the functional on the algebra $\cc(K)$ obtained through the Dixmier logarithmic trace $tr_\omega$, where $d$ is the abscissa of convergence of the function $s\to \tr(|D|^{-s})$.
\end{dfn}

As first result in this section, we describe how the dimensional spectrum of the triple depends on $\b>0$ and that, independently on $0<\b<\a \frac{\log 3}{\log 2}$, the measure induced on $K$ by the functional $f\mapsto \oint f$ is a suitable multiple of its Hausdorff measure.

\begin{thm}\label{thm:betavolmeas}
The volume zeta function $\mathcal{Z}_D$ of the triple $(\ca,\ch,D)$, i.e. the meromorphic extension of the function $\mathbb{C}\ni s\mapsto \tr(|D|^{-s})$ initially defined for $s\in\mathbb{C}$ having large real part, is given by
$$
\mathcal{Z}_D (s)=\frac{4\z(\a s)}{1-3\cdot 2^{-\b s}}\, ,
$$
where  $\z$ denotes the Riemann zeta function. Therefore, the dimensional spectrum  of the triple is
$$
\mathcal{S}_{\it dim} =\{\a^{-1}\}\cup \bigg\{ \frac{\log3}{\b \log 2}\left(1+\frac{2\pi i}{\log 3}k\right): k\in \mathbb{Z} \bigg\} \subset \mathbb{C}\, .
$$
As a consequence, the metric dimension $d_D$ of the spectral triple $(\ca,\ch,D)$, namely the abscissa of convergence of its volume zeta function, is $\max \{ \a^{-1},\b^{-1}d_H\}$.

When $0<\b<\a\, d_H$, i.e. $d_D=\b^{-1}d_H$, $\mathcal{Z}_D$ has a simple pole in $d_D$, and the measure associated via Riesz theorem with the functional $\cc(K)\ni f\to\oint f$ coincides with a multiple of the Hausdorff measure $H_{\dH}$:
\[
\vol(f) \equiv \int_K f\,d\,{\rm vol}:=tr_\omega(f|D|^{-d_D})=\frac{4d_D}{\log 3 }\frac{\z(d_D)}{(2\pi)^{d_D}} \int_K f\, dH_{\dH}\qquad f\in C(K).
\]
\end{thm}

The next result says that for $\a,\b\in(0,1]$ the above triple is a spectral triple according to Connes \cite{Co}, the associated seminorm is a Lip-morm on $\cc(K)$ in the sense of Rieffel \cite{Rief} and the induced topology on $K$ coincides with the original one.
\par\noindent
If moreover $\a<\b<1$, we obtain that the associated Connes metric is bi-Lipschitz w.r.t. the root $(\r_{geo})^\b$ of the geodesic metric on $K$.

\begin{cor}\label{cor:betatriple}
For any $\a,\b\in(0,1]$  the triple $(\ca,\ch,D)$ is a finitely summable spectral triple, and the seminorm $f\to\|[D,f]\|$ is a Lip-norm according to Rieffel  \cite{Rief}. Therefore, the Connes metric
$$
\r_D(x,y)= \sup\{ |f(x)-f(y)| :f\in\ca\, ,\|[D,f]\|\le 1\}
$$
induces the original topology on $K$. Let $\r_{geo}$ denote the Euclidean geodesic metric on $K$. Then, if $\b>\a$, the seminorm $\|[D,f]\|$ and the H\"older seminorm $\|f\|_{C^{0,\b}(K)}$ are equivalent, and the metric $\r_D$ is bi-Lipschitz w.r.t. the  metric $(\r_{geo})^\b$ on $K$.
%\\
%In particular, if $\b=1$, the metric $\r_D$ is bi-Lipschitz w.r.t.  $\r_{geo}$.
%
\end{cor}

Let $\a,\b\in(0,1]$, and denote by $(\ca,\ch,D)$   the spectral triple for the gasket considered above. Let $F$ be the phase of $D$, and, for any $\s\in\Sigma$, denote by $\g_\s,\eps_\s,S_\s$, a copy of the operators $\g,\eps,S$ associated, by Proposition \ref{FrModS1}, to the lacuna $\ell_\s$, identified with $\T$. Finally, let $\g=\oplus_{\s\in\Sigma} \g_\s$,  $\eps =\oplus_{\s\in\Sigma} \eps_\s$,  $S = \oplus_{\s\in\Sigma} S_\s$.

\begin{thm}
The quintuple $\cf = (\ch,\pi,F,\g,\eps)$ is a tamely degenerate 1-graded Fredholm module on $\ca$.
The ungraded Fredholm module $\cf^+ = (\ch^+,\pi^+,F^+)$ associated to it by Proposition \ref{Ex8.8.4} is a tamely degenerate module on $\ca$, and $F^+= S$.
The module $\cf^+$ is non trivial in K-homology. In particular, it pairs non trivially with the generators of the (odd) K-theory of the gasket associated with the lacunas.
\end{thm}

As a last result in this section, we compute how the energy dimension of the triple depends on $\b>0$ and that, independently of the values of $\a$ and $\b$ in suitable ranges, the induced quadratic form on $K$, defined as the residue of a suitable energy functional, is a  multiple of its standard Dirichlet form.

\begin{thm}\label{thm:res=betaen}
Assume as above that $\b>0$, $\frac12<\a\leq\a_0$, with $\a_0=\frac{\log(10/3)}{\log4}\approx 0.87$, and assume $f$ has finite energy
\footnote{The conditions $\b>0$ and $2-\frac{\log 5/3}{\b\log 2}>\a^{-1}$ indeed imply $\a>\frac12$. %and $\b(2-\a^{-1})>\frac{\log(5/3)}{\log 2}$.
}.
Then  the abscissa of convergence of the energy functional
\[
\tr(|D|^{-s/2}|[D,f]|^2|D|^{-s/2})
\]
is given by $\deD = \max \{ \a^{-1},2-\frac{\log 5/3}{\b\log 2}\}$.
Whenever  $\b(2-\a^{-1})>\frac{\log(5/3)}{\log 2}$, the energy functional has a simple pole at  $\deD$ and its residue coincides with the value of the standard Dirichlet form on $f$
\begin{equation}\label{energybeta}
{\rm Res}_{s=\d_D}\tr(|D|^{-s/2}|[D,f]|^2|D|^{-s/2})
=\const\cdot \ce[f].
\end{equation}
\end{thm}
Finally, notice that, up to a multiplicative  constant, the residue of the energy functional coincides with the standard Dirichlet form for any $\b$ in a suitable range, so that Corollary \ref{cor:KigamiEnergy} remains true as it is.
\begin{rem}
We note here that $\b$ rescales the Euclidean geodesic metric $\rho_{geo}$ of $K$ essentially to $\rho_{geo}^\b$, therefore it has to be expected that the associated metric dimension scales from $d_H$ to $\b^{-1}d_H$, while the corresponding volume measure remains the same up to a multiplicative constant. A similar effect occurs for the energy: the energy dimension passes from $2-\frac{\log 5/3}{\log 2}$ to $2-\b^{-1}\frac{\log 5/3}{\log 2}$, the energy form remaining the same up to a multiplicative constant.
\end{rem}
\vskip0.5truecm
Let us summarize the properties of the family of triples introduced above, for the different values of the parameters $\a$ and $\b$.
\vskip0.2truecm
\begin{itemize}
\item The whole construction makes sense only if $0<\a\leq1$, $\b\in\br$.
\item If $\b>0$, the inverse of $D$ on the orthogonal complement of the kernel is compact.
\item if $\b>0$ and $\a>\b/d_H$, the noncommutative volume measure coincides (up to a constant factor) with the Hausdorff measure $H_{d_H}$, $d_H=\frac{\log3}{\log2}$ being the Hausdorff (or similarity) dimension. The metric dimension is $d_D=\frac{d_H}{\b}$.
\item If $0<\b\leq1$, $\|[D,f]\|$ is a densely defined Lip-norm, $(\ca,\ch,D)$ is a spectral triple and $(\pi,\ch,F,\gamma,\eps)$ is a 1-graded Fredholm module. The latter has non-trivial pairing with the topological K-theory group $K^{1}(K)$ of the gasket.
%\item If $0<\a<\b\le 1$, the Connes metric $\r_D$ is bi-Lipschitz w.r.t. $(\r_{geo})^\b$. Moreover, if $\b=1$ and $0.79\approx \frac{\log2}{\log(12/5)}<\a\leq\frac{\log(10/3)}{\log4}\approx 0.87$,  the dimension $d_D$ coincides with the Hausdorff dimension of the gasket and $\r_D$ is bi-Lipschitz with the Euclidean geodesic distance $\r_{geo}$.
\item If $d_{E,\b}^{-1} <\a\leq\a_0<\b\leq1$, where $d_{E,\b}=2-\frac{\log(5/3)}{\log 2}\b^{-1}$, then the residue at $s=d_{E,\b}$ of the energy functional $\tr (|D|^{-s/2}[D,f]^2|D|^{-s/2})$ coincides, up to a multiplicative factor, with  the standard Dirichlet form $\ce[f]$, for all finite energy functions.
\end{itemize}

     \begin{figure}[hbpt]
 	 \centering
	 \psfig{file=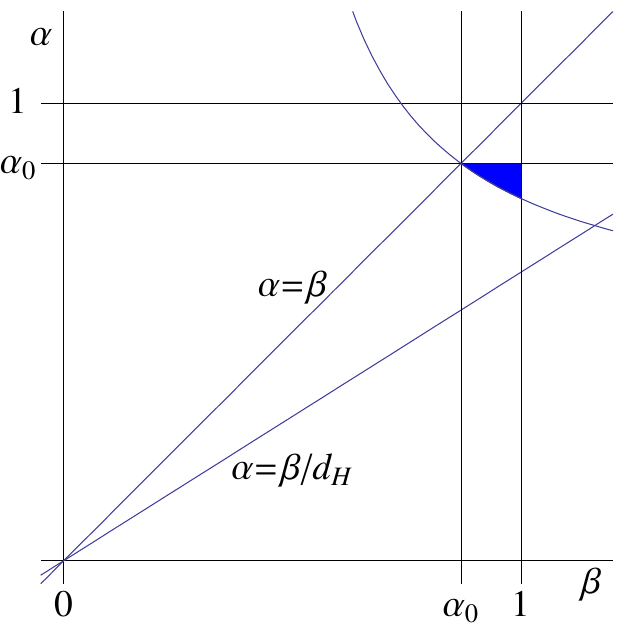,height=4.0in
	 }
	  \caption{The $(\b,\a)$ plane}
	 \label{GammaAlpha}
     \end{figure}

\newpage

\appendix

\section{Estimates on the Clausen function}

%\subsection{Polylogarithms}

According to (\cite{Lewin}, p. 236, \cite{Erdelyi} section 1.11) the analytic extension of the polylogarithm function of order $s$  on the whole complex plane with the line $[1,+\infty)$ removed is given by
\begin{align*}
	\Li_s(z) = - \frac{z\G(1-s)}{2\pi i}\, \int_\g \frac{(-t)^{s-1}}{e^t-z}\, dt,
\end{align*}
where $\g$ is a path as in figure \ref{FigPath}.
% figura
     \begin{figure}[ht]
 	 \centering
	 \psfig{file=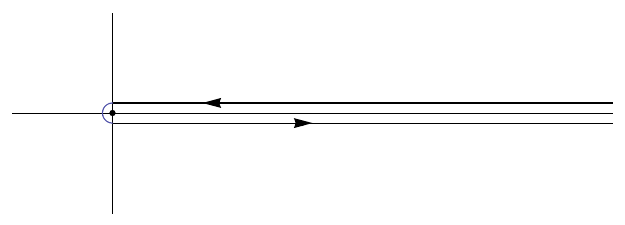,height=1.2in}
	 \caption{Path used for the analytic extension of polylogarithm.}
	 \label{FigPath}
     \end{figure}
%If $\re s \in (0,1)$, we have $\lim_{z\to 1} (1-z)^{1-s} \Li_s(z) = \G(1-s)$.

Therefore the  Clausen cosine function $\Ci_s(t)$ can be defined as
\begin{align*}
	\Ci_s(\th) = - \re \frac{\G(1-s)}{2\pi i}\, \int_\g \frac{(-t)^{s-1}}{e^{t-i\th}-1}\, dt,
\end{align*}

\begin{lem}\label{Clausen-estimate}
When $\re s<1$, $\displaystyle\lim_{t\to0^\pm}|t|^{1-s}\Li_s(e^{it})=\Gamma(1-s)e^{\pm i\pi(1-s)/2}$, as a consequence, for $\a\in(0,1]$,
$$
\lim_{t\to0}|t|^{1+2\a}\Ci_{-2\a}(t)=-\Gamma(1+2\a)\sin\pi\a.
$$
Moreover, when $\a\in[\frac12,1)$ and $|t|\leq\frac\pi4$, $\Ci_{-2\a}$ is strictly negative, and
\begin{align}
&|\Ci_{-2\a}(t)+\Gamma(1+2\a)\sin\pi\a\,|t|^{-(2\a+1)}|
\leq\frac{31}{2\pi^2}\Gamma(1+2\a)\sin\pi\a,\label{eqClausen1}\\
&\frac1{32}\sin(\pi\a)\Gamma(1+2\a)
\leq-\Ci_{-2\a}(t)|t|^{2\a+1}
\leq \frac{63}{32}\sin(\pi\a)\Gamma(1+2\a).\label{eqClausen2}
\end{align}
Finally, when $|t|\geq\pi/4$,
$$
|\Ci_{-2\a}(t)|\,|t|^{2\a+1}\leq23.
$$
\end{lem}

\begin{proof}
Let $0<|\th|\leq\pi/4$. Then we may choose $\g$ in figure \ref{FigPath} as  $\g_0-\s$ where $\g_0$ is made of the half lines $\sqrt{\pi^2-\eps^2}+t\pm i\eps$, $t>0$, and (most of) the circle of radius $\pi$ centered at the origin, and $\s$ is a suitably small positively oriented cycle surrounding the point $i\th$. Then
\begin{align*}
\Li_s(e^{i\th})
&= - \frac{e^{i\th}\G(1-s)}{2\pi i}\, \int_{\g_0} \frac{(-t)^{s-1}}{e^t-e^{i\th}}\, dt
+ \frac{e^{i\th}\G(1-s)}{2\pi i}\, \int_{\s} \frac{(-t)^{s-1}}{e^t-e^{i\th}}\, dt\\
&= - \frac{\G(1-s)}{2\pi i}\, \int_{\g_0} \frac{(-t)^{s-1}}{e^{(t-i\th)}-1}\, dt
+ \G(1-s)\, {\rm Res}_{t=i\th} \frac{(-t)^{s-1}}{e^{(t-i\th)}-1}\\
&= - \frac{\G(1-s)}{2\pi i}\, \int_{\g_0} \frac{(-t)^{s-1}}{e^{(t-i\th)}-1}\, dt
+ \Gamma(1-s)e^{i\sgn(\th)\pi(1-s)/2}|\th|^{s-1}.
\end{align*}
In particular,
\begin{align*}
\Ci_{-2\a}(\th)
&= \re \Li_{-2\a}(e^{i\th})
=- \frac{\G(1+2\a)}{2\pi}\,\im \left(\int_{\g_0} \frac{(-t)^{-(1+2\a)}}{e^{(t-i\th)}-1}\, dt\right)
- \Gamma(1+2\a)\sin\pi\a\ |\th|^{-(2\a+1)},
\end{align*}
from which the first relations hold. Moreover,
\begin{align*}
-\Ci_{-2\a}(\th)
&=  \Gamma(1+2\a)\sin\a\pi\ |\th|^{-2\a-1}
+\frac{\G(1+2\a)}{2\pi}\,\im \left(\int_{\g_0} \frac{(-t)^{-(1+2\a)}}{e^{(t-i\th)}-1}\, dt\right),
\end{align*}
hence
\begin{align*}
|\Ci_{-2\a}(\th)+\Gamma(1+2\a)\sin\a\pi\ |\th|^{-2\a-1}|
&=  \frac{\G(1+2\a)}{2\pi}\,\left|\im \left(\int_{\g_0} \frac{(-t)^{-(1+2\a)}}{e^{(t-i\th)}-1}\, dt\right)\right|.
\end{align*}
We now assume $\a\geq\frac12$, and observe that the part of the path constituted by the half lines $\sqrt{\pi^2-\eps^2}+t\pm i\eps$, $t>0$ is invariant under reflection w.r.t. to the real axis, which sends the variable of integration to its conjugate. Therefore,
\begin{align*}
\int_{\g_0} \frac{(-t)^{-(1+2\a)}}{e^{(t-i\th)}-1}\, dt
&=
\int_{|z|=\pi} \frac{(-z)^{-(1+2\a)}}{e^{(z-i\th)}-1}\, dz
+
2i\sin(2\pi\a)\int_{\pi}^\infty \frac{t^{-(1+2\a)}}{e^{(t-i\th)}-1}\, dt\ .
\end{align*}
As for the second integral, we have
\begin{align}
\bigg| 2\sin(2\pi\a)\ &\im\left(i\int_{\pi}^\infty \frac{t^{-(1+2\a)}}{e^{(t-i\th)}-1}\, dt \right) \bigg|\notag\\
&=
2|\sin(2\pi\a)|\ \bigg|\int_{\pi}^\infty \frac{
\big(e^{t}\cos\th -1\big)
t^{-(1+2\a)}
}
{|e^{(t-i\th)}-1|^2}\, dt \bigg|\notag\\
&\leq
2|\sin(2\pi\a)|\,\pi^{-(1+2\a)}\int_{\pi}^\infty \frac{
e^{t}
}
{e^{2t}+1-2e^t\cos\th}\, dt\notag\\
&\leq
2|\sin(2\pi\a)|\,\pi^{-(1+2\a)}\int_{\pi}^\infty \frac{
e^{t}
}
{(e^{t}-1)^2}\, dt\notag\\
&\leq
4|\sin(\pi\a)|\,\pi^{-(1+2\a)}(e^\pi-1)^{-1},\label{eq:A}
\end{align}
where in the first inequality we used $|\th|\leq\frac\pi4$, which implies $|e^{t}\cos\th -1|\leq e^t$ for $t\geq\pi$.
We now come to the first integral. Let us observe that, when $\a\in\bz$, it is a contour integral of a meromorphic function, therefore it may be computed via residues. In particular, when $\a\ne0$, the only residue comes from $z=i\th$, whose real part vanishes, as shown above. To get an estimate which is small for $\a$ close to 1, we set
$$
\psi(\a,\th)=\im\left(  \int_{|z|=\pi} \frac{(-z)^{-(1+2\a)}}{e^{(z-i\th)}-1}\, dz\right),
$$
so that we have
\begin{equation}\label{eq:B1}
|\psi(\a,\th)|=\left|\int_\a^1\frac{\partial\psi}{\partial\a}(s,\th)ds\right|
\leq
\int_\a^1\left|\frac{\partial\psi}{\partial\a}(s,\th)\right|ds \leq
(1-\a)\sup_{\a\leq s\leq1}\left|\frac{\partial\psi}{\partial \a}(s,\th)\right|\,.
\end{equation}
Moreover,
\begin{align*}
|\partial_\a\psi(s,\th)|
&=\left|
\im\left(  \int_{|z|=\pi}-2\log(-z) \frac{(-z)^{-(1+2s)}}{e^{(z-i\th)}-1}\, dz\right)
\right|
\\
&=\left|
\im\left(
\int_{0}^{2\pi}  -2 (\log \pi + i(t-\pi) )\pi^{-2s}
\frac{e^{-i(t-\pi)(1+2s)}}{e^{(\pi e^{it}-i\th)}-1}
i\,e^{it}\, dt\right)\right|\\
&\leq4\pi^{1-2\a} (\log\pi + \pi)
\left(\min_{t\in[0,2\pi]} | e^{(\pi e^{it}-i\th)}-1 |^2\right)^{-1/2}\, ,\quad \a\leq s\leq1.
\end{align*}
We now consider the two cases $\pi|\sin t|\leq |\th|+\pi/2$, and $\pi|\sin t|\geq |\th|+\pi/2$.
\\
If $\pi|\sin t|\leq |\th|+\pi/2$, $|\cos t|\geq(1-(|\th|/\pi+1/2)^2)^{1/2}$, and
\begin{align*}
| e^{(\pi e^{it}-i\th)}-1 |^2
&=e^{2\pi\cos t} + 1 - 2e^{\pi\cos t}\cos(\pi\sin t - \th)\\
&\geq (e^{\pi\cos t} - 1)^2
\geq (1-e^{-(\pi^2-(|\th|+\pi/2)^2)^{1/2}} )^2.
\end{align*}
If $\pi|\sin t|\geq |\th|+\pi/2$, $\frac32\pi\geq|\pi\sin t - \th|\geq|\pi\sin t| - |\th|\geq\pi/2$, therefore $\cos(\pi\sin t - \th)\leq0$, and
\begin{align*}
| e^{(e^{it}-i\th)}-1 |^2
&=e^{2\pi\cos t} + 1 - 2e^{\pi\cos t}\cos(\pi\sin t - \th)\geq1.
\end{align*}
We have proved that
\begin{equation}\label{eq:B2}
|\psi(\a,\th)|\leq  4(1-\a)\pi^{1-2\a} (\log\pi + \pi)(1-e^{-(\pi^2-(|\th|+\pi/2)^2)^{1/2}} )^{-1}\,,
\end{equation}
hence, by inequalities \eqref{eq:A}, \eqref{eq:B1}, \eqref{eq:B2}, and since $\a\geq1/2$ implies $2(1-\a)\leq \sin\pi\a$,
\begin{align}\label{eq:C}
\left|\im \left(\int_{\g_0} \frac{(-t)^{-(1+2\a)}}{e^{(t-i\th)}-1}\, dt\right)\right|
&\leq \frac{2\sin(\pi\a)}{\pi^{(1+2\a)}}\Big(
\frac2{e^\pi-1}
+
\frac{\pi^2(\log\pi + \pi)}{1-e^{-(\pi^2-(|\th|+\pi/2)^2)^{1/2}} }
 \Big)\,.
\end{align}
Then,
\begin{align*}
\left|\frac{\Ci_{-2\a}(\th)|\th|^{2\a+1}}{\sin(\pi\a)\Gamma(1+2\a)}+1\right|
&\leq  \frac{|\th|^{2\a+1}}{2\pi\sin(\pi\a)}\,\left|\im \left(\int_{\g_0} \frac{(-t)^{-(1+2\a)}}{e^{(t-i\th)}-1}\, dt\right)\right|
\leq \left(\frac{|\th|}\pi\right)^{2\a+1}h\left(\frac{|\th|}\pi\right),
\end{align*}
where the function $h(r)$, $r\in(0,1/4]$ is given by
$$
h(r)=\frac{ \pi(\log\pi + \pi)}{1-\exp\left(-\pi\sqrt{1-\left(r+1/2\right)^2}\right) }
+\frac{2}{\pi(e^\pi-1)}.
$$
Since $h$ is increasing, it attains its maximum for $r=\frac14$, where $h(\frac14)<\frac{31}2$. Hence,
\begin{equation}
\left(1-\frac{31}2\left(\frac{|\th|}{\pi}\right)^{2\a+1}\right)
\leq\frac{-\Ci_{-2\a}(\th)|\th|^{2\a+1}}{\sin(\pi\a)\Gamma(1+2\a)}
\leq \left(1+\frac{31}2\left(\frac{|\th|}{\pi}\right)^{2\a+1}\right),
\end{equation}
which implies \eqref{eqClausen1}.
Now, since $|\th|\leq\pi/4$ and $\a\geq1/2$, we get
 $\frac{31}2(|\th|/\pi)^{2\a+1}\leq \frac{31}{32}$, hence
$$
\frac1{32}\sin(\pi\a)\Gamma(1+2\a)
\leq-\Ci_{-2\a}(\th)|\th|^{2\a+1}
\leq \frac{63}{32}\sin(\pi\a)\Gamma(1+2\a),
$$
showing in particular that $-\Ci_{-2\a}(\th)$ is strictly positive for $|\th|\leq\pi/4$.

\bigskip

We finally estimate $|\Ci_{-2\a}(\th)|$ for $|\th|\geq\frac\pi4$. We simply choose the contour $\g$ as the circle of radius $\l |\th|$ around the origin and the half lines $\sqrt{ \l^2\th^2-\eps^2}+t\pm i\eps$, $t>0$, for $\frac12<\l<1$. As for the first integral, we get
$$
\left|  \int_{|z|=\l|\th|} \frac{(-z)^{-(1+2\a)}}{e^{(z-i\th)}-1}\, dz\right|
\leq2\pi\l|\th|(\l|\th|)^{-(2\a+1)}
\left(\min_{t\in[0,2\pi]} | e^{ \l |\th| e^{it}-i\th }-1 |^2\right)^{-1/2}\,.
$$
Since $|\l\sin t - 1|\,|\th|\geq (1-\l)\,|\th|$, we get $\cos((\l\sin t - 1)\th)\leq\cos((1-\l)\,\th)$, therefore
\begin{align*}
| e^{\l |\th| e^{it}-i\th}-1 |^2
&=
e^{2\l |\th| \cos t}+1-2e^{\l |\th| \cos t}\cos((\l\sin t-1)\th)\\
&\geq
e^{2\l |\th| \cos t}+1-2e^{\l |\th| \cos t}\cos((1-\l)\th)\\
&\geq\sin^2[(1-\l)\,|\th|].
\end{align*}
As a consequence,
$$
\left|  \int_{|z|=\l|\th|} \frac{(-z)^{-(1+2\a)}}{e^{(z-i\th)}-1}\, dz\right|
\leq (\l|\th|)^{-(2\a+1)}
\frac{2\pi\l |\th|} {\sin[(1-\l)\,|\th|]}
\leq (\l|\th|)^{-(2\a+1)}
\frac{2\l\pi^2} {\sin[(1-\l)\pi]}
$$
The second integral is estimated, as above, by
\begin{align*}
\bigg| 2i\sin(2\pi\a)\ & \int_{\l |\th|}^\infty \frac{ t^{-(1+2\a)} }{ e^{(t-i\th)}-1 }\, dt  \bigg|
\leq 4\sin(\pi\a)(\l|\th|)^{-(1+2\a)} \int_{\l |\th|}^\infty \frac{e^t +1}{(e^t-1)^2}\, dt\\
& \leq 8\sin(\pi\a)(\l|\th|)^{-(1+2\a)} (e^{\l|\th|} -1)^{-1}
\leq 8(\l|\th|)^{-(1+2\a)} (e^{\frac{\l\pi}{4}} -1)^{-1}
\end{align*}
whence
\begin{align*}
|\Ci_{-2\a}(\th)|
\leq
\frac{\G(1+2\a)}{2\pi}(\l|\th|)^{-(2\a+1)}
\left(\frac{2\l\pi^2} {\sin[(1-\l)\pi]}+\frac8{ e^{ \frac{\l\pi}{4} } -1 }     \right).
\end{align*}
Finally, for any $\l\in(0,1)$,
\begin{align*}
|\Ci_{-2\a}(\th)|\,|\th|^{2\a+1}
&\leq \frac{\G(1+2\a)}{2\pi}\l^{-(2\a+1)}
\left(\frac{2\l\pi^2} {\sin[(1-\l)\pi]}+\frac8{ e^{ \frac{\l\pi}{4} } -1 }     \right)
\\
&\leq \l^{-3}
\left(\frac{2\l\pi} {\sin[(1-\l)\pi]}+\frac8{\pi( e^{ \frac{\l\pi}{4} } -1 )}     \right).
\end{align*}
Suitably choosing $\l$, one gets $|\Ci_{-2\a}(\th)|\,|\th|^{2\a+1}<23$.
\end{proof}

\begin{prop}\label{Prop:pairing}
Let $f$ be an even $\cc^\infty$ function on $\T$ vanishing in $0$. Then, for $\a\in(0,1)$,
$$
\int_{-\pi}^\pi\Ci_{-2\a}(\th)f(\th)\,d\th=\pi (\{k^{2\a}\},\{f_k\})_{\ell^2(\bn)},
$$
where $f_k=\frac1{\pi}\int_{-\pi}^{\pi}\cos(k\th)f(\th)\,d\th$.
\end{prop}
\begin{proof}
Let us set $\Ci(s,\r,\th):=\re (\Li_s(\r e^{i\th}))$. Then, reasoning as in the proof of Lemma \ref{Clausen-estimate}, it is not difficult to show that
\begin{equation}\label{unifEstimate}
\sup_{\r\in[0,1],|\th|\leq\pi}|\Ci(-2\a,\r,\th)|\th|^{2\a+1}|<\infty.
\end{equation}
We may assume that $f$ is real valued, namely $f(\th)=\sum_{k\geq0}f_k\cos(k\th)$, with $f_k\in\br$. The other properties of $f$ amount to $f_k$ rapidly decreasing and $\sum_kf_k=0$. Since $f$ is even, it has indeed a zero of order 2 in $\th=0$, hence, by \eqref{unifEstimate}, $\Ci(-2\a,\r,\th)f(\th)$ is uniformly $L^1(\th)$, for $\r\in[0,1]$. By dominated convergence,
\begin{align*}
\int_{-\pi}^\pi\Ci_{-2\a}(\th)f(\th)\,d\th
&=\lim_{\r\to1}\int_{-\pi}^\pi\Ci(-2\a,\r,\th)f(\th)\,d\th\\
&=\lim_{\r\to1}\re\left(
-i \int_{|z|=\r}\Li_{-2\a}(z)\left(\frac12\sum_{k=0}^\infty f_k(\r^{-k}z^k+\r^k z^{-k})\right)\,\frac{dz}{z}
\right)\\
&=\lim_{\r\to1}\re\left(
\frac{-i}2\sum_{k=0}^\infty f_k\r^k\sum_{n=1}^\infty n^{2\a}\int_{|z|=\r}z^{n-k}\,\frac{dz}{z}
\right)\\
&=\pi\lim_{\r\to1}\sum_{k=0}^\infty \r^k f_k\ k^{2\a}
=\pi\sum_{k=0}^\infty f_k\ k^{2\a}.
\end{align*}
\end{proof}

\begin{prop}\label{Prop:Holder}
Let $\a\in(0,1)$, and consider  the seminorm $\pa(f)$, $ f\in C(\T)$, given by
\[
\pa(f)^2 = \frac1{2\pi}\, \sup_{x\in\T}\int_\T \,\f_{\a}(x-y)|f(x)-f(y)|^2 dy<+\infty,
\]
where $\f_\a(t)=- 2\pi\Ci_{-2\a}(t)$, and denote by $\|f\|_{0,\a} = \displaystyle\sup_{x,y}\frac{|f(x)-f(y)|}{d(x,y)^\a}$ the H\"older seminorm.
Then,
\item[$(i)$] $\forall\eps>0$,
$\pa(f)\leq c_\eps\|f\|_{0,\a+\eps}$, where  $c_\eps=  \frac{1}{\sqrt{\eps} } \big( \frac\pi4 \big)^{\eps}  \left( 4 + 23(4^{2\eps}-1)\right)^{1/2}$,

\item[$(ii)$] for $\a\geq \frac12$,
%\forall\eps>0\ \exists\, c,c_\eps>0:
$\tilde{c}_\a \|f\|_{0,\a}\leq \pa(f)$,
where
$\tilde{c}_\a=\frac{\sqrt{3\sin(\pi\a)}}{16\sqrt2}$.
%As a consequence, $C^{0,\a+\eps}(\T) \subset \mathcal{A}_\alpha\subset C^{0,\a}(\T)$.
\end{prop}

\begin{proof}
\item[$(i)$] If $f$ is $(\a+\eps)$-H\"older  then
\begin{align*}
\sup_{x\in\T}\int_\T \,\f_{\a}(x-y)|f(x)-f(y)|^2 dy
&\leq
\|f\|^2_{0,\a+\eps}\sup_{x\in\T}\int_\T \f_{\a}(x-y) d(x,y)^{2(\a+\eps)}\, dy\\
&=2\|f\|^2_{0,\a+\eps}\int_{0}^\pi \f_{\a}(t) t^{2(\a+\eps)}\, dt\, .
\end{align*}
Making use of the estimates in Lemma \ref{Clausen-estimate}, one gets
\begin{align*}
\int_{0}^\pi \f_{\a}(t) t^{2(\a+\eps)}\, dt
&=\int_{0}^{\pi/4} \f_{\a}(t) t^{2(\a+\eps)}\, dt+\int_{\pi/4}^\pi \f_{\a}(t) t^{2(\a+\eps)}\, dt\\
&\leq 2\pi\frac{63}{32} \sin(\pi\a)\Gamma(1+2\a)\int_{0}^{\pi/4}  t^{2\eps-1}\, dt
+2\pi \cdot 23\int_{\pi/4}^\pi t^{2\eps-1}\, dt \\
&\leq 4\pi \sin(\pi\a)\Gamma(1+2\a)\frac{(\pi/4)^{2\eps}}{2\eps}
+46\pi\frac{\pi^{2\eps}-(\pi/4)^{2\eps}}{2\eps}\\
&\leq \frac\pi\eps \Big( \frac\pi4 \Big)^{2\eps}
\left(4 + 23(4^{2\eps}-1)\right).
\end{align*}

\item[$(ii)$] Assume $\pa(f)<\infty$,  let $x,y\in\T$, and denote by $\s$ the distance between $x$ and $y$, and by $I_\s$ the arc of length $\s$ with end-points $x$ and $y$.
By Lemma \ref{Clausen-estimate}, $\f_\a(t)>0$ for $|t|\leq\frac\pi4$.

Then, for $\s\leq\frac\pi4$,
\begin{align*}
	\bigg| f(x)-\frac1\s\, \int_{I_\s} f(z)\, dz \bigg|
	& \leq \frac1\s\, \int_{I_\s} | f(x) - f(z)| \, dz
	\\
	& = \s^{-1} \, \int_{I_\s} | f(x) - f(z)|\f_\a(x-z)^{1/2} \f_\a(x-z)^{-1/2} \, dz \\
	& \leq  \s^{-1} \, \pa(f) \sqrt{2\pi} \cdot \bigg( \int_0^\s  \frac{1}{\f_\a(t)t^{1+2\a}}{t^{1+2\a}}\, dt \bigg)^{\frac12}\\
	& \leq (2+2\a)^{-1/2} \bigg( \sup_{0< t\leq\frac\pi4}  \frac{2\pi}{\f_\a(t)t^{1+2\a}} \bigg)^{1/2}\s^{\a}\,  \pa(f)\\
	& \leq (2+2\a)^{-1/2} \bigg(   \frac{32}{  \sin(\pi\a)\Gamma(1+2\a) } \bigg)^{1/2}\s^{\a}\,  \pa(f)\\
	& \leq 4\big((1+\a)\sin(\pi\a)\Gamma(1+2\a)\big)^{-1/2} \,  \s^{\a}\pa(f).
\end{align*}
	Therefore, using the triangle inequality we obtain, for all $x,y$, such that $d(x,y)\leq\frac\pi4$,
	$$
	 \frac{ |f(x)-f(y)|}{|x-y|^{\a}} \leq \frac{8}{\sqrt{(1+\a)\sin(\pi\a)\Gamma(1+2\a)}}\,  \pa(f).
	$$
A direct computation then shows
$$
\|f\|_{0,\a}\leq \frac{32\cdot 4^{-\a}}{\sqrt{(1+\a)\sin(\pi\a)\Gamma(1+2\a)}}\,\pa(f)
\leq \frac{16\sqrt2}{\sqrt{3\sin(\pi\a)}}\,\pa(f).
$$
\end{proof}

%: REFERENCES

\end{document}